\documentclass{amsart}

\usepackage{amsfonts}
\usepackage{amsmath}
\usepackage{amssymb}
\usepackage{amsthm}
\usepackage{mathrsfs}
\usepackage{textcomp}
\usepackage{wasysym}
\usepackage{stmaryrd} 
\usepackage{multirow}
\usepackage{hhline}
\usepackage{bbm}
\usepackage{tikz}
\usetikzlibrary{decorations.pathmorphing,cd,patterns}
\usepackage[colorlinks, citecolor = green!70!black, urlcolor = blue]{hyperref}
\usepackage[capitalise]{cleveref}

\newtheorem{theorem}{Theorem}[section]
\newtheorem*{theorem*}{Theorem}
\newtheorem{proposition}[theorem]{Proposition}
\newtheorem{lemma}[theorem]{Lemma}
\newtheorem{corollary}[theorem]{Corollary}

\theoremstyle{definition}
\newtheorem{definition}[theorem]{Definition}

\newtheorem{claim}{Claim}
\crefname{claim}{Claim}{Claims}
\newtheorem*{claim*}{Claim}
\newtheorem*{temp}{Temporary definition}
\newtheorem{notation}[theorem]{Notation}

\theoremstyle{remark}
\newtheorem*{remark}{Remark}

\crefname{step}{Step}{Steps}
\mathchardef\mhyphen="2D

\newcommand{\A}{\mathscr{A}}
\newcommand{\B}{\mathscr{B}}
\newcommand{\C}{\mathscr{C}}
\newcommand{\D}{\mathscr{D}}
\newcommand{\Jv}{\mathscr{J}_v}
\newcommand{\Jh}{\mathscr{J}_h}
\newcommand{\HH}{\mathscr{H}}
\newcommand{\T}{\mathscr{T}}

\newcommand{\M}{\mathscr{M}}
\newcommand{\N}{\mathscr{N}}

\newcommand{\E}{\mathcal{E}}

\renewcommand{\H}{\mathcal{H}}
\newcommand{\I}{\mathcal{I}}
\newcommand{\J}{\mathcal{J}}

\newcommand{\NN}{{\mathbb{N}}}
\newcommand{\BB}{{\mathbb{B}}}

\newcommand{\celll}{\operatorname{cell}}

\newcommand{\id}{\mathrm{id}}
\newcommand{\op}{^{\mathrm{op}}}
\newcommand{\co}{^{\mathrm{co}}}

\newcommand{\colim}{\operatorname{colim}}

\newcommand{\ob}{\operatorname{ob}}
\newcommand{\cut}{\operatorname{cut}}
\newcommand{\sil}{\operatorname{sil}}

\newcommand{\two}{\mathbbm{2}}
\newcommand{\cell}{\Theta_2}
\newcommand{\spine}{\Xi}
\newcommand{\horn}{\Lambda}

\newcommand{\hattheta}{\widehat{\Theta_2}}
\newcommand{\hatdelta}{\widehat{\Delta}}

\newcommand{\mnd}{\mathscr{M}\hspace{-1pt}nd}

\newcommand{\twoCat}{2\mhyphen\underline{\mathrm{Cat}}}
\newcommand{\Cat}{\underline{\mathrm{Cat}}}

\newcommand{\Set}{\underline{\mathrm{Set}}}

\newcommand{\nq}{{[n;\mathbf{q}]}}

\renewcommand{\mp}{{[m;\mathbf{p}]}}

\newcommand{\ff}{\boldsymbol{f}}
\newcommand{\ffd}{\boldsymbol{f'}}

\newcommand{\pp}{\boldsymbol{p}}
\newcommand{\qq}{\boldsymbol{q}}
\newcommand{\qqd}{\boldsymbol{q'}}

\renewcommand{\ss}{\boldsymbol{s}}
\renewcommand{\tt}{\boldsymbol{t}}
\newcommand{\xx}{\boldsymbol{x}}
\newcommand{\yy}{\boldsymbol{y}}
\newcommand{\zz}{\boldsymbol{z}}

\newcommand{\aalpha}{{\boldsymbol{\alpha}}}
\newcommand{\aalphak}{\boldsymbol{\alpha^k}}

\newcommand{\eepsilon}{{\boldsymbol{\epsilon}}}
\newcommand{\llambda}{\boldsymbol{\lambda}}
\newcommand{\ttheta}{\boldsymbol{\theta}}
\newcommand{\kkappa}{\boldsymbol{\kappa}}
\newcommand{\iid}{\mathbf{id}}
\newcommand{\zzero}{\boldsymbol{0}}

\newcommand{\XX}{{\boldsymbol{X}}}
\newcommand{\YY}{{\boldsymbol{Y}}}
\newcommand{\ZZ}{{\boldsymbol{Z}}}
\newcommand{\WW}{{\boldsymbol{W}}}

\newcommand{\incl}{\hookrightarrow}
\newcommand{\defeq}{\overset{\text{def}}{=}}

\newcommand{\tast}{\textasteriskcentered}
\newcommand{\rest}{\hspace{-3pt}\upharpoonright\hspace{-2pt}}

\newcommand{\tensor}{\scalebox{1.2}{$\otimes$}}
\newcommand{\ttensor}{\raisebox{-1.1pt}{\scalebox{1.2}{$\boxtimes$}}}
\newcommand{\lax}{_{\mathrm{lax}}}
\newcommand{\oplax}{_{\mathrm{oplax}}}

\newcommand{\cellt}{\widetilde{\Theta}_2}
\newcommand{\hornt}{\widetilde{\horn}}

\newcommand{\lelex}{\le_{lex}}

\newcommand{\pre}{\preceq}
\newcommand{\prep}{\pre_p}
\newcommand{\prez}{\pre_0}
\newcommand{\preq}{\pre_q}

\newcommand{\pres}{\pre_s}

\newcommand{\presp}{\pre_{s_\phi}\hspace{-2pt}}
\newcommand{\prepp}{\pre_{p_\phi}\hspace{-2pt}}

\newcommand{\suc}{\succeq}
\newcommand{\sucp}{\suc_p}
\newcommand{\sucz}{\suc_0}

\newcommand{\sucs}{\suc_s}

\newcommand{\ltri}{\hspace{-2.5pt}\leftarrowtriangle\hspace{-2.5pt}}
\newcommand{\rtri}{\hspace{-2.5pt}\rightarrowtriangle\hspace{-2.5pt}}

\title{The Gray tensor product for 2-quasi-categories}
\author{Yuki Maehara}

\address{Centre of Australian Category Theory, Macquarie University, NSW 2109, Australia}
\email{yuki.maehara@mq.edu.au}


\begin{document}

\begin{abstract}
	We construct an $(\infty,2)$-version of the (\emph{lax}) \emph{Gray tensor product}.
	On the 1-categorical level, this is a binary (or more generally an $n$-ary) functor on the category of $\cell$-sets, and it is shown to be left Quillen with respect to Ara's model structure.
	Moreover we prove that this tensor product forms part of a ``homotopical'' (biclosed) monoidal structure, or more precisely a normal lax monoidal structure that is associative up to homotopy.
\end{abstract}

\maketitle

\section{introduction}
Generally speaking, $(\infty,1)$-category theory is developed by imitating ordinary category theory while taking care of relevant homotopical information.
Thus one may reasonably expect to gain a better understanding of certain aspects of $(\infty,1)$-category theory by first developing $(\infty,2)$-category theory and then imitating \emph{formal category theory} therein.
In particular, the content of this paper is intended as a steppingstone towards reconstructing Street's \emph{formal theory of monads} \cite{Street:monads} in the homotopy coherent context.

There is a well-defined notion of \emph{monad} in an arbitrary 2-category $\A$, which of course reduces to the familiar one when $\A = \Cat$.
The totality of all monads in fixed $\A$ is itself organised into a 2-category and plays a crucial role in \cite{Street:monads}.
Although it is not explicitly stated in that paper, this 2-category of monads may be realised as the 2-category $[\mnd,\A]\lax$ of 2-functors $\mnd \to \A$, \emph{lax natural transformations} and \emph{modifications}, where $\mnd$ is the free 2-category containing a monad.
The monoidal structure corresponding to the closed structure $[-,-]\lax$ is the (\emph{lax}) \emph{Gray tensor product} \cite{Gray}.

The particular model for $(\infty,2)$-categories we employ in this paper is \emph{2-quasi-categories}.
These are the fibrant objects in the presheaf category $[\cell\op,\Set]$ with respect to a (Cisinski) model structure due to Ara \cite{Ara:nqcat}, where $\cell$ is Joyal's \emph{2-cell category} which may be regarded as a 2-dimensional analogue to $\Delta$.
It should be noted that, for \emph{complicial sets} (which model $(\infty,\infty)$-categories), a relatively simple definition of the Gray tensor product was given by Verity in \cite{Verity:strict,Verity:I}.
Verity also proved, among other things, the complicial counterpart of our main results.
We however prefer to work with 2-quasi-categories as they admit straightforward duality operations, both ``op'' and ``co'', the latter of which is very difficult to capture in the complicial framework.

In order to develop the formal theory of monads for 2-quasi-categories, we first need a 2-quasi-categorical version of $[\mnd,\A]\lax$.
The purpose of the present paper is to construct more generally a 2-quasi-categorical analogue of the aforementioned biclosed monoidal structure on $\twoCat$.

The formal theory of monads is not the only potential motivation for studying the Gray tensor product of $(\infty,2)$-categories.
For example, in their book on derived algebraic geometry \cite{Gaitsgory;Rozenblyum} Gaitsgory and Rozenblyum listed and exploited various properties such a tensor product should have, but they did not prove its existence.
Our main results correspond to some of the unproven statements in that book, namely Propositions 3.2.6 and 3.2.9.

We are trying to model what is really a biclosed monoidal $(\infty,1)$-category by a biclosed monoidal 1-category, and this gap manifests itself in the following two facets in our setting.
Firstly we need to ``manually'' check that our tensor product respects homotopy, or more precisely that it is left Quillen.
The proof of this fact occupies roughly the first half of this paper.
Secondly the resulting 1-categorical structure is only \emph{lax} monoidal.
The tensor product is unital up to isomorphism and tensoring on either side admits a genuine right adjoint, but it is not associative up to isomorphism.
The second half of the paper is devoted to proving its associativity up to homotopy.

We start by reviewing the necessary background material in \cref{background}.
In \cref{section Gray}, we define the 2-quasi-categorical Gray tensor product (after recalling the ordinary 2-categorical case) and then analyse its basic combinatorics.
\cref{section mono} is devoted to proving that the Leibniz/relative version of the Gray tensor product preserves monomorphisms (= cofibrations in Ara's model structure).
In \cref{section visual}, we introduce and illustrate the notions of \emph{silhouette} and \emph{cut-point} which play important combinatorial roles in later sections.
We prove the \emph{binary} Gray tensor product to be left Quillen with respect to Ara's model structure in \cref{section binary}.
This does not immediately generalise to arbitrary arity since the tensor product is not associative up to isomorphism.
Nevertheless, we prove in \cref{section assoc} that it is associative up to homotopy in a suitable sense.
Consequences of this associativity are discussed in \cref{section consequences}, one of which is indeed that the Gray tensor product of arbitrary arity is left Quillen.

\section{Background}\label{background}
We review the necessary background material in this section.
There is a significant overlap with \cite[\textsection 2]{Maehara:horns}.

\subsection{The category $\cell$}\label{Theta_2}
The category $\Delta$ can be seen as the full subcategory of $\Cat$ spanned by the free categories $[n]$ generated by linear graphs:
\[
\begin{tikzcd}
0 \arrow [r] & 1 \arrow [r] & \dots \arrow [r] & n
\end{tikzcd}
\]
Similarly, Joyal's \emph{2-cell category} $\cell$ is the full subcategory of $\twoCat$ spanned by the free 2-categories $[n;q_1,\dots,q_n]$ generated by ``linear-graph-enriched linear graphs'':
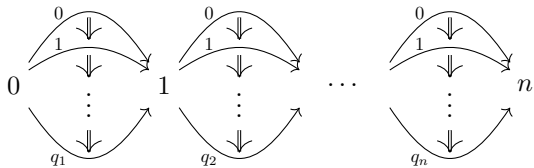
\[
\begin{tikzpicture}[scale = 2]
\node at (0,0) {0};
\node at (1,0) {1};
\node at (2.2,0) {$\dots$};
\node at (3.4,0) {$n$};
\draw[->] (0.1,0.15) .. controls (0.4,0.6) and (0.6,0.6) .. (0.9,0.15);
\draw[->] (0.1,0.1) .. controls (0.4,0.3) and (0.6,0.3) .. (0.9,0.1);
\draw[->] (0.1,-0.15) .. controls (0.4,-0.6) and (0.6,-0.6) .. (0.9,-0.15);
\draw[->, double] (0.5,0.45) -- (0.5,0.3);
\draw[->, double] (0.5,-0.3) -- (0.5,-0.45);
\draw[->, double] (0.5,0.2) -- (0.5,0.05);
\node at (0.5,-0.08) {$\vdots$};
\node[scale = 0.7] at (0.3,0.48) {0};
\node[scale = 0.7] at (0.3,0.27) {1};
\node[scale = 0.7] at (0.3,-0.5) {$q_1$};
\draw[->] (1.1,0.15) .. controls (1.4,0.6) and (1.6,0.6) .. (1.9,0.15);
\draw[->] (1.1,0.1) .. controls (1.4,0.3) and (1.6,0.3) .. (1.9,0.1);
\draw[->] (1.1,-0.15) .. controls (1.4,-0.6) and (1.6,-0.6) .. (1.9,-0.15);
\draw[->, double] (1.5,0.45) -- (1.5,0.3);
\draw[->, double] (1.5,-0.3) -- (1.5,-0.45);
\draw[->, double] (1.5,0.2) -- (1.5,0.05);
\node at (1.5,-0.08) {$\vdots$};
\node[scale = 0.7] at (1.3,0.48) {0};
\node[scale = 0.7] at (1.3,0.27) {1};
\node[scale = 0.7] at (1.3,-0.5) {$q_2$};
\draw[->] (2.5,0.15) .. controls (2.8,0.6) and (3,0.6) .. (3.3,0.15);
\draw[->] (2.5,0.1) .. controls (2.8,0.3) and (3,0.3) .. (3.3,0.1);
\draw[->] (2.5,-0.15) .. controls (2.8,-0.6) and (3,-0.6) .. (3.3,-0.15);
\draw[->, double] (2.9,0.45) -- (2.9,0.3);
\draw[->, double] (2.9,-0.3) -- (2.9,-0.45);
\draw[->, double] (2.9,0.2) -- (2.9,0.05);
\node at (2.9,-0.08) {$\vdots$};
\node[scale = 0.7] at (2.7,0.48) {0};
\node[scale = 0.7] at (2.7,0.27) {1};
\node[scale = 0.7] at (2.7,-0.5) {$q_n$};
\end{tikzpicture}
\]
whose hom-categories are given by
\[
\hom(k,\ell) = \left\{\begin{array}{cl}
[q_{k+1}] \times \cdots \times [q_\ell] & \text {if $k \le \ell$,}\\
\varnothing & \text {if $k > \ell$}.
\end{array}\right.
\]
More precisely, $\cell$ has objects $\nq = [n; q_1, \dots, q_n]$ where $n, q_k \in \NN$ for each $k$.
A morphism $[\alpha; \aalpha] = [\alpha; \alpha_{\alpha(0)+1}, \dots, \alpha_{\alpha(m)}]: \mp \to \nq$ consists of simplicial operators $\alpha : [m] \to [n]$ and $\alpha_k : [p_\ell] \to [q_k]$ for each $k \in [n]$ such that there exists (necessarily unique) $\ell \in [m]$ with $\alpha(\ell-1)<k\le\alpha(\ell)$.
By a \emph{cellular operator} we mean a morphism in $\cell$.
Clearly $[0]$ is a terminal object in $\cell$, and we will write $! : \nq \to [0]$ for any cellular operator into $[0]$.

\begin{remark}
	Here we are describing $\cell = \Delta \wr \Delta$ as an instance of Berger's \emph{wreath product} construction.
	For any given category $\C$, the wreath product $\Delta \wr \C$ may be thought of as the category of free $\C$-enriched categories generated by linear $\C$-enriched graphs.
	The precise definition can be found in \cite[Definition 3.1]{Berger:wreath}.
\end{remark}

The category $\Delta$ has an automorphism $(-)\op$ which is the identity on objects and sends $\alpha : [m] \to [n]$ to $\alpha\op : [m] \to [n]$ given by $\alpha\op(i) = n-\alpha(m-i)$.
This induces two automorphisms on $\cell$, namely:
\begin{itemize}
	\item $(-)\co : \cell \to \cell$, which sends $[\alpha;\aalpha] : \mp \to \nq$ to
	\[
	[\alpha;\alpha\op_{\alpha(0)+1},\dots,\alpha\op_{\alpha(m)}] : \mp \to \nq;
	\]
	and
	\item $(-)\op : \cell \to \cell$, which sends $[\alpha;\aalpha] : \mp \to \nq$ to
	\[
	[\alpha\op;\alpha_{\alpha(m)},\dots,\alpha_{\alpha(0)+1}] : [m;p_m,\dots,p_1] \to [n;q_n,\dots,q_1].
	\]
\end{itemize}

\subsection{Face maps in $\cell$}
There is a Reedy category structure on $\cell$ defined as follows; see \cite[Proposition 2.11]{Bergner;Rezk:Reedy} or \cite[Lemma 2.4]{Berger:nerve} for a proof.
\begin{definition}
	The \emph{dimension} of $\nq$ is $\dim\nq \defeq n + \sum_{k=1}^n q_k$.
	A cellular operator $[\alpha;\aalpha] : \mp \to \nq$ is a \emph{face operator} if $\alpha$ is monic and $\{\alpha_k:\alpha(\ell-1) < k \le \alpha(\ell)\}$ is jointly monic for each $1 \le \ell \le m$.
	It is a \emph{degeneracy operator} if $\alpha$ and all $\alpha_k$ are surjective.
\end{definition}

\begin{definition}\label{special faces}
We say a face map $[\alpha;\aalpha] : \mp \to \nq$ is:
\begin{itemize}
	\item \emph{inner} if $\alpha$ and all $\alpha_k$ preserve the top and bottom elements, and otherwise \emph{outer};
	\item \emph{horizontal} if each $\alpha_k$ is surjective; and
	\item \emph{vertical} if $\alpha = \id$.
\end{itemize}
(Examples of each kind can be found in \cref{faces}.)
A horizontal face map of the form $[\delta^k;\aalpha]$ will be called a \emph{$k$-th horizontal face}.
\end{definition}

By the \emph{codimension} of a face map $[\alpha;\aalpha] : \mp \to \nq$, we mean the difference $\dim\nq - \dim\mp$.
We will in particular be interested in the face maps of codimension 1, which we call \emph{hyperfaces}.
Such a map $[\alpha;\aalpha]$ has precisely one of the following forms:
\begin{itemize}
	\item for $n \ge 1$, $\nq$ always has a unique \emph{0-th horizontal face}
	\[
	\delta_h^0 \defeq [\delta^0;\iid] : [n-1;q_2,\dots,q_n] \to \nq
	\]
	which has codimension 1 if and only if $q_1 = 0$;
	\item similarly, if $q_n = 0$ then the unique \emph{$n$-th horizontal face}
	\[
	\delta_h^n \defeq [\delta^n;\iid] : [n-1;q_1,\dots,q_{n-1}] \to \nq
	\]
	has codimension 1;
	\item for each $1 \le k \le n-1$, there is a family of \emph{$k$-th horizontal hyperfaces}
	\[
	\delta_h^{k;\langle\beta,\beta'\rangle} \defeq [\delta^k;\aalpha] : [n-1;q_1,\dots,q_{k-1},q_k+q_{k+1},q_{k+2},\dots,q_n] \to \nq
	\]
	indexed by $(q_k,q_{k+1})$-shuffles $\langle \beta,\beta' \rangle$ (that is, non-degenerate $(q_k+q_{k+1})$-simplices $\langle\beta,\beta'\rangle : \Delta[q_k+q_{k+1}] \to \Delta[q_k] \times \Delta[q_{k+1}]$) where $\alpha_\ell = \id$ for $k \neq \ell \neq k+1$, $\alpha_k = \beta$ and $\alpha_{k+1} = \beta'$; and
	\item for each $1 \le k \le n$ satisfying $q_k \ge 1$ and for each $0 \le i \le q_k$, the \emph{$(k;i)$-th vertical hyperface}
	\[
	\delta_v^{k;i} \defeq [\id;\aalpha] : [n;q_1,\dots,q_{k-1},q_k-1,q_{k+1},\dots,q_n] \to \nq
	\]
	is given by $\alpha_k = \delta^i$ and $\alpha_\ell = \id$ for $\ell \neq k$.
\end{itemize}

In \cref{faces}, we have listed various faces of $[2;0,2]$.
We will briefly describe how to read the pictures.
In the first row is the ``standard picture'' of $[2;0,2]$, in which we have nicely placed its objects ($\begin{tikzpicture}[baseline = -3]\filldraw (0,0) circle [radius = 1pt]; \end{tikzpicture}$), generating 1-cells ($\begin{tikzpicture}[baseline = -3]\draw[->] (0,0)--(0.5,0); \end{tikzpicture}$) and generating 2-cells ($\begin{tikzpicture}[baseline = -3]\draw[->, double] (0,0)--(0.5,0); \end{tikzpicture}$).
In the rest of the table, a face operator $[\alpha;\aalpha] : \mp \to [2;0,2]$ is illustrated as the standard picture of $\mp$ appropriately distorted so that the $\ell$-th object appears in the $\alpha(\ell)$-th position and each generating 1-cell lies roughly where the factors of its image used to.
In the third row (where $\alpha_1$ is not injective), we have left small gaps between the generating 1-cells so that they do not intersect with each other.
\begin{table}
\begin{tabular}{|c||c|c|c|c|c|}\hline
	& picture & domain & inner/outer & horizontal & vertical \\\hhline{|=#=|=|=|=|=|}
	$\id$ &
	$\begin{tikzpicture}[baseline = -3]
	\filldraw
	(0,0) circle [radius = 1pt]
	(1,0) circle [radius = 1pt]
	(2,0) circle [radius = 1pt];
	\draw[->] (0.1,0)--(0.9,0);
	\draw[->] (1.1,0.15) .. controls (1.4,0.6) and (1.6,0.6) .. (1.9,0.15);
	\draw[->] (1.1,-0.15) .. controls (1.4,-0.6) and (1.6,-0.6) .. (1.9,-0.15);
	\draw[->] (1.1,0)--(1.9,0);
	\draw[->, double] (1.5,-0.05)--(1.5,-0.4);
	\draw[->, double] (1.5,0.4)--(1.5,0.05);
	\end{tikzpicture}$
	& $[2;0,2]$ & inner & \checkmark & \checkmark \\\hline
	
	$\delta_h^0 = [\delta^0;\id]$ &
	$\begin{tikzpicture}[baseline = -3]
	\draw[gray!50!white, fill = gray!50!white]
	(0,0) circle [radius = 1pt];
	\draw[->,gray!50!white] (0.1,0)--(0.9,0);
	\filldraw
	(1,0) circle [radius = 1pt]
	(2,0) circle [radius = 1pt];
	\draw[->] (1.1,0.15) .. controls (1.4,0.6) and (1.6,0.6) .. (1.9,0.15);
	\draw[->] (1.1,-0.15) .. controls (1.4,-0.6) and (1.6,-0.6) .. (1.9,-0.15);
	\draw[->] (1.1,0)--(1.9,0);
	\draw[->, double] (1.5,0.4)--(1.5,0.05);
	\draw[->, double] (1.5,-0.05)--(1.5,-0.4);
	\end{tikzpicture}$
	& $[1;2]$ & outer & \checkmark & $\times$ \\\hline
	
	$\delta_h^{1;\langle!,\id\rangle} = [\delta^1;!,\id]$ &
	$\begin{tikzpicture}[baseline = -3]
	\draw[white] (0,0) circle [radius = 1pt];
	\filldraw
	(0,0) circle [radius = 1pt]
	(2,0) circle [radius = 1pt];
	\draw[->, yshift = 2pt] (0.1,0) -- (1,0) .. controls (1.4,0.6) and (1.6,0.6) .. (1.9,0.15);
	\draw[->, yshift = -2pt] (0.1,0) -- (1,0) .. controls (1.4,-0.6) and (1.6,-0.6) .. (1.9,-0.15);
	\draw[->] (0.1,0)--(1.9,0);
	\draw[->, double] (1.5,0.45)--(1.5,0.05);
	\draw[->, double] (1.5,-0.05)--(1.5,-0.45);
	\end{tikzpicture}$
	& $[1;2]$ & inner & \checkmark & $\times$ \\\hline
	 
	$\delta_v^{2;0} = [\id;\id,\delta^0]$ &
	$\begin{tikzpicture}[baseline = -3]
	\filldraw
	(0,0) circle [radius = 1pt]
	(1,0) circle [radius = 1pt]
	(2,0) circle [radius = 1pt];
	\draw[->] (0.1,0)--(0.9,0);
	\draw[->, gray!50!white] (1.1,0.15) .. controls (1.4,0.6) and (1.6,0.6) .. (1.9,0.15);
	\draw[->] (1.1,-0.15) .. controls (1.4,-0.6) and (1.6,-0.6) .. (1.9,-0.15);
	\draw[->] (1.1,0)--(1.9,0);
	\draw[->, double] (1.5,-0.05)--(1.5,-0.4);
	\draw[->, double, gray!50!white] (1.5,0.4)--(1.5,0.05);
	\end{tikzpicture}$
	& \multirow{2}{*}[-12pt]{$[2;0,1]$} & \multirow{2}{*}[-12pt]{outer} & \multirow{2}{*}[-12pt]{$\times$} & \multirow{2}{*}[-12pt]{\checkmark} \\
	
	$\delta_v^{2;2} = [\id;\id,\delta^2]$ &
	$\begin{tikzpicture}[baseline = -3]
	\filldraw
	(0,0) circle [radius = 1pt]
	(1,0) circle [radius = 1pt]
	(2,0) circle [radius = 1pt];
	\draw[->] (0.1,0)--(0.9,0);
	\draw[->] (1.1,0.15) .. controls (1.4,0.6) and (1.6,0.6) .. (1.9,0.15);
	\draw[->, gray!50!white] (1.1,-0.15) .. controls (1.4,-0.6) and (1.6,-0.6) .. (1.9,-0.15);
	\draw[->] (1.1,0)--(1.9,0);
	\draw[->, double, gray!50!white] (1.5,-0.05)--(1.5,-0.4);
	\draw[->, double] (1.5,0.4)--(1.5,0.05);
	\end{tikzpicture}$
	& & & & \\\hline
	
	$\delta_v^{2;1} = [\id;\id,\delta^1]$ &
	$\begin{tikzpicture}[baseline = -3]
	\filldraw
	(0,0) circle [radius = 1pt]
	(1,0) circle [radius = 1pt]
	(2,0) circle [radius = 1pt];
	\draw[->] (0.1,0)--(0.9,0);
	\draw[->] (1.1,0.15) .. controls (1.4,0.6) and (1.6,0.6) .. (1.9,0.15);
	\draw[->] (1.1,-0.15) .. controls (1.4,-0.6) and (1.6,-0.6) .. (1.9,-0.15);
	\draw[->, double] (1.5,0.4)--(1.5,-0.4);
	\end{tikzpicture}$
	& $[2;0,1]$ & inner & $\times$ & \checkmark \\\hline

	$\delta_h^2 = [\delta^2;\id]$ &
	$\begin{tikzpicture}[baseline = -3]
	\filldraw
	(0,0) circle [radius = 1pt]
	(1,0) circle [radius = 1pt];
	\filldraw[gray!50!white]
	(2,0) circle [radius = 1pt];
	\draw[->] (0.1,0)--(0.9,0);
	\draw[->, gray!50!white] (1.1,0.15) .. controls (1.4,0.6) and (1.6,0.6) .. (1.9,0.15);
	\draw[->, gray!50!white] (1.1,-0.15) .. controls (1.4,-0.6) and (1.6,-0.6) .. (1.9,-0.15);
	\draw[->, gray!50!white] (1.1,0)--(1.9,0);
	\draw[->, double, gray!50!white] (1.5,-0.05)--(1.5,-0.4);
	\draw[->, double, gray!50!white] (1.5,0.4)--(1.5,0.05);
	\end{tikzpicture}$
	& $[1;0]$ & outer & \checkmark & $\times$ \\\hline
	
	$[\{0\}]$ &
	$\begin{tikzpicture}[baseline = -3]
	\filldraw
	(0,0) circle [radius = 1pt];
	\filldraw[gray!50!white]
	(1,0) circle [radius = 1pt]
	(2,0) circle [radius = 1pt];
	\draw[->, gray!50!white] (0.1,0)--(0.9,0);
	\draw[->, gray!50!white] (1.1,0.15) .. controls (1.4,0.6) and (1.6,0.6) .. (1.9,0.15);
	\draw[->, gray!50!white] (1.1,-0.15) .. controls (1.4,-0.6) and (1.6,-0.6) .. (1.9,-0.15);
	\draw[->, gray!50!white] (1.1,0)--(1.9,0);
	\draw[->, double, gray!50!white] (1.5,-0.05)--(1.5,-0.4);
	\draw[->, double, gray!50!white] (1.5,0.4)--(1.5,0.05);
	\end{tikzpicture}$
	& $[0]$ & outer & \checkmark & $\times$ \\\hline
\end{tabular}
\caption{Some faces of $[2;0,2]$}\label{faces}
\end{table}

The hyperfaces of $\nq$ are precisely the maximal faces of $\nq$ in the following sense.
\begin{proposition}[{\cite[Proposition 6.2.4]{Watson}}]\label{Watson}
	Any face map $[\alpha;\aalpha] : \mp \to \nq$ of positive codimension factors through a hyperface of $\nq$.
\end{proposition}

We often denote a simplicial operator (\emph{i.e.}~a morphism in $\Delta$) by its ``image''.
For example, $\{0,2\} = \delta^1 : [1] \to [2]$ is the $1$st elementary face operator.

\begin{definition}\label{etah}
	For any $\nq \in \cell$ and any $1 \le k \le n$, we denote by $\eta_h^k$ the face map
	\[
	\eta_h^k \defeq \bigl[\{k-1,k\};\id\bigr] : [1;q_k] \to \nq.
	\]
\end{definition}
\begin{definition}\label{etav}
	For any $0 \le i \le q$, we denote by $\eta_v^i$ the face map
	\[
	\eta_v^i \defeq \bigl[\id;\{i\}\bigr] : [1;0] \to [1;q].
	\]
\end{definition}

\subsection{Cellular sets}\label{cellular sets}
We will write $\hattheta$ for the category $[\Theta_2\op,\Set]$ of \emph{cellular sets}.
If $X$ is a cellular set, $x \in X_{n;\qq} \defeq X(\nq)$ and $[\alpha;\aalpha] :\mp \to \nq$ is a cellular operator, then we will write $x \cdot [\alpha;\aalpha]$ for the image of $x$ under $X([\alpha;\aalpha])$.
The Reedy structure on $\cell$ is (\emph{EZ} and hence) \emph{elegant}, which means the following.
\begin{theorem}[{\cite[Corollary 4.5]{Bergner;Rezk:Reedy}}]\label{elegance}
	For any cellular set $X$ and for any $x \in X_{m;\pp}$, there is a unique way to express $x$ as $x = y \cdot [\alpha;\aalpha]$ where $[\alpha;\aalpha] : \mp \to \nq$ is a degeneracy operator and $y \in X_{n;\qq}$ is non-degenerate.
\end{theorem}

\begin{definition}
A \emph{cellular subset} of $X \in \hattheta$ is a subfunctor of $X$.
If $S$ is a set of cells in $X \in \hattheta$ (not necessarily closed under the action of cellular operators), the smallest cellular subset $A$ of $X$ containing $S$ is given by
\[
A_{m;\pp} = \bigl\{s \cdot [\alpha;\aalpha] : s \in S_{n;\qq}, \begin{tikzcd} {\mp} \arrow [r, "{[\alpha;\aalpha]}"] & \nq \end{tikzcd}\bigr\}.
\]
We call such $A$ the cellular subset of $X$ \emph{generated by $S$}.
\end{definition}

Sending $[\alpha;\aalpha] : \mp \to \nq$ to $\alpha : [m] \to [n]$ yields a functor $\cell \to \Delta$.
We will regard $\hatdelta$ as a full subcategory of $\hattheta$ via the embedding $\hatdelta \to \hattheta$ induced by this functor.
Note that this identification makes the square
\[
\begin{tikzcd}
\Cat
\arrow [r]
\arrow [d, "N", swap] &
\twoCat
\arrow [d, "N"] \\
\hatdelta
\arrow [r, hook, "\subset"] &
\hattheta
\end{tikzcd}
\]
commutative up to isomorphism, where the upper horizontal map sends each category to the obvious locally discrete 2-category, and the vertical maps are the nerve functors induced by the inclusions $\Delta \incl \Cat$ and $\cell \incl \twoCat$.

There is another way to turn simplicial sets into cellular sets.
For any $X \in \hatdelta$, its \emph{suspension} $\cell[1;X]$ is the nerve of the following simplicially enriched category:
\[
\begin{tikzcd}
0
\arrow [loop left, "{\Delta[0]}"]
\arrow [r, bend left, "X"] &
1
\arrow [l, bend left, "\varnothing"]
\arrow [loop right, "{\Delta[0]}"]
\end{tikzcd}
\]
This construction can be made into a functor $\hatdelta \to \hattheta$ in the obvious manner.
\begin{definition}
	We denote the image of a map $f:X \to Y$ in $\hatdelta$ under the suspension functor by
	\[
	[\id;f] : \cell[1;X] \to \cell[1;Y].
	\]
\end{definition}
Our notation is motivated by the fact that the suspension functor extends the functor $\Delta \to \hattheta$ given by sending $\alpha :[m] \to [n]$ to $[\id;\alpha]:\cell[1;m] \to \cell[1;n]$.
In fact, the suspension functor is the left Kan extension of this functor if we regard them both as mapping into the slice of $\hattheta$ under the boundary $\partial\cell[1;0]$ defined below.

\subsection{Oury's anodyne extensions}\label{O-anodyne}

\begin{definition}\label{boundary is boundary}
The \emph{boundary} $\partial\cell\nq \subset \cell\nq$ is the cellular subset consisting precisely of those maps into $\nq$ that factor through objects of lower dimension.
\end{definition}
\begin{proposition}
The cellular subset $\partial\cell\nq \subset \cell\nq$ is generated by the hyperfaces of $\cell\nq$.
\end{proposition}
\begin{proof}
    This follows from \cref{Watson,boundary is boundary}.
\end{proof}

For example, the boundary $\partial\cell[2;0,2]$ of $\cell[2;0,2]$ is generated by $\delta_h^0$, $\delta_h^{1;\langle!,\id\rangle}$, $\delta_v^{2;0}$, $\delta_v^{2;1}$ and $\delta_v^{2;2}$ (see \cref{faces}).

\begin{definition}\label{I}
We write $\I$ for the set of boundary inclusions, \emph{i.e.}
\[
\I \defeq \bigl\{\partial\cell\nq \incl \cell\nq : \nq \in \cell\bigr\}.
\]
\end{definition}
The following proposition follows from \cref{elegance}.

\begin{definition}
	For any set $\mathcal{S}$ of morphisms in a category with pushouts and transfinite composites, let $\celll(\mathcal{S})$ denote the closure of $\mathcal{S}$ under transfinite composition and taking pushouts along arbitrary maps.
\end{definition}

\begin{proposition}\label{mono}
The class $\celll(\I)$ consists precisely of the monomorphisms in $\hattheta$.
\end{proposition}

\begin{definition}\label{horizontal horn description}
The \emph{$k$-th horizontal horn} $\horn_h^k\nq \subset \cell\nq$ is the cellular subset generated by all hyperfaces except for the $k$-th horizontal ones.
\end{definition}
For example, the horizontal horn $\horn_h^1[2;0,2]$ is generated by $\delta_h^0$, $\delta_v^{2;0}$, $\delta_v^{2;1}$ and $\delta_v^{2;2}$.

\begin{remark}
	The faces $[\alpha;\aalpha] : \cell\mp \to \cell\nq$ not contained in the horizontal horn $\horn_h^k\nq$ are precisely the $k$-th horizontal ones.
	In particular, $\horn_h^k\nq$ may be missing faces of $\cell\nq$ that have codimension greater than $1$.
	For example, one can check that $\horn_h^1[2;1,1]$ is generated by the vertical hyperfaces
	\[
	\begin{gathered}
	\delta_v^{1;0} = \left\{
	\begin{tikzpicture}[baseline = -3]
	\filldraw
	(0,0) circle [radius = 1pt]
	(1,0) circle [radius = 1pt]
	(2,0) circle [radius = 1pt];
	\draw[->, gray!50!white] (0.1,0.1) .. controls (0.4,0.4) and (0.6,0.4) .. (0.9,0.1);
	\draw[->] (0.1,-0.1) .. controls (0.4,-0.4) and (0.6,-0.4) .. (0.9,-0.1);
	\draw[->] (1.1,0.1) .. controls (1.4,0.4) and (1.6,0.4) .. (1.9,0.1);
	\draw[->] (1.1,-0.1) .. controls (1.4,-0.4) and (1.6,-0.4) .. (1.9,-0.1);
	\draw[->, double, gray!50!white] (0.5,0.25)--(0.5,-0.25);
	\draw[->, double] (1.5,0.25)--(1.5,-0.25);
	\end{tikzpicture}\right\}, \hspace{10pt}
	\delta_v^{1;1} = \left\{
	\begin{tikzpicture}[baseline = -3]
	\filldraw
	(0,0) circle [radius = 1pt]
	(1,0) circle [radius = 1pt]
	(2,0) circle [radius = 1pt];
	\draw[->] (0.1,0.1) .. controls (0.4,0.4) and (0.6,0.4) .. (0.9,0.1);
	\draw[->, gray!50!white] (0.1,-0.1) .. controls (0.4,-0.4) and (0.6,-0.4) .. (0.9,-0.1);
	\draw[->] (1.1,0.1) .. controls (1.4,0.4) and (1.6,0.4) .. (1.9,0.1);
	\draw[->] (1.1,-0.1) .. controls (1.4,-0.4) and (1.6,-0.4) .. (1.9,-0.1);
	\draw[->, double, gray!50!white] (0.5,0.25)--(0.5,-0.25);
	\draw[->, double] (1.5,0.25)--(1.5,-0.25);
	\end{tikzpicture}\right\},\\
	\delta_v^{2;0} = \left\{
	\begin{tikzpicture}[baseline = -3]
	\filldraw
	(0,0) circle [radius = 1pt]
	(1,0) circle [radius = 1pt]
	(2,0) circle [radius = 1pt];
	\draw[->] (0.1,0.1) .. controls (0.4,0.4) and (0.6,0.4) .. (0.9,0.1);
	\draw[->] (0.1,-0.1) .. controls (0.4,-0.4) and (0.6,-0.4) .. (0.9,-0.1);
	\draw[->, gray!50!white] (1.1,0.1) .. controls (1.4,0.4) and (1.6,0.4) .. (1.9,0.1);
	\draw[->] (1.1,-0.1) .. controls (1.4,-0.4) and (1.6,-0.4) .. (1.9,-0.1);
	\draw[->, double] (0.5,0.25)--(0.5,-0.25);
	\draw[->, double, gray!50!white] (1.5,0.25)--(1.5,-0.25);
	\end{tikzpicture}\right\} \hspace{10pt} \text {and} \hspace{10pt}
	\delta_v^{2;1} = \left\{
	\begin{tikzpicture}[baseline = -3]
	\filldraw
	(0,0) circle [radius = 1pt]
	(1,0) circle [radius = 1pt]
	(2,0) circle [radius = 1pt];
	\draw[->] (0.1,0.1) .. controls (0.4,0.4) and (0.6,0.4) .. (0.9,0.1);
	\draw[->] (0.1,-0.1) .. controls (0.4,-0.4) and (0.6,-0.4) .. (0.9,-0.1);
	\draw[->] (1.1,0.1) .. controls (1.4,0.4) and (1.6,0.4) .. (1.9,0.1);
	\draw[->, gray!50!white] (1.1,-0.1) .. controls (1.4,-0.4) and (1.6,-0.4) .. (1.9,-0.1);
	\draw[->, double] (0.5,0.25)--(0.5,-0.25);
	\draw[->, double, gray!50!white] (1.5,0.25)--(1.5,-0.25);
	\end{tikzpicture}\right\}
	\end{gathered}
	\]
	and so it does not contain the face
	\[
	[\delta^1;\id,\id] = \left\{\begin{tikzpicture}[baseline = -3]
	\filldraw
	(0,0) circle [radius = 1pt]
	(2,0) circle [radius = 1pt];
	\draw[->, yshift = 1pt] (0.1,0.1) .. controls (0.4,0.4) and (0.6,0.4) .. (1,0) .. controls (1.4,0.4) and (1.6,0.4) .. (1.9,0.1);
	\draw[->, yshift = -1pt] (0.1,-0.1) .. controls (0.4,-0.4) and (0.6,-0.4) .. (1,0) .. controls (1.4,-0.4) and (1.6,-0.4) .. (1.9,-0.1);
	\draw[->, double] (0.5,0.25)--(0.5,-0.25);
	\end{tikzpicture}\right\}
	\]
	of codimension $2$.
	(The last face may equally well be depicted as $\left\{\begin{tikzpicture}[baseline = -3]
	\filldraw
	(0,0) circle [radius = 1pt]
	(2,0) circle [radius = 1pt];
	\draw[->, yshift = 1pt] (0.1,0.1) .. controls (0.4,0.4) and (0.6,0.4) .. (1,0) .. controls (1.4,0.4) and (1.6,0.4) .. (1.9,0.1);
	\draw[->, yshift = -1pt] (0.1,-0.1) .. controls (0.4,-0.4) and (0.6,-0.4) .. (1,0) .. controls (1.4,-0.4) and (1.6,-0.4) .. (1.9,-0.1);
	\draw[->, double] (1.5,0.25)--(1.5,-0.25);
	\end{tikzpicture}\right\}$; the position of the double arrow has no significance.)
	This differs from the more commonly found definition of a horn (\emph{e.g.}~\cite{Berger:nerve,Watson}) as ``boundary with one hyperface removed''.
	The relationship between such alternative horns and Oury's horns is investigated in \cite[\textsection 4]{Maehara:horns}.
\end{remark}

\begin{definition}\label{vertical horn description}
The \emph{$(k;i)$-th vertical horn} $\horn_v^{k;i}\nq \subset \cell\nq$, where $0 \le k \le n$ satisfies $q_k \ge 1$ and $0 \le i \le q_k$, is the cellular subset generated by all hyperfaces except for the $(k;i)$-th vertical ones.
It is called \emph{inner} if $1 \le i \le q_k-1$.
\end{definition}
For example, the vertical horn $\horn_v^{2;1}[2;0,2]$ is generated by $\delta_h^0$, $\delta_h^{1;\langle!,\id\rangle}$, $\delta_v^{2;0}$ and $\delta_v^{2;2}$.

\begin{definition}
	We write $e : \cell[0] \to J$ for the \emph{horizontal equivalence extension}, which is the nerve of the inclusion $\{\lozenge\}\incl\{\lozenge\cong\blacklozenge\}$ into the chaotic category on two objects.
	Its suspension $[\id;e] : \cell[1;0] \to \cell[1;J]$, called the \emph{vertical equivalence extension}, is (isomorphic to) the nerve of the 2-functor
	\[
	\left\{
	\begin{tikzpicture}[baseline = -3]
	\filldraw
	(0,0) circle [radius = 1pt]
	(1,0) circle [radius = 1pt];
	\draw[->] (0.1,0.1) .. controls (0.4,0.4) and (0.6,0.4) .. (0.9,0.1);
	\end{tikzpicture}\right\}
	\incl
	\left\{
	\begin{tikzpicture}[baseline = -3]
	\filldraw
	(0,0) circle [radius = 1pt]
	(1,0) circle [radius = 1pt];
	\draw[->] (0.1,0.1) .. controls (0.4,0.4) and (0.6,0.4) .. (0.9,0.1);
	\draw[->] (0.1,-0.1) .. controls (0.4,-0.4) and (0.6,-0.4) .. (0.9,-0.1);
	\node[rotate = -90] at (0.5,0) {$\cong$};
	\end{tikzpicture}\right\}
	\]
	whose codomain is locally chaotic.
\end{definition}

\begin{definition}\label{J}
	Let $\H_h$ and $\H_v$ denote the sets of inner horizontal horn inclusions and inner vertical horn inclusions respectively.
	We write $\J$ for the union
	\[
	\J \defeq \H_h \cup \H_v \cup \bigl\{[\id;e],e\bigr\}.
	\]
\end{definition}

\subsection{Leibniz construction}\label{Leibniz}
Suppose that we are given $F : \C_1 \times \dots \times \C_n \to \D$ into a category $\D$ with finite connected colimits.
Then the (\emph{$n$-ary}) \emph{Leibniz construction}
\[
\hat F : \C_1^\two \times \dots \times \C_n^\two \to \D^\two
\]
of $F$, where $\two = \{0 \to 1\}$ is the generic arrow category, is defined as follows.
Let $f_i : X^0_i \to X^1_i$ be an object in $\C_i^\two$ for each $i$.
Then the assignment
\[
(\epsilon_1, \dots, \epsilon_n) \mapsto F(X_1^{\epsilon_1}, \dots, X_n^{\epsilon_n})
\]
defines a functor $G : \two^n \to \D$.
Denote by $I$ the inclusion of the full subcategory of $\two^n$ spanned by all non-terminal objects.
Then $G$ defines a cone under the diagram $GI$, so we obtain an induced morphism $\colim GI \to F(X^1_1, \dots, X^1_n)$.
Sending $(f_1, \dots, f_n)$ to this morphism defines the object part of $\hat F$, and the morphism part is defined in the obvious way by the universal property.

\begin{lemma}\label{Leibniz mono}
	Let $F : \C_1 \times \dots \times \C_n \to \D$ be a functor into a presheaf category $\D$.
	Let $f_i : X_i^0 \to X_i^1$ in each $\C_i$ and suppose $G$ (as above) sends each square of the form
	\begin{equation}\label{ij square}
	\begin{tikzpicture}[baseline = 25]
	\node at (0,2) {$(1,\dots,1,0,1,\dots,1,0,1,\dots,1)$};
	\node at (5,2) {$(1,\dots,1,0,1,\dots,1)$};
	\node at (0,0) {$(1,\dots,1,0,1,\dots,1)$};
	\node at (5,0) {$(1,\dots,1)$};
	
	\draw[->] (2,0) -- (4,0);
	\draw[->] (0,1.5) -- (0,0.5);
	\draw[->] (5,1.5) -- (5,0.5);
	\draw[->] (2.7,2) -- (3.2,2);
	
	\node[scale = 0.7] at (0,-0.7) {$i$-th};
	\node[scale = 0.7] at (5,2.7) {$j$-th};
	\node[scale = 0.7] at (-0.8,2.7) {$i$-th};
	\node[scale = 0.7] at (0.8,2.7) {$j$-th};
	
	\node[scale = 0.7] at (0,-0.4) {$\uparrow$};
	\node[scale = 0.7] at (5,2.4) {$\downarrow$};
	\node[scale = 0.7] at (-0.8,2.4) {$\downarrow$};
	\node[scale = 0.7] at (0.8,2.4) {$\downarrow$};
	\end{tikzpicture}
	\end{equation}
	to a pullback square of monomorphisms.
	Then $\hat F(f_1,\dots,f_n)$ is a monomorphism.
\end{lemma}
\begin{proof}
	This is straightforward to check when $\D = \Set$, and the general result follows from this special instance since limits and colimits in presheaf categories are computed pointwise.
\end{proof}

\begin{lemma}\label{celll}
	Suppose that a functor $F : \C_1 \times \dots \times \C_n \to \D$ preserves pushouts and transfinite compositions in each variable.
	Let $\mathcal{S}_1, \dots, \mathcal{S}_n$ be collections of morphisms in $\C_1, \dots, \C_n$ respectively.
	Then
	\[
	\hat F \bigl(\celll(\mathcal{S}_1),\dots,\celll(\mathcal{S}_n)\bigr) \subset \celll\bigl(\hat F(\mathcal{S}_1,\dots,\mathcal{S}_n)\bigr).
	\]
\end{lemma}
\begin{proof}
    A proof can be found in \cite[Corollary 1.4.14]{Gindi:rigidification}, which is essentially \cite[Corollary 3.11]{Oury} with errors corrected.
    The case $n=2$ is also proved in \cite[Proposition 5.12]{RV:Reedy}.
\end{proof}

\subsection{Ara's model structure for 2-quasi-categories} \label{model structure}
In \cite{Ara:nqcat}, Ara defines a model structure on $\widehat{\Theta_n}$ whose fibrant objects (called \emph{$n$-quasi-categories}) model $(\infty,n)$-categories.
Here we review Ara's characterisation of this model structure, but specialise to the case $n=2$.

First we recall the notion of \emph{spine}.
\begin{definition}
	The only \emph{vertebra} of $\cell[0]$ is the identity map $\id : \cell[0] \to \cell[0]$.
	For $\nq \in \cell$ with $n \ge 1$:
	\begin{itemize}
		\item if $1 \le k \le n$ and $q_k = 0$, then 
		\[
		[\{k-1,k\};\id] : \cell[1;0] \to \cell\nq
		\]
		is a \emph{vertebra}; and
		\item if $1 \le k \le n$ and $q_k \ge 1$, then for each $1 \le i \le q_k$,
		\[
		[\{k-1,k\};\{i-1,i\}] : \cell[1;1] \to \cell\nq
		\]
		is a \emph{vertebra}.
	\end{itemize}
	Let $\spine\nq \subset \cell\nq$ denote the cellular subset generated by the vertebrae of $\cell\nq$, and call it the \emph{spine} of $\cell\nq$.
\end{definition}
For example, the spine $\spine[2;0,2]$ is generated by the following three vertebrae:
\[
\left\{\begin{tikzpicture}[baseline = -3]
\filldraw
(0,0) circle [radius = 1pt]
(1,0) circle [radius = 1pt];
\draw[gray!50!white, fill = gray!50!white]
(2,0) circle [radius = 1pt];
\draw[->] (0.1,0)--(0.9,0);
\draw[->, gray!50!white] (1.1,0.15) .. controls (1.4,0.6) and (1.6,0.6) .. (1.9,0.15);
\draw[->, gray!50!white] (1.1,-0.15) .. controls (1.4,-0.6) and (1.6,-0.6) .. (1.9,-0.15);
\draw[->, gray!50!white] (1.1,0)--(1.9,0);
\draw[->, double, gray!50!white] (1.5,-0.05)--(1.5,-0.4);
\draw[->, double, gray!50!white] (1.5,0.4)--(1.5,0.05);
\end{tikzpicture}\right\}, \hspace{5pt}
\left\{\begin{tikzpicture}[baseline = -3]
\filldraw
(2,0) circle [radius = 1pt]
(1,0) circle [radius = 1pt];
\draw[gray!50!white, fill = gray!50!white]
(0,0) circle [radius = 1pt];
\draw[->, gray!50!white] (0.1,0)--(0.9,0);
\draw[->] (1.1,0.15) .. controls (1.4,0.6) and (1.6,0.6) .. (1.9,0.15);
\draw[->, gray!50!white] (1.1,-0.15) .. controls (1.4,-0.6) and (1.6,-0.6) .. (1.9,-0.15);
\draw[->] (1.1,0)--(1.9,0);
\draw[->, double, gray!50!white] (1.5,-0.05)--(1.5,-0.4);
\draw[->, double] (1.5,0.4)--(1.5,0.05);
\end{tikzpicture}\right\}, \hspace{5pt} \text {and} \hspace{5pt}
\left\{\begin{tikzpicture}[baseline = -3]
\filldraw
(2,0) circle [radius = 1pt]
(1,0) circle [radius = 1pt];
\draw[gray!50!white, fill = gray!50!white]
(0,0) circle [radius = 1pt];
\draw[->, gray!50!white] (0.1,0)--(0.9,0);
\draw[->, gray!50!white] (1.1,0.15) .. controls (1.4,0.6) and (1.6,0.6) .. (1.9,0.15);
\draw[->] (1.1,-0.15) .. controls (1.4,-0.6) and (1.6,-0.6) .. (1.9,-0.15);
\draw[->] (1.1,0)--(1.9,0);
\draw[->, double] (1.5,-0.05)--(1.5,-0.4);
\draw[->, double, gray!50!white] (1.5,0.4)--(1.5,0.05);
\end{tikzpicture}\right\}.
\]

\begin{definition}\label{JA}
Let $\J_A$ denote the union of
\[
\bigl\{\bigl(\cell[0] \overset{e}{\incl}J\bigr) \hat \times \bigl(\partial \cell\nq \incl \cell\nq\bigr) : \nq \in \cell \bigr\}
\]
and the closure of
\[
\bigl\{\spine\nq \incl \cell\nq:\nq \in \cell\bigr\} \cup \bigl\{[\id;e]\bigr\}
\]
under taking Leibniz products
\[
(-)\hat \times \bigl(\cell[0] \amalg \cell[0] \incl J\bigr)
\]
with the nerve of $\{\lozenge\} \amalg \{\blacklozenge\} \incl \{\lozenge \cong \blacklozenge\}$.
\end{definition}
\begin{theorem}[{\cite[\textsection 2.10 and \textsection 5.17]{Ara:nqcat}}]\label{Ara characterisation}
	There is a model structure on $\hattheta$ characterised by the following properties:
	\begin{itemize}
		\item the cofibrations are precisely the monomorphisms; and
		\item a map $f : X \to Y$ into a fibrant cellular set $Y$ is a fibration if and only if it has the right lifting property with respect to all maps in $\J_A$.
	\end{itemize}
\end{theorem} 

In particular, the fibrant objects, called \emph{2-quasi-categories}, are precisely those objects with the right lifting property with respect to all maps in $\J_A$.

This is the only model structure on $\hattheta$ with which we are concerned in this paper, hence no confusion should arise in the following when we simply refer to ``(trivial) (co)fibrations'' in $\hattheta$ without further qualification.

In \cite[Theorem 6.1]{Maehara:horns}, we characterised Ara's model structure using Oury's anodyne extensions (of which $\J$ in \cref{J} is a subset).
The following theorem (whose proof is deferred to \cref{left Quillen proof}) is a consequence of that result.

\begin{theorem}\label{left Quillen characterisation}
	Let
	\[
	F: \hattheta \times \dots \times \hattheta \to \mathscr{M}
	\]
	be an $n$-ary functor into a model category $\mathscr{M}$.
	Suppose that for any $1 \le k \le n$ and for any choice of objects $X_i \in \hattheta$ for $i \neq k$, the functor
	\[
	F(X_1,\dots,X_{k-1},-,X_{k+1},\dots,X_n) : \hattheta \to \M
	\]
	admits a right adjoint.
	Then $F$ is left Quillen if and only if:
	\begin{itemize}
		\item[(i)] each map in $\hat F(\I,\dots,\I)$ is a cofibration; and
		\item[(ii)] each map in $\hat F(\I,\dots,\I,\J,\I,\dots,\I)$ is a trivial cofibration for any position of $\J$.
	\end{itemize}
In particular, each map in $\J$ is a trivial cofibration.
\end{theorem}

\subsection{Gluing}\label{gluing}
Suppose we have a pullback square
\[
\begin{tikzcd}
W \arrow [r, hook, "\subset"] \arrow [d] \arrow [dr, phantom, "\lrcorner", very near start] & X \arrow [d, "f"] \\
Y \arrow [r, hook, "\subset"] & Z
\end{tikzcd}
\]
in $\hattheta$ such that $Z = f(X) \cup Y$, and $f$ is injective on $f^{-1}(Z \setminus Y) = X \setminus W$.
Then the square is also a pushout, and we say that $Z$ is obtained from $Y$ by \emph{gluing $X$ along $W$}.
Note that if $Y$ is generated by a set $S$ of cells in $Z$, then $W$ is generated by the pullbacks of $\begin{tikzcd} \cell\nq \arrow [r, "s"] & Z \end{tikzcd}$ along $f$ for all $s \in S$.

\section{The Gray tensor product}\label{section Gray}
We will briefly review the ordinary Gray tensor product and then define the 2-quasi-categorical version.
Basic combinatorics of the latter is then analysed, and in particular we prove (in \cref{braid}) that the Gray tensor product of representable $\cell$-sets is always the nerve of a poset-enriched category.

\subsection{Ordinary Gray tensor product}
The (\emph{lax}) \emph{Gray tensor product} \cite[Theorem I.4.9]{Gray} of two (small) 2-categories $\A$ and $\B$ is the 2-category $\A \boxtimes \B$ given by the following generators-and-relations presentation.
Its object set is $\ob (\A \boxtimes \B) = \ob \A \times \ob \B$.
Its underlying 1-category is generated by the maps of the form
\begin{equation}\label{generating 1-cells}
\begin{tikzpicture}[baseline = -3]
\node at (0,0) {$(x,y)$};
\node at (2.5,0) {$(x',y)$};
\draw[->] (0.5,0) -- (2,0);
\node[scale = 0.8] at (1.25,0.3) {$(f,y)$};
\end{tikzpicture}
, \hspace{20pt}
\begin{tikzpicture}[baseline = 25]
\node at (0,0) {$(x,y')$};
\node at (0,2) {$(x,y)$};
\draw[->] (0,1.7) -- (0,0.3);
\node[scale = 0.8] at (0.5,1) {$(x,g)$};
\end{tikzpicture}
\end{equation}
where $f : x \to x'$ in $\A$ and $g : y \to y'$ in $\B$, subject to the relations $(f', y)(f, y) = (f'f,y)$ and $(x, g')(x, g) = (x, g'g)$ whenever these composites make sense, and $\id_{(x,y)} = (\id_x, y) = (x, \id_y)$.
Similarly, we have generating 2-cells
\begin{equation}\label{generating 2-cells}
\begin{tikzpicture}[baseline = -3]
\node at (0,0) {$(x,y)$};
\node at (3,0) {$(x',y)$};
\draw[->] (0.5,0.2) .. controls  (1,0.7) and (2,0.7) .. (2.5,0.2);
\draw[->] (0.5,-0.2) .. controls  (1,-0.7) and (2,-0.7) .. (2.5,-0.2);
\node[scale = 0.8] at (1.5,0.8) {$(f',y)$};
\node[scale = 0.8] at (1.5,-0.8) {$(f,y)$};
\draw[->, double] (1.5,-0.5) -- (1.5,0.5);
\node[scale = 0.8, fill = white] at (1.5,0) {$(\alpha,y)$};
\end{tikzpicture}
, \hspace{20pt}
\begin{tikzpicture}[baseline = 25]
\node at (0,0) {$(x,y')$};
\node at (0,2) {$(x,y)$};
\draw[->] (0.2,1.7) .. controls (0.6,1.3) and (0.6,0.7) ..  (0.2,0.3);
\draw[->] (-0.2,1.7) .. controls (-0.6,1.3) and (-0.6,0.7) ..  (-0.2,0.3);
\node[scale = 0.8] at (-0.9,1) {$(x,g)$};
\node[scale = 0.8] at (0.9,1) {$(x,g')$};
\draw[->, double] (-0.3,1) -- (0.3,1);
\node[scale = 0.8] at (0,1.3) {$(x,\beta)$};
\end{tikzpicture}
\end{equation}
for any 2-cells $\alpha : f \Rightarrow f' : x \to x'$ in $\A$ and $\beta : g \Rightarrow g' : y \to y'$ in $\B$, subject to the obvious relations involving the horizontal and vertical compositions in $\A$ and $\B$.
There are additional generating 2-cells of the form
\begin{equation}\label{gamma}
\begin{tikzpicture}[baseline = 25]
\node at (0,0) {$(x,y')$};
\node at (0,2) {$(x,y)$};
\node at (2.5,0) {$(x',y')$};
\node at (2.5,2) {$(x',y)$};
\draw[->] (0,1.7) -- (0,0.3);
\node[scale = 0.8] at (-0.5,1) {$(x,g)$};
\draw[->] (2.5,1.7) -- (2.5,0.3);
\node[scale = 0.8] at (3,1) {$(x',g)$};
\draw[->] (0.5,0) -- (2,0);
\node[scale = 0.8] at (1.25,-0.3) {$(f,y')$};
\draw[->] (0.5,2) -- (2,2);
\node[scale = 0.8] at (1.25,2.3) {$(f,y)$};
\draw[->, double] (0.5,0.4) -- (2,1.6);
\node[scale = 0.8, fill = white] at (1.25,1) {$\gamma_{f,g}$};
\end{tikzpicture}
\end{equation}
for 1-cells $f$ in $\A$ and $g$ in $\B$.
The relations we impose on these 2-cells are:
\begin{equation}\label{coherence:natural}
\begin{tikzpicture}[baseline = 32]
\node at (0,0) {$(x,y')$};
\node at (0,2.5) {$(x,y)$};
\node at (3,0) {$(x',y')$};
\node at (3,2.5) {$(x',y)$};
\draw[->] (0.5,2.7) .. controls  (1,3.2) and (2,3.2) .. (2.5,2.7);
\draw[->] (0.5,2.3) .. controls  (1,1.8) and (2,1.8) .. (2.5,2.3);
\draw[->] (0,2.2) -- (0,0.3);
\draw[->] (0.5,-0.2) .. controls  (1,-0.7) and (2,-0.7) .. (2.5,-0.2);
\node[scale = 0.8] at (-0.5,1.25) {$(x,g)$};
\draw[->] (3,2.2) -- (3,0.3);
\node[scale = 0.8] at (3.5,1.25) {$(x',g)$};
\node[scale = 0.8] at (1.5,3.3) {$(f',y)$};
\node[scale = 0.8] at (1,1.7) {$(f,y)$};
\node[scale = 0.8] at (1.5,-0.8) {$(f,y')$};
\draw[->, double] (1.5,2) -- (1.5,3);
\node[scale = 0.8, fill = white] at (1.5,2.5) {$(\alpha,y)$};
\draw[->, double] (0.6,0) -- (2.4,1.8);
\node[scale = 0.8, fill = white] at (1.5,0.9) {$\gamma_{f,g}$};
\end{tikzpicture}
\hspace{10pt} = \hspace{10pt}
\begin{tikzpicture}[baseline = 32]
\node at (0,0) {$(x,y')$};
\node at (0,2.5) {$(x,y)$};
\node at (3,0) {$(x',y')$};
\node at (3,2.5) {$(x',y)$};
\draw[->] (0.5,0.2) .. controls  (1,0.7) and (2,0.7) .. (2.5,0.2);
\draw[->] (0.5,-0.2) .. controls  (1,-0.7) and (2,-0.7) .. (2.5,-0.2);
\draw[->] (0,2.2) -- (0,0.3);
\draw[->] (0.5,2.7) .. controls  (1,3.2) and (2,3.2) .. (2.5,2.7);
\node[scale = 0.8] at (-0.5,1.25) {$(x,g)$};
\draw[->] (3,2.2) -- (3,0.3);
\node[scale = 0.8] at (3.5,1.25) {$(x',g)$};
\draw[->, double] (1.5,-0.5) -- (1.5,0.5);
\node[scale = 0.8, fill = white] at (1.5,0) {$(\alpha,y')$};
\draw[->, double] (0.6,0.7) -- (2.4,2.5);
\node[scale = 0.8, fill = white] at (1.5,1.6) {$\gamma_{f',g}$};
\node[scale = 0.8] at (1.5,-0.8) {$(f,y')$};
\node[scale = 0.8] at (2,0.8) {$(f',y')$};
\node[scale = 0.8] at (1.5,3.3) {$(f',y)$};
\end{tikzpicture}
\end{equation}

\begin{equation}\label{coherence:identity}
\begin{tikzpicture}[baseline = 25]
\node at (0,0) {$(x,y')$};
\node at (0,2) {$(x,y)$};
\node at (2.5,0) {$(x,y')$};
\node at (2.5,2) {$(x,y)$};
\draw[->] (0,1.7) -- (0,0.3);
\node[scale = 0.8] at (-0.5,1) {$(x,g)$};
\draw[->] (2.5,1.7) -- (2.5,0.3);
\node[scale = 0.8] at (3,1) {$(x,g)$};
\draw[->] (0.5,0) -- (2,0);
\node[scale = 0.8] at (1.25,-0.3) {$(\id_x,y')$};
\draw[->] (0.5,2) -- (2,2);
\node[scale = 0.8] at (1.25,2.3) {$(\id_x,y)$};
\draw[->, double] (0.5,0.4) -- (2,1.6);
\node[scale = 0.8, fill = white] at (1.25,1) {$\gamma_{\id_x,g}$};
\end{tikzpicture}
\hspace{10pt} = \hspace{10pt}
\begin{tikzpicture}[baseline = 25]
\node at (0,0) {$(x,y')$};
\node at (0,2) {$(x,y)$};
\node at (2.5,0) {$(x,y')$};
\node at (2.5,2) {$(x,y)$};
\draw[->] (0,1.7) -- (0,0.3);
\node[scale = 0.8] at (-0.5,1) {$(x,g)$};
\draw[->] (2.5,1.7) -- (2.5,0.3);
\node[scale = 0.8] at (3,1) {$(x,g)$};
\draw[->] (0.5,0) -- (2,0);
\node[scale = 0.8] at (1.25,-0.3) {$\id_{(x,y')}$};
\draw[->] (0.5,2) -- (2,2);
\node[scale = 0.8] at (1.25,2.3) {$\id_{(x,y)}$};
\draw[->, double] (0.5,0.4) -- (2,1.6);
\node[scale = 0.8, fill = white] at (1.25,1) {$\id$};
\end{tikzpicture}
\end{equation}

\begin{equation}\label{coherence:composition}
\begin{tikzpicture}[baseline = 25]
\node at (0,0) {$(x,y')$};
\node at (0,2) {$(x,y)$};
\node at (2.5,0) {$(x',y')$};
\node at (2.5,2) {$(x',y)$};
\node at (5,0) {$(x'',y')$};
\node at (5,2) {$(x'',y)$};
\draw[->] (0,1.7) -- (0,0.3);
\node[scale = 0.8] at (-0.4,1) {$(x,g)$};
\draw[->] (2.5,1.7) -- (2.5,0.3);
\node[scale = 0.8, fill = white] at (2.5,1) {$(x',g)$};
\draw[->] (5,1.7) -- (5,0.3);
\node[scale = 0.8] at (5.5,1) {$(x'',g)$};
\draw[->] (0.5,0) -- (2,0);
\node[scale = 0.8] at (1.25,-0.3) {$(f,y')$};
\draw[->] (0.5,2) -- (2,2);
\node[scale = 0.8] at (1.25,2.3) {$(f,y)$};
\draw[->] (3,0) -- (4.5,0);
\node[scale = 0.8] at (3.75,-0.3) {$(f',y')$};
\draw[->] (3,2) -- (4.5,2);
\node[scale = 0.8] at (3.75,2.3) {$(f',y)$};
\draw[->, double] (0.5,0.4) -- (2,1.6);
\node[scale = 0.8, fill = white] at (1.25,1) {$\gamma_{f,g}$};
\draw[->, double] (3,0.4) -- (4.5,1.6);
\node[scale = 0.8, fill = white] at (3.75,1) {$\gamma_{f',g}$};
\end{tikzpicture}
\hspace{3pt}=\hspace{3pt}
\begin{tikzpicture}[baseline = 25]
\node at (0,0) {$(x,y')$};
\node at (0,2) {$(x,y)$};
\node at (2.5,0) {$(x'',y')$};
\node at (2.5,2) {$(x'',y)$};
\draw[->] (0,1.7) -- (0,0.3);
\node[scale = 0.8] at (-0.4,1) {$(x,g)$};
\draw[->] (2.5,1.7) -- (2.5,0.3);
\node[scale = 0.8] at (3,1) {$(x'',g)$};
\draw[->] (0.5,0) -- (2,0);
\node[scale = 0.8] at (1.25,-0.3) {$(f'f,y')$};
\draw[->] (0.5,2) -- (2,2);
\node[scale = 0.8] at (1.25,2.3) {$(f'f,y)$};
\draw[->, double] (0.5,0.4) -- (2,1.6);
\node[scale = 0.8, fill = white] at (1.25,1) {$\gamma_{f'f,g}$};
\end{tikzpicture}
\end{equation}
and their ``vertical'' counterparts, involving the 2-category structure of $\B$.
A description of the 2-cells in $\A \boxtimes \B$ as equivalence classes of (vertically composable) strings of equivalence classes of (horizontally composable) strings of generating 2-cells, making this presentation more explicit, can be found in \cite[Theorem I.4.9]{Gray}.

This tensor product extends to a functor $\twoCat \times \twoCat \to \twoCat$, and forms part of a biclosed monoidal structure on $\twoCat$.
In particular, there are natural bijections
\[
\twoCat(\B,[\A,\C]\lax) \cong \twoCat(\A \boxtimes \B, \C) \cong \twoCat(\A,[\B,\C]\oplax)
\]
where $[\A,\C]\lax$ is the 2-category of 2-functors $\A \to \C$, lax natural transformations and modifications, and $[\B,\C]\oplax$ is similar but has oplax natural transformations as 1-cells.
This monoidal structure is not braided, but we have natural isomorphisms $(\A \boxtimes \B)\op \cong \B\op \boxtimes \A\op$ and $(\A \boxtimes \B)\co \cong \B\co \boxtimes \A\co$.

\begin{remark}
	The functor $\boxtimes : \twoCat \times \twoCat \to \twoCat$ does not extend to a 2-functor (with respect to 2-natural transformations).
	For instance, regard the unique non-identity 1-cell in $[1;0]$ as a 2-natural transformation between two 2-functors $[0] \to [1;0]$, and consider the ``tensor'' of this transformation with another copy of $[1;0]$.
\end{remark}

\subsection{2-quasi-categorical Gray tensor product}
\begin{definition}\label{Gray definition}
	For each $a \ge 0$, we define the \emph{$a$-ary Gray tensor product} functor
	\[
	\tensor_a : \underbrace{\hattheta \times \dots \times \hattheta}_{\text{$a$ times}} \to \hattheta
	\]
	by extending the composite
	\[
	\begin{tikzcd}
	\cell \times \dots \times \cell \arrow [r, hook] & \twoCat \times \dots \times \twoCat \arrow [r, "\ttensor_a"] & \twoCat \arrow [r, "N"] & \hattheta
	\end{tikzcd}
	\]
	cocontinuously in each variable, where the second map $\ttensor_a$ is the $a$-ary Gray tensor product of $2$-categories.
	For $a \le 1$, we choose $\tensor_0 \defeq \cell[0]$ and $\tensor_1 \defeq \id_{\cell}$.
\end{definition}
Therefore the tensor product $\tensor_a(X^1,\dots,X^a)$ admits a coend description
\[
\tensor_a(X^1,\dots,X^a) \cong \int^{\theta_1,\dots,\theta_a \in \cell}
\bigl(X^1_{\theta_1}\times\dots\times X^a_{\theta_a}\bigr)*
N\bigl(\ttensor_a(\theta_1, \dots, \theta_a)\bigr)
\]
where $*$ denotes the copower.
One can make this presentation more explicit and identify $\tensor_a(X^1,\dots,X^a)$ with the set of connected components of a suitable category; see the remark after \cref{fondind}.

Note that the obvious 2-functors $\pi_i : \ttensor_a(\theta_1,\dots,\theta_a) \to \theta_i$ induce cellular maps $\pi_i : \tensor_a(X^1,\dots,X^a) \to X^i$ for $1 \le i \le a$.

\begin{remark}
    One might (reasonably) object that, although \cref{Gray definition} involves the ordinary Gray tensor product $\ttensor_a$, this does not fully justify calling the functor $\tensor_a$ the Gray tensor product.
    Ideally we would respond to such an objection by exhibiting that ``everything'' we ever do with $\ttensor_a$ admits an analogue for $\tensor_a$.
    The present paper, however, will only focus on proving the existence of a homotopical biclosed monoidal structure.
    Further justifications will be given in our future work.
    
    Meanwhile, the following argument is the best justification we can provide.
    For objects $\theta,\theta' \in \bigl\{[0], [1;0],[1;1]\bigr\}$, it is easy to compute and even draw the binary tensor product $\cell^\theta \otimes \cell^{\theta'} = \tensor_2\bigl(\cell^\theta, \cell^{\theta'}\bigr)$.
    These low-dimensional examples look ``correct'' (in the sense that they match what one would expect the Gray tensor products of these simple $(\infty,2)$-categories to be), hence:
    \begin{enumerate}
        \item $\cell^\theta \otimes \cell^{\theta'}$ must be ``correct'' for all $\theta,\theta' \in \cell$ since $\otimes$ is left Quillen (\cref{binary tensor is left Quillen}) and any $\cell^\theta$ is a homotopy colimit of $\cell[0]$, $\cell[1;0]$ and $\cell[1;1]$ (via the spine);
        \item thus the whole functor $\otimes = \tensor_2$, whose action is solely determined by its restriction on the representable objects (and morphisms), must be ``correct''; and
        \item it follows that $\tensor_a$ must be ``correct'' for arbitrary $a$ since $\tensor$ is associative up to homotopy (\cref{associativity}).
    \end{enumerate}
\end{remark}

\begin{remark}
	For \emph{complicial sets} (which model $(\infty,\infty)$-categories), a relatively simple definition of the Gray tensor product was given by Verity in \cite{Verity:strict,Verity:I}, where he also proved the complicial counterpart of our main results (and much more).
	One drawback of complicial sets is that there is only one obvious duality operation, namely the \emph{odd dual} induced by the automorphism $(-)\op$ on $\Delta$, although one would expect to be able to reverse the $n$-cells for any $n \ge 1$.
	On the other hand, for 2-quasi-categories both the horizontal and vertical duals are easy to describe, but the Gray tensor product does not admit a concrete description.
\end{remark}

\subsection{Tensoring cells}
Since the objects of $\cell$ are very simple 2-categories, we can describe $\ttensor_a(\ttheta) = \ttensor_a(\theta_1,\dots,\theta_a)$ explicitly for any $\theta_i \in \cell$.
First, consider the case where $\theta_i = [n_i;\zzero]$ for each $i$.
We will describe a 2-category $\T$ and prove that it is isomorphic to $\ttensor_a(\ttheta)$.
The underlying 1-category of $\T$ is the free one on the directed graph given as follows:
\begin{itemize}
	\item the vertex set is $\{0,\dots,n_1\} \times \dots \times \{0,\dots,n_a\}$; and
	\item there is a unique edge
	\begin{equation}\label{atomic}
	(x_1,\dots,x_i-1,\dots,x_a) \to (x_1,\dots,x_i,\dots,x_a)
	\end{equation}
	whenever $1 \le i \le a$, $0 < x_i \le n_i$ and $0 \le x_j \le n_j$ for $j \neq i$.
\end{itemize}
Before describing the 2-cells in $\T$, let us analyse the 1-cells.
\begin{definition}
	Given $0 \le s_i \le t_i \le n_i$ for each $i$, let
	\[
	S(\ss,\tt) \defeq \bigl\{(i|k):1 \le i \le n,~s_i < k \le t_i\bigr\}.
	\]
\end{definition}
We have adopted the notation $(i|k)$ in order to distinguish the elements of $S(\ss,\tt)$ from other kinds of pairs, \emph{e.g.}~objects in $\theta \boxtimes \theta'$.

Observe that for any 1-cell $f$ in $\T$ from $\ss = (s_1,\dots,s_n)$ to $\tt = (t_1,\dots,t_n)$, assigning the pair $(i|x_i)$ to the atomic factor of the form (\ref{atomic}) yields a bijection between $S(\ss,\tt)$ and the set of atomic factors of $f$.
Moreover, the obvious total order on the latter set induces a total order $\pre$ on $S(\ss,\tt)$ satisfying

\begin{itemize}
	\item[(\textreferencemark)] $(i|k)\pre(i|\ell)$ for any $1 \le i \le a$ and $s_i < k \le \ell \le t_i$.
\end{itemize}
Informally speaking, $\pre$ orders the set $S(\ss,\tt)$ of ``instructions'' where $(i|k)$ is to be interpreted as ``move in the $i$-th direction by one step so that the new $i$-th coordinate is $k$''.
Conversely, any total order $\pre$ satisfying (\textreferencemark) uniquely determines a 1-cell from $\ss$ to $\tt$.
Hence we may identify the 1-cells $\ss\to\tt$ in $\T$ with such total orders on $S(\ss,\tt)$.

\begin{definition}
    A \emph{shuffle} on $S(\ss,\tt)$ is a total order $\pre$ on $S(\ss,\tt)$ satisfying (\textreferencemark).
\end{definition}

Finally we define the hom-category $\T(\ss,\tt)$ to be the poset given by the partial order $\LHD$ defined below.
It is straightforward to check that $\T$ is a poset-enriched category and hence a 2-category.

\begin{definition}
	Let $\pre$ and $\pre'$ be shuffles on $S(\ss,\tt)$.
	Then $\pre \LHD \pre'$ if and only if $(i|k) \pre (j|\ell)$ and $i<j$ imply $(i|k) \pre' (j|\ell)$ for any $(i|k),(j|\ell) \in S(\ss,\tt)$.
\end{definition}

For instance, when $a=2$ and $n_1=n_2 = 1$, the 2-category $\T$ looks like
\[
\begin{tikzpicture}[baseline = 25]
\node at (0,0) {$(0,1)$};
\node at (0,2) {$(0,0)$};
\node at (2.5,0) {$(1,1)$};
\node at (2.5,2) {$(1,0)$};
\draw[->] (0,1.7) -- (0,0.3);
\node[scale = 0.8] at (-0.5,1) {$(2|1)$};
\draw[->] (2.5,1.7) -- (2.5,0.3);
\node[scale = 0.8] at (3,1) {$(2|1)$};
\draw[->] (0.5,0) -- (2,0);
\node[scale = 0.8] at (1.25,-0.3) {$(1|1)$};
\draw[->] (0.5,2) -- (2,2);
\node[scale = 0.8] at (1.25,2.3) {$(1|1)$};
\draw[->, double] (0.5,0.4) -- (2,1.6);
\node[scale = 0.8, fill = white] at (1.25,1) {$\LHD$};
\end{tikzpicture}
\]
where the 2-cell corresponds to the relation
\[
\bigl\{(2|1)\pre(1|1)\bigr\} \LHD \bigl\{(1|1)\pre(2|1)\bigr\}.
\]

\begin{lemma}\label{tensor of bars}
	$\ttensor_a(\theta_1,\dots,\theta_a) \cong \T$.
\end{lemma}
\begin{proof}
	It is easy to check using the generators-and-relations presentation of the Gray tensor product that $\ttensor_a(\ttheta)$ and $\T$ have isomorphic underlying 1-categories.
	Moreover, the 2-cells in the former 2-category are generated (under vertical and horizontal compositions) by those of the form
	
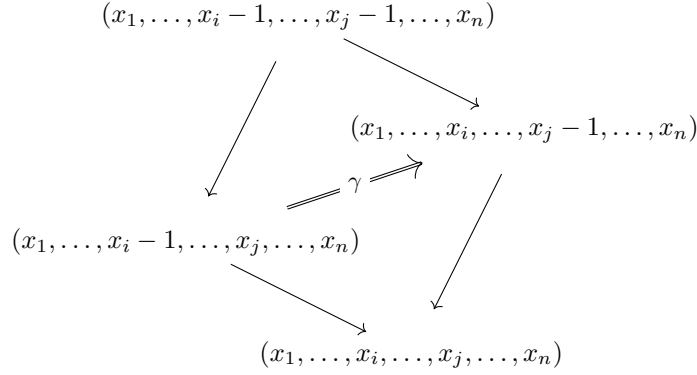
\[
	\begin{tikzpicture}[scale = 1.5]
	\node at (-1,3){$(x_1,\dots,x_i-1,\dots,x_j-1,\dots,x_n)$};
	\node at (-2,1) {$(x_1,\dots,x_i-1,\dots,x_j,\dots,x_n)$};
	\node at (1,2) {$(x_1,\dots,x_i,\dots,x_j-1,\dots,x_n)$};
	\node at (0,0) {$(x_1,\dots,x_i,\dots,x_j,\dots,x_n)$};
	
	\draw[->] (-0.6,2.8)--(0.6,2.2);
	\draw[->] (-1.2,2.6)--(-1.8,1.4);
	\draw[->] (-1.6,0.8)--(-0.4,0.2);
	\draw[->] (0.8,1.6)--(0.2,0.4);
	
	\draw[->,double] (-1.1,1.3) --(0.1,1.7);
	\node[fill = white, scale = 0.9] at (-0.5,1.5) {$\gamma$};
	\end{tikzpicture}
	\]
	(see (\ref{gamma})) and this 2-cell has the same domain and codomain as the 2-cell
	
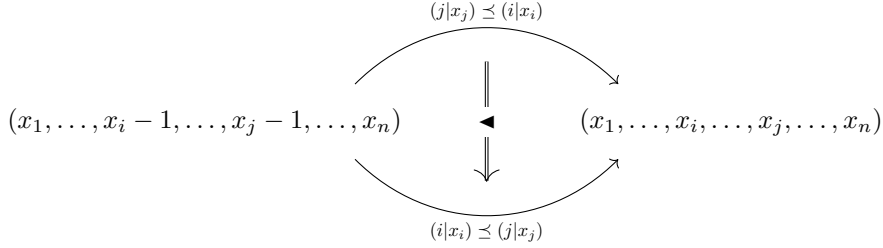
\[
	\begin{tikzpicture}
	\node at (0,0){$(x_1,\dots,x_i-1,\dots,x_j-1,\dots,x_n)$};
	\node at (7,0) {$(x_1,\dots,x_i,\dots,x_j,\dots,x_n)$};
	
	\draw[->] (2,0.5) .. controls (3,1.5) and (4.5,1.5) .. node[scale = 0.7, above]{$(j|x_j)\pre(i|x_i)$} (5.5,0.5);
	\draw[->] (2,-0.5) .. controls (3,-1.5) and (4.5,-1.5) .. node[scale = 0.7, below]{$(i|x_i)\pre(j|x_j)$} (5.5,-0.5);
	
	\draw[->, double] (3.75,0.8) -- (3.75,-0.8);
	\node[fill=white, scale=0.9] at (3.75,0) {$\LHD$};
	\end{tikzpicture}
	\]
	in $\T$.
	It follows that we have a 2-functor $F : \ttensor_a(\ttheta) \to \T$ which is bijective on objects and 1-cells.
	
	We next prove that this 2-functor $F$ is locally full.
	Consider a morphism in the hom-category $\T(\ss,\tt)$, or equivalently, a pair of shuffles $\pre_0$ and $\pre_1$ on $S(\ss,\tt)$ such that $\pre_0 \LHD \pre_1$.
	\begin{temp}
	By the \emph{rank} of this morphism, we mean the cardinality of the set
	\[
	\Bigl\{\bigl((i|k),(j|\ell)\bigr) \in S(\ss,\tt)^2 : (i|k) \suc_0 (j|\ell),~(i|k) \pre_1 (j|\ell)\Bigr\}.
	\]
	\end{temp}
	We prove by induction on the rank that this morphism is in the image of $F$.
	The base case is easy since the rank of a morphism is $0$ if and only if it is the identity.
	For the inductive step, assume that the rank is positive.
	Then there is a pair $(i_\dagger|k_\dagger),(j_\dagger|\ell_\dagger) \in S(\ss,\tt)$ such that:
	\begin{itemize}
		\item $(i_\dagger|k_\dagger)$ is the immediate $\pre_0$-successor of $(j_\dagger|\ell_\dagger)$; and
		\item $(i_\dagger|k_\dagger) \pre_1 (j_\dagger|\ell_\dagger)$.
	\end{itemize}
	(Such a pair must exist for otherwise $\pre_0$ and $\pre_1$ coincide).
	Now define a total order $\pre$ on $S(\ss,\tt)$ so that it agrees with $\pre_0$ on all pairs of elements in $S(\ss,\tt)$ except that $(i_\dagger|k_\dagger) \pre (j_\dagger|\ell_\dagger)$.
	Then clearly $\pre$ is a shuffle and moreover we have $\pre_0 \LHD \pre \LHD \pre_1$, giving a factorisation of the original morphism.
	The first paragraph of this proof implies that any morphism of rank 1, and in particular the first factor $\pre_0 \LHD \pre$, is in the image of $F$.
	The second factor is in the image too by the inductive hypothesis.
	This proves that $F$ is locally full.
	
	The proof that $F$ is locally faithful is deferred to \cref{braid}.
\end{proof}

\begin{definition}
	Given any $\theta = \nq \in \cell$, we will write $\bar \theta$ for $[n;\zzero] \in \cell$.
\end{definition}
In the rest of this paper, any unlabelled cellular operator of the form $\theta \to \bar \theta$ is assumed to be the unique one $\nq \to [n;\zzero]$ whose horizontal component is the identity.

The following lemma (combined with \cref{tensor of bars}) provides an explicit description of $\ttensor_a(\ttheta)$ for arbitrary $\theta_i \in \cell$.

\begin{lemma}\label{Gray pullback bar}
	For any $\theta_1,\dots,\theta_a \in \cell$, the square
	\[
	\begin{tikzcd}[row sep = large]
	\ttensor_a(\theta_1,\dots,\theta_a)
	\arrow [r]
	\arrow [d, "{\langle\pi_1,\dots,\pi_a\rangle}", swap] &
	\ttensor_a(\bar \theta_1, \dots, \bar \theta_a)
	\arrow [d, "{\langle\pi_1,\dots,\pi_a\rangle}"] \\
	\theta_1 \times \dots \times \theta_a
	\arrow [r] &
	\bar \theta_1 \times \dots \times \bar \theta_a
	\end{tikzcd}
	\]
	is a pullback in $\twoCat$.
\end{lemma}
\begin{proof}
	We will sketch the proof and leave the details to the reader.
	Clearly the above square is at least commutative, thus there is an induced 2-functor from $\ttensor_a(\ttheta)$ to the pullback of the cospan.
	It is straightforward to see that this 2-functor is bijective on objects and 1-cells.
	Moreover one can check that it is locally full, similarly to the proof of \cref{tensor of bars}.
	It then suffices to prove that $\ttensor_a(\ttheta)$ is poset-enriched.
	It follows from \cref{coherence:composition,coherence:natural} that any 2-cell in $\ttensor_a(\ttheta)$ can be (vertically) factorised as a composite of $\gamma$'s (\ref{gamma}) followed by a composite of 2-cells ``coming from $\theta_i$'s'' (\ref{generating 2-cells}); \emph{e.g.}~for $a = 2$ such a factorisation typically looks like
	
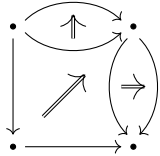
\[
	\begin{tikzpicture}[scale = 0.8]
	\filldraw
	(0,0) circle [radius = 1.2pt]
	(2,0) circle [radius = 1.2pt]
	(0,2) circle [radius = 1.2pt]
	(2,2) circle [radius = 1.2pt];
	
	\draw[->] (0,1.8) -- (0,0.2);
	\draw[->] (0.2,0) -- (1.8,0);
	\draw[->] (0.2,2.1) .. controls (0.6,2.5) and (1.4,2.5) .. (1.8,2.1);
	\draw[->] (0.2,1.9) .. controls (0.6,1.5) and (1.4,1.5) .. (1.8,1.9);
	\draw[->] (1.9,1.8) .. controls (1.5,1.4) and (1.5,0.6) .. (1.9,0.2);
	\draw[->] (2.1,1.8) .. controls (2.5,1.4) and (2.5,0.6) .. (2.1,0.2);
	
	\draw[->, double] (0.5,0.5) -- (1.2,1.2);
	\draw[->, double] (1,1.8) -- (1,2.2);
	\draw[->, double] (1.8,1) -- (2.2,1);
	\end{tikzpicture}
	\]
	where the second factor is the horizontal composite of the two globe-shaped 2-cells.
	Observe that for any parallel pair of 2-cells, this factorisation yields the same middle 1-cell.
	Therefore the desired result follows from \cref{tensor of bars} and the observation that each $\theta_i$ is poset-enriched.
\end{proof}

\cref{Gray mono} below is straightforward to prove using this explicit description of $\ttensor_a(\ttheta)$.
Note that the underlying 1-category of $\ttensor_a(\ttheta)$ is free on the obvious graph and hence each 1-cell in $\ttensor_a(\ttheta)$ admits a unique atomic decomposition.

\begin{definition}
    By the \emph{endpoints} of an $(n;\qq)$-cell $\phi : \nq \to \ttensor_a(\ttheta)$, we mean the objects $\phi (0)$ and $\phi (n)$.
\end{definition}

\begin{definition}\label{underlying orders}
	Let $\phi : [1;q] \to \ttensor_a(\theta_1,\dots,\theta_a)$ be a $(1;q)$-cell with endpoints $\ss,\tt$.
	By the \emph{underlying shuffles} of $\phi$, we mean the shuffles on $S(\ss,\tt)$ corresponding to the composites
	\[
	\begin{tikzcd}
	{[1;0]}
	\arrow [r, "\eta_v^i"] &
	{[1;q]}
	\arrow [r, "\phi"] &
	\ttensor_a(\theta_1,\dots,\theta_a)
	\arrow [r] &
	\ttensor_a(\bar \theta_1,\dots,\bar{\theta}_a)
	\end{tikzcd}
	\]
	for $0 \le i \le q$.
\end{definition}
\begin{definition}\label{visits}
	Given a 1-cell $f: \ss \to \tt$ and an object $\xx$ in $\ttensor_a(\ttheta)$, we say \emph{$f$ visits $\xx$} to mean that the atomic decomposition of $f$ involves $\xx$.
	Equivalently, $f$ visits $x$ if and only if either $\xx = \ss$ or there is (necessarily unique) $(j|\ell) \in S(\ss,\tt)$ such that
	\[
	x_i = \min\bigl(\bigl\{k: (i|k) \pre (j|\ell)\bigr\} \cup \{s_i\}\bigr)
	\]
	for each $i$ where $\pre$ is the (unique) underlying shuffle of $f$.
\end{definition}

\begin{lemma}\label{Gray mono}
	Let $\delta_i : \zeta_i \to \theta_i$ be a face operator in $\cell$ for $1 \le i \le a$.
	Then
	\[
	\ttensor_a(\delta_1,\dots,\delta_n) : \ttensor_a(\zeta_1,\dots,\zeta_a) \to \ttensor_a(\theta_1, \dots, \theta_a)
	\]
	is a monomorphism in $\twoCat$.
	Consequently, its nerve
	\[
	\tensor_a(\delta_1,\dots,\delta_n) : \tensor_a\bigl(\cell^{\zeta_1},\dots,\cell^{\zeta_a}\bigr) \to \tensor_a\bigl(\cell^{\theta_1},\dots,\cell^{\theta_a}\bigr)
	\]
	is a monomorphism in $\hattheta$.
	Moreover, a $\kappa$-cell $\phi : \kappa \to \ttensor_a(\ttheta)$ in $\tensor_a\bigl(\cell^{\theta_1},\dots,\cell^{\theta_a}\bigr)$ is in the image of this map if and only if:
	\begin{itemize}
		\item[(i)] $\begin{tikzcd}[column sep = small] \kappa \arrow [r, "\phi"] & \ttensor_a(\ttheta) \arrow [r, "\pi_i"] & \theta_i \end{tikzcd}$ factors through $\delta_i$ for each $i$; and
		\item[(ii)] if a 1-cell $f$ in the image of $\phi$ visits two distinct objects $\xx$ and $\yy$ such that $x_i = y_i$ for some $i$, then the object $x_i = y_i \in \theta_i$ is in the image of $\delta_i$.
	\end{itemize}
\end{lemma}

For example, consider the map
\[
\delta_h^1\otimes\id : \cell[1;0] \otimes \cell[1;0] \to \cell[2;0] \otimes \cell[1;0]
\]
where $\otimes = \tensor_2$.
This map is the nerve of the inclusion 2-functor
\[
\left\{\begin{tikzpicture}[baseline = 12]
\filldraw
(0,0) circle [radius = 1pt]
(2,0) circle [radius = 1pt]
(0,1) circle [radius = 1pt]
(2,1) circle [radius = 1pt];

\draw[->] (0.1,0) -- (1.9,0);
\draw[->] (0.1,1) -- (1.9,1);
\draw[->] (0,0.9) -- (0,0.1);
\draw[->] (2,0.9) -- (2,0.1);

\draw[->, double] (0.3,0.3) -- (1.7,0.7);
\end{tikzpicture}\right\}
\hspace{5pt} \incl \hspace{5pt}
\left\{\begin{tikzpicture}[baseline = 12]
\filldraw
(0,0) circle [radius = 1pt]
(1,0) circle [radius = 1pt]
(2,0) circle [radius = 1pt]
(0,1) circle [radius = 1pt]
(1,1) circle [radius = 1pt]
(2,1) circle [radius = 1pt];
\draw[->] (0.1,0) -- (0.9,0);
\draw[->] (1.1,0) -- (1.9,0);
\draw[->] (0.1,1) -- (0.9,1);
\draw[->] (1.1,1) -- (1.9,1);
\draw[->] (0, 0.9) -- (0, 0.1);
\draw[->] (1, 0.9) -- (1, 0.1);
\draw[->] (2, 0.9) -- (2, 0.1);
\draw[->, double] (0.3, 0.3) -- (0.7, 0.7);
\draw[->, double] (1.3, 0.3) -- (1.7, 0.7);
\end{tikzpicture}\right\}.
\]
A cell in the codomain violates (i) if and only if it contains an object in the middle column, \emph{e.g.}
\[
\left\{\begin{tikzpicture}[baseline = 12]
\filldraw[gray!50!white]
(0,0) circle [radius = 1pt]
(1,0) circle [radius = 1pt];
\filldraw
(2,0) circle [radius = 1pt]
(0,1) circle [radius = 1pt]
(1,1) circle [radius = 1pt]
(2,1) circle [radius = 1pt];
\draw[->, gray!50!white] (0.1,0) -- (0.9,0);
\draw[->, gray!50!white] (1.1,0) -- (1.9,0);
\draw[->] (0.1,1) -- (0.9,1);
\draw[->] (1.1,1) -- (1.9,1);
\draw[->, gray!50!white] (0, 0.9) -- (0, 0.1);
\draw[->, gray!50!white] (1, 0.9) -- (1, 0.1);
\draw[->] (2, 0.9) -- (2, 0.1);
\draw[->, double, gray!50!white] (0.3, 0.3) -- (0.7, 0.7);
\draw[->, double, gray!50!white] (1.3, 0.3) -- (1.7, 0.7);
\end{tikzpicture}\right\},
\]
and it violates (ii) if and only if it involves moving down in the middle column, \emph{e.g.}
\[
\left\{\begin{tikzpicture}[baseline = 12]
\filldraw[gray!50!white]
(0,0) circle [radius = 1pt]
(2,1) circle [radius = 1pt];
\filldraw
(2,0) circle [radius = 1pt]
(0,1) circle [radius = 1pt];
\draw[->, gray!50!white] (0.1,0) -- (0.9,0);
\draw[->] (0.1,1) -- (1,1) -- (1,0) -- (1.9,0);
\draw[->, gray!50!white] (1.1,1) -- (1.9,1);
\draw[->, gray!50!white] (0, 0.9) -- (0, 0.1);
\draw[->, gray!50!white] (2, 0.9) -- (2, 0.1);
\draw[->, double, gray!50!white] (0.3, 0.3) -- (0.7, 0.7);
\draw[->, double, gray!50!white] (1.3, 0.3) -- (1.7, 0.7);
\end{tikzpicture}\right\}.
\]

\section{$\hat \otimes_a$ preserves monomorphisms}\label{section mono}

Fix $a \ge 2$ and $\theta_1, \dots, \theta_a \in \cell$.
The aim of this section is to prove the following lemma.
\begin{lemma}\label{boundary}
	The Leibniz Gray tensor product
	\begin{equation}\label{Leibniz n-ary tensor}
	\hat \tensor_a\bigl(\partial\cell^{\theta_1} \incl \cell^{\theta_1}, \dots, \partial\cell^{\theta_a} \incl \cell^{\theta_a}\bigr)
	\end{equation}
	is a monomorphism.
\end{lemma}
\begin{proof}
By \cref{Leibniz mono}, it suffices to prove that the functor
\[
G : \two^a \to \hattheta
\]
(defined as in \cref{Leibniz} with $F = \tensor_a$) sends each square of the form (\ref{ij square}) to a pullback square of monomorphisms.
We will prove in \cref{fondind} below that $G$ sends each map in $\two^a$ to a monomorphism.
Assuming this fact, it is straightforward to deduce using \cref{Gray mono} that the desired square is indeed a pullback.
\end{proof}

Observe that the hypothesis in the following lemma is satisfied for $X^i = \cell^{\theta_i}$ and $X^i = \partial\cell^{\theta_i}$.
\begin{lemma}\label{fondind}
    Fix $1 \le b \le a$ and let $X^i \in \hattheta$ for $1 \le i \le a$ with $i \neq b$.
    Suppose that in each $X^i$, any face of a non-degenerate cell is itself non-degenerate.
    Then
	\begin{equation}\label{n-ary fondind}
	\tensor_a(X^1, \dots, X^{b-1}, \partial\cell^\theta \incl \cell^\theta, X^{b+1},\dots,X^a)
	\end{equation}
	is a monomorphism for any $\theta \in \cell$.
\end{lemma}

\begin{remark}
    Given a small category $\mathscr{X}$, a functor $G: \mathscr{X} \to \Set$ and a weight $W: \mathscr{X}\op \to \Set$, the colimit of $G$ weighted by $W$ may be computed as the (conical) colimit of the composite
    \[
    H : 
    \begin{tikzcd}
    \int W
    \arrow [r] &
    \mathscr{X}
    \arrow [r, "G"] &
    \Set
    \end{tikzcd}
    \]
    where $\int W \to \mathscr{X}$ is the Grothendieck construction of $W$.
    Moreover the colimit of any $\Set$-valued functor $H$ is isomorphic to the set of connected components in $\int H$.
    
    Since the coend formula expresses the cellular set $\tensor_a(X^1, \dots, X^a)$ as the colimit of the composite
    \[
	\begin{tikzcd}
	\cell \times \dots \times \cell \arrow [r, hook] & \twoCat \times \dots \times \twoCat \arrow [r, "\ttensor_a"] & \twoCat \arrow [r, "N"] & \hattheta
	\end{tikzcd}
	\]
    weighted by
    \[
    (\theta_1, \dots, \theta_a) \mapsto X^1_{\theta_1}\times\dots\times X^a_{\theta_a},
    \]
    it follows that the value of $\tensor_a(X^1, \dots, X^a)$ at any $\zeta \in \cell$ may be computed as the set of connected components in a suitable category.
    This is how we obtain the categories $\B$ and $\C$ in the proof below.
\end{remark}

\begin{proof}
Fix $\zeta \in \cell$.
We will give a more explicit description of the $\zeta$-component of the natural transformation (\ref{n-ary fondind}).

Let $\A$ denote the category whose objects are $(2a+1)$-tuples
\[
(\kkappa,\phi,\xx) = (\kappa_1,\dots,\kappa_a,\phi,x_1,\dots,x_a)
\]
where:
\begin{itemize}
	\item $\kappa_i \in \cell$ for each $1 \le i \le a$;
	\item $\phi : \zeta \to \ttensor_a(\kappa_1,\dots,\kappa_a)$ is a 2-functor;
	\item $x_i \in X^i_{\kappa_i}$ for $i \neq b$; and
	\item $x_b : \kappa_b \to \theta$ is a 2-functor
\end{itemize}
and whose morphisms $\aalpha : (\kkappa,\phi,\xx) \to (\llambda,\chi,\yy)$ consist of cellular operators $\alpha_i : \kappa_i \to \lambda_i$ for $1 \le i \le a$ such that:
\begin{itemize}
	\item $\chi = \ttensor_a(\alpha_1,\dots,\alpha_a) \circ \phi$; and
	\item $x_i = y_i \cdot \alpha_i$ for $1 \le i \le a$.
\end{itemize}
\begin{notation}\label{things}
	If $\boldsymbol{\tau} = (\tau_1,\dots,\tau_a)$ is an $a$-tuple of ``things'' and $\upsilon$ is another ``thing'' then we will denote by $\boldsymbol{\tau}\{\upsilon\}$ the $a$-tuple
	\[
	\boldsymbol{\tau}\{\upsilon\} \defeq (\tau_1,\dots,\tau_{b-1},\upsilon,\tau_{b+1},\dots,\tau_a).
	\]
\end{notation}

Let $\B$ (``boundary'') and $\C$ (``cell'') be the full subcategories of $\A$ spanned by those $(\kkappa,\phi,\xx)$ such that:
\begin{itemize}
	\item [($\B$)] $x_b \in \partial\cell^\theta$; and
	\item [($\C$)] $\kappa_b = \theta$ and $x_b = \id_\theta$
\end{itemize} 
respectively.
Then there is a functor $F : \B \to \C$ given by
\[
F(\kkappa,\phi,\xx) = \bigl(\kkappa\{\theta\}, \ttensor_a(\iid\{x_b\}) \circ \phi, \xx\{\id_\theta\}\bigr)
\]
and
\[
F(\aalpha) = \aalpha\{\id_\theta\}.
\]
The $\zeta$-component of the natural transformation (\ref{n-ary fondind}) can be identified with the function $\pi_0(F) : \pi_0(\B) \to \pi_0(\C)$ where $\pi_0 : \Cat \to \Set$ is the connected components functor.

Thus, to prove that (\ref{n-ary fondind}) is a monomorphism, it suffices to show that if $(\kkappa,\phi,\xx)$ and $(\boldsymbol{\kappa'},\phi',\boldsymbol{x'})$ are objects in $\B$ and there is a zigzag of (possibly identity) arrows
\begin{equation}\label{zigzag}
\begin{tikzcd}
F(\kkappa,\phi,\xx)
\arrow [r, "\boldsymbol{\alpha^1}"] &
(\boldsymbol{\lambda^1}, \chi^1, \boldsymbol{y^1}) &
\dots
\arrow [l, swap, "\boldsymbol{\alpha^2}"]\\
\dots
\arrow [r, "\boldsymbol{\alpha^m}"] &
(\boldsymbol{\lambda^m}, \chi^m,\boldsymbol{y^m}) &
F(\boldsymbol{\kappa'},\phi',\boldsymbol{x'})
\arrow [l, swap, "\boldsymbol{\alpha^{m+1}}"]
\end{tikzcd}
\end{equation}
in $\C$ then $(\kkappa,\phi,\xx)$ and $(\boldsymbol{\kappa'},\phi',\boldsymbol{x'})$ lie in the same connected component of $\B$.
(Here we are assuming $m$ to be odd so that the zigzag really ends with a left-pointing arrow; we do not lose generality by doing so since $\boldsymbol{\alpha^{m+1}}$ is allowed to be the identity.)

First, we prove that we may assume each object $(\boldsymbol{\lambda^k}, \chi^k, \boldsymbol{y^k})$ to be in the image of $F$.
\begin{temp}
	We call a zigzag of the form (\ref{zigzag}) \emph{$k$-admissible} if the objects $(\boldsymbol{\lambda^\ell}, \chi^\ell, \boldsymbol{y^\ell})$ (but not necessarily the arrows between them) are in the image of $F$ for all $1 \le \ell \le k$.
\end{temp}
\begin{claim*}
	For any $(k-1)$-admissible zigzag of the form (\ref{zigzag}), there exists a $k$-admissible zigzag that has the same length and the same endpoints.
\end{claim*}

\begin{proof}[Proof of the claim]
	The easier case is when $k$ is odd.
	In this case, by assumption we have a map
	\[
	\boldsymbol{\alpha^k} : F(\llambda,\chi,\yy) \to (\boldsymbol{\lambda^k}, \chi^k,\boldsymbol{y^k})
	\]
	for some $(\llambda,\chi,\yy) \in \B$.
	Then it is easy to check that
	\[
	(\boldsymbol{\lambda^k}, \chi^k,\boldsymbol{y^k}) = F\bigl(\boldsymbol{\lambda^k}\{\lambda_b\}, \ttensor_a(\boldsymbol{\alpha^k}\{\id_{\lambda_b}\}) \circ \chi,\boldsymbol{y^k}\{y_b\}\bigr).
	\]
	
	Next suppose that $k$ is even so that we have
	\[
	\begin{tikzcd}
	F(\llambda,\chi,\yy) &
	(\boldsymbol{\lambda^k}, \chi^k,\boldsymbol{y^k})
	\arrow [l, "\boldsymbol{\alpha^{k}}", swap]
	\end{tikzcd}
	\]
	for some $(\llambda,\chi,\yy) \in \B$.
	We first treat the special case where each $\alpha^k_i$ is a face operator.
		By the definition of $\B$, we have $y_b \in \partial\cell^\theta$.
		Thus in the Reedy factorisation
		\[
		y_b : 
		\begin{tikzcd}
		\lambda_b
		\arrow [r, "\sigma"] &
		\lambda'_b
		\arrow [r, "y'_b"] &
		\theta
		\end{tikzcd}
		\]
		the second factor $y'_b$ is a non-identity face map.
		But then we have
		\[
		F(\llambda,\chi,\yy) = F\bigl(\llambda\{\lambda'_b\}, \ttensor_a(\iid\{\sigma\}) \circ \chi,\yy\{y'_b\}\bigr).
		\]
		Thus we may assume that $y_b$ is itself a non-identity face map.
		Then the inner square in
		\[
		\begin{tikzcd}[row sep = huge, column sep = large]
		\zeta
		\arrow [drr, bend left = 15, "\chi^k"]
		\arrow [ddr, bend right, swap, "\chi", end anchor = north west]
		\arrow [dr, dashed, "\psi"] & & \\
		& \ttensor_a(\boldsymbol{\lambda^k}\{\lambda_b\})
		\arrow [r, "\ttensor_a{(\iid\{y_b\})}"]
		\arrow [d, "\ttensor_a{(\boldsymbol{\alpha^k}\{\id_{\lambda_b}\})}"]
		\arrow [dr, phantom] &
		\ttensor_a(\boldsymbol{\lambda^k}\{\theta\})
		\arrow [d, "\ttensor_a{(\boldsymbol{\alpha^k})}"] \\
		& \ttensor_a(\llambda)
		\arrow [r, "\ttensor_a{(\iid\{y_b\})}", swap] &
		\ttensor_a(\llambda\{\theta\})
		\end{tikzcd}
		\]
		is a pullback square (which can be checked using \cref{Gray mono}), and the outer square commutes since $\aalphak$ is a morphism in $\C$.
		Hence we obtain the induced map $\psi$ which then satisfies
		\[
		(\boldsymbol{\lambda^k}, \chi^k,\boldsymbol{y^k}) = F(\boldsymbol{\lambda^k}\{\lambda_b\},\psi,\boldsymbol{y^k}\{y_b\}).
		\]
		This completes the proof of the special case where each $\alpha^k_i$ is a face operator.

		Now consider the general case.
		Note that $(\boldsymbol{\lambda^{k+1}}, \chi^{k+1},\boldsymbol{y^{k+1}})$ is well-defined since $k$ is even and $m$ is odd.
		If $y^{k+1}_i = z_i\cdot\iota_i$ for some $z_i \in X^i_{\mu_i}$ and $\iota_i : \lambda_i \to \mu_i$, then we can replace $\boldsymbol{\alpha^{k+1}}$ and $\boldsymbol{\alpha^{k+2}}$ by their respective composites with $\boldsymbol{\iota}$:
		\[
		\begin{tikzcd}
		& (\boldsymbol{\mu},\ttensor_a(\boldsymbol{\iota})\circ\chi^{k+1},\zz) & \\
		(\boldsymbol{\lambda^{k}}, \chi^{k},\boldsymbol{y^{k}})
		\arrow [r, "\boldsymbol{\alpha^{k+1}}"] &
		(\boldsymbol{\lambda^{k+1}}, \chi^{k+1},\boldsymbol{y^{k+1}})
		\arrow [u, "\boldsymbol{\iota}"] &
		?
		\arrow [l, swap, "\boldsymbol{\alpha^{k+2}}"]
		\end{tikzcd}
		\]
		(in which ``$?$'' is either $(\boldsymbol{\lambda^{k+2}}, \chi^{k+2},\boldsymbol{y^{k+2}})$ or $F(\boldsymbol{\kappa'},\phi',\boldsymbol{x'})$)
		to obtain a new zigzag.
		Thus we may assume that each $y^{k+1}_i$ is non-degenerate.
		Similarly we may assume that each $y_i$ is non-degenerate.
		
		Let $\alpha^{k}_i = \delta^k_i \circ \sigma^{k}_i$ and $\alpha^{k+1}_i = \delta^{k+1}_i\circ\sigma^{k+1}_i$ be the Reedy factorisations of $\alpha^k_i$ and $\alpha^{k+1}_i$ respectively.
		Then we have the solid part of the following commutative diagram in $\C$:
		\[
		\begin{tikzpicture}
		\node at (2.5,2) {$\bigl(\ttensor_a(\boldsymbol{\sigma^k})\circ\chi^k, \boldsymbol{y\cdot\delta^k}\bigr)$};
		\node at (7.5,2) {$\bigl(\ttensor_a(\boldsymbol{\sigma^{k+1}})\circ\chi^k, \boldsymbol{y^{k+1}\cdot\delta^{k+1}}\bigr)$};
		\node at (0,0) {$F(\chi,\yy)$};
		\node at (5,0) {$(\chi^k,\boldsymbol{y^k})$};
		\node at (10,0) {$(\chi^{k+1},\boldsymbol{y^{k+1}})$};
		
		\draw[->] (4,0) -- (1,0);
		\draw[->] (6,0) -- (8.5,0);
		\draw[->] (2,1.6) -- (0.5,0.4);
		\draw[->] (4.5,0.4) -- (3,1.6);
		\draw[->] (5.5,0.4) -- (7,1.6);
		\draw[->] (8,1.6) -- (9.5,0.4);
		\draw[double, dashed] (5,2) -- (4.3,2);

		\node at (2.5,-0.3) {$\boldsymbol{\alpha^k}$};
		\node at (7.5,-0.3) {$\boldsymbol{\alpha^{k+1}}$};
		\node at (1, 1) {$\boldsymbol{\delta^k}$};
		\node at (4.2, 1) {$\boldsymbol{\sigma^k}$};
		\node at (5.8, 1) {$\boldsymbol{\sigma^{k+1}}$};
		\node at (9.4, 1) {$\boldsymbol{\delta^{k+1}}$};
		\end{tikzpicture}
		\]
		where we are omitting the first $a$ coordinates of each object.
		For each $i \neq b$, the cells $y_i \cdot \delta_i^k$ and $y^{k+1}_i \cdot \delta^{k+1}_i$ are non-degenerate since $y_i$ and $y^{k+1}_i$ are non-degenerate and the non-degenerate cells in $X^i$ are assumed to be closed under taking faces.
		Thus both
		\[
		y^k_i = y_i^{k+1} \cdot \alpha^{k+1}_i = (y_i^{k+1} \cdot \delta_i^{k+1}) \cdot \sigma^{k+1}_i
		\]
		and
		\[
		y^k_i = y_i \cdot \alpha^k_i = (y_i \cdot \delta_i^k) \cdot \sigma^k_i
		\]
		express $y^k_i$ as a degeneracy of a non-degenerate cell.
		By the uniqueness of such a presentation, we must have $y^{k+1}_i \cdot \delta^{k+1}_i = y_i \cdot \delta_i^k$ and $\sigma^{k+1}_i = \sigma_i^k$, and so we have an equality as indicated above.
		Therefore we can replace the segment
		\[
		\begin{tikzcd}
		F(\chi,\yy) &
		(\chi^k,\boldsymbol{y^k})
		\arrow [l, swap, "\boldsymbol{\alpha^k}"]
		\arrow [r, "\boldsymbol{\alpha^{k+1}}"] &
		(\chi^{k+1},\boldsymbol{\alpha^{k+1}})
		\end{tikzcd}
		\]
		of the zigzag by
		\[
		\begin{tikzcd}
		F(\chi,\yy) &
		\bigl(\ttensor_a(\boldsymbol{\sigma^k})\circ\chi^k, \boldsymbol{y\cdot\delta^k}\bigr)
		\arrow [l, swap, "\boldsymbol{\delta^k}"]
		\arrow [r, "\boldsymbol{\delta^{k+1}}"] &
		(\chi^{k+1},\boldsymbol{\alpha^{k+1}})
		\end{tikzcd}
		\]
		which reduces the problem to the special case treated above.
		This completes the proof of the claim.
\end{proof}

Thus by induction, we can turn any zigzag of the form (\ref{zigzag}) into a $k$-admissible one for any $k$.
In particular, we may assume that the zigzag is $m$-admissible so that each object $(\boldsymbol{\lambda^k}, \chi^k,\boldsymbol{y^k})$ is in the image of $F$.
Therefore it suffices to prove that, if $(\kkappa,\phi,\xx), (\boldsymbol{\kappa'},\phi',\boldsymbol{x'}) \in \B$ and there is a morphism
\[
\aalpha : F(\kkappa,\phi,\xx) \to F(\boldsymbol{\kappa'},\phi',\boldsymbol{x'})
\]
in $\C$ then $(\kkappa,\phi,\xx)$ and $(\boldsymbol{\kappa'},\phi',\boldsymbol{x'})$ lie in the same connected component of $\B$.
Note that if $x_b : \begin{tikzcd} \kappa_b \arrow [r, "\sigma"] & \lambda \arrow [r, "\delta"] & \theta \end{tikzcd}$ is the Reedy factorisation of $x_b$ then
\[
\iid\{\sigma\} : (\kkappa,\phi,\xx) \to (\kkappa\{\lambda\}, \ttensor_a(\iid\{\sigma\})\circ\phi,\xx\{\delta\})
\]
is a map in $\B$ and $F$ sends it to the identity at $F(\kkappa,\phi,\xx)$.
Thus we may assume that $x_b$ is a non-identity face map into $\theta$, and similarly for $x'_b$.

Consider the following diagram, where the solid part commutes since $\alpha$ is a morphism in $\C$:
\[
\begin{tikzcd}[column sep = small]
\zeta
\arrow [rr, "\phi"]
\arrow [dd, "\phi'" left] 
\arrow [dr, dashed] & & 
\ttensor_a(\kkappa)
\arrow [d, "\pi_b"] \\
& \lambda
\arrow [r, "\iota", dashed]
\arrow [d, "\iota'" left, dashed] &
\kappa_b
\arrow [d, "x_b"] \\
\ttensor_a(\boldsymbol{\kappa'})
\arrow [r, "\pi_b" below] &
\kappa'_b
\arrow [r, "x'_b" below] &
\theta
\end{tikzcd}
\]
We will construct the dashed part as follows.
Let $s,t \in \theta$ be the images of the first and last objects in $\zeta$ under the (unique) composite $\zeta \to \theta$ respectively.
Let $m_0,\dots,m_n \in \theta$ be the increasingly ordered list of objects $m$ such that $s \le m \le t$ and $m$ is in the images of both $x_b$ and $x'_b$.
For each $1 \le k \le n$, let $f^k_0,\dots,f_{q_k}^k$ be those 1-cells $(m-1) \to m$ through which some 1-cell in the image of $\zeta \to \theta$ factors (again increasingly ordered).
Then we set $\lambda = \nq \in \cell$, and the obvious maps
\[
\zeta \to \lambda, \hspace{10pt} \iota : \lambda \to \kappa_b, \hspace{10pt} \iota' : \lambda \to x'_b
\]
fit into the above commutative diagram.
Now consider the following diagram:
\[
\begin{tikzcd}[row sep = huge]
\zeta
\arrow [r, "\phi"]
\arrow [d, dashed, "\chi"]
\arrow [dd, start anchor = south west, end anchor = north west, "\phi'" left, bend right = 50] &
\ttensor_a(\kkappa)
\arrow [r, "\ttensor_a(\iid\{x_b\})"]
\arrow [d, "\ttensor_a(\aalpha\{\id_{\kappa_b}\})"] &
\ttensor_a(\kkappa\{\theta\})
\arrow [dd, "\ttensor_a(\aalpha)"] \\
\ttensor_a(\boldsymbol{\kappa'}\{\lambda\})
\arrow [r, "\ttensor_a(\iid\{\iota\})"]
\arrow [d, "\ttensor_a(\iid\{\iota'\})"] &
\ttensor_a(\boldsymbol{\kappa'}\{\kappa_b\})
\arrow [dr, "\ttensor_a(\iid\{x_b\})" description] & \\
\ttensor_a(\boldsymbol{\kappa'})
\arrow [rr, "\ttensor_a(\iid\{x'_b\})" below] & &
\ttensor_a(\boldsymbol{\kappa'}\{\theta\})
\end{tikzcd}
\]
The perimeter commutes because $\aalpha$ is a morphism in $\C$, whereas the bottom quadrangle commutes because it is the image of the inner square in the previous diagram under $\ttensor_a(\boldsymbol{\kappa'}\{-\})$.
That the right quadrangle commutes is just the functoriality of $\ttensor_a$.
It can be seen from our construction of the span $\begin{tikzcd} \kappa'_b & \lambda \arrow [l, "\iota'" above] \arrow [r, "\iota"] & \kappa_b \end{tikzcd}$ and \cref{Gray mono} that there is a map $\chi$ that renders the whole diagram commutative.
Thus the following zigzag in $\B$ connects $(\kkappa,\phi,\xx)$ and $(\boldsymbol{\kappa'},\phi',\boldsymbol{x'})$:
\[
\begin{tikzcd}
(\kkappa,\phi,\xx)
\arrow [r, "\aalpha\{\id_{\kappa_b}\}"] &
\bigl(\boldsymbol{\kappa'}\{\kappa_b\},\ttensor_a(\aalpha\{\id_{\kappa_b}\})\circ\phi,\boldsymbol{x'}\{x_b\}\bigr)
\arrow [d, equal] \\
\bigl(\boldsymbol{\kappa'}\{\lambda\},\chi,\boldsymbol{x'}\{x_b\cdot\iota\}\bigr)
\arrow [d, equal]
\arrow [r, "\iid\{\iota\}"] &
\bigl(\boldsymbol{\kappa'}\{\kappa_b\},\ttensor_a(\iid\{\iota\})\circ\chi,\boldsymbol{x'}\{x_b\}\bigr) \\
\bigl(\boldsymbol{\kappa'}\{\lambda\},\chi,\boldsymbol{x'}\{x'_b\cdot\iota'\}\bigr)
\arrow [r, "\iid\{\iota'\}"] &
(\boldsymbol{\kappa'},\phi',\boldsymbol{x'})
\end{tikzcd}
\]
This completes the proof.
\end{proof}

\section{Some visual concepts}\label{section visual}
This section is devoted to the notions of \emph{silhouette} and \emph{cut-point}.
The following example exhibits the typical roles these notions will play in the rest of this paper.

\subsection{A low dimensional example}\label{archetype}
We will give a ``visual'' proof that the map
\[
\bigl(\horn_h^1[2;0,0] \incl \cell[2;0,0]\bigr) \hat\otimes \bigl(\partial\cell[1;0] \incl \cell[1;0]\bigr)
\]
is in $\celll(\H_h \cup \H_v)$.
The codomain $\cell[2;0,0] \otimes \cell[1;0]$ of this map is by definition the nerve of the 2-category $[2;0,0] \boxtimes [1;0]$ which looks like
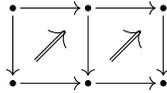
\[
\begin{tikzpicture}
\filldraw
(0,0) circle [radius = 1pt]
(1,0) circle [radius = 1pt]
(2,0) circle [radius = 1pt]
(0,1) circle [radius = 1pt]
(1,1) circle [radius = 1pt]
(2,1) circle [radius = 1pt];
\draw[->] (0.1,0) -- (0.9,0);
\draw[->] (1.1,0) -- (1.9,0);
\draw[->] (0.1,1) -- (0.9,1);
\draw[->] (1.1,1) -- (1.9,1);
\draw[->] (0, 0.9) -- (0, 0.1);
\draw[->] (1, 0.9) -- (1, 0.1);
\draw[->] (2, 0.9) -- (2, 0.1);
\draw[->, double] (0.3, 0.3) -- (0.7, 0.7);
\draw[->, double] (1.3, 0.3) -- (1.7, 0.7);
\end{tikzpicture}
\]
and its domain is the cellular subset
\[
X = \bigl(\horn_h^1[2;0,0] \otimes \cell[1;0]\bigr) \cup \bigl(\cell[2;0,0] \otimes \partial\cell[1;0]\bigr)
\]
of $\cell[2;0,0] \otimes \cell[1;0]$.
The first part $\horn_h^1[2;0,0] \otimes \cell[1;0]$ is generated by the nerves of the sub-2-categories
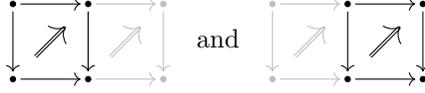
\[
\begin{tikzpicture}[baseline = 12]
\filldraw
(0,0) circle [radius = 1pt]
(1,0) circle [radius = 1pt]
(0,1) circle [radius = 1pt]
(1,1) circle [radius = 1pt];
\filldraw[gray!50!white]
(2,0) circle [radius = 1pt]
(2,1) circle [radius = 1pt];
\draw[->] (0.1,0) -- (0.9,0);
\draw[->, gray!50!white] (1.1,0) -- (1.9,0);
\draw[->] (0.1,1) -- (0.9,1);
\draw[->, gray!50!white] (1.1,1) -- (1.9,1);
\draw[->] (0, 0.9) -- (0, 0.1);
\draw[->] (1, 0.9) -- (1, 0.1);
\draw[->, gray!50!white] (2, 0.9) -- (2, 0.1);
\draw[->, double] (0.3, 0.3) -- (0.7, 0.7);
\draw[->, double, gray!50!white] (1.3, 0.3) -- (1.7, 0.7);
\end{tikzpicture} \hspace{10pt} \text {and} \hspace{10pt}
\begin{tikzpicture}[baseline = 12]
\filldraw
(2,0) circle [radius = 1pt]
(1,0) circle [radius = 1pt]
(2,1) circle [radius = 1pt]
(1,1) circle [radius = 1pt];
\filldraw[gray!50!white]
(0,0) circle [radius = 1pt]
(0,1) circle [radius = 1pt];
\draw[->, gray!50!white] (0.1,0) -- (0.9,0);
\draw[->] (1.1,0) -- (1.9,0);
\draw[->, gray!50!white] (0.1,1) -- (0.9,1);
\draw[->] (1.1,1) -- (1.9,1);
\draw[->, gray!50!white] (0, 0.9) -- (0, 0.1);
\draw[->] (1, 0.9) -- (1, 0.1);
\draw[->] (2, 0.9) -- (2, 0.1);
\draw[->, double, gray!50!white] (0.3, 0.3) -- (0.7, 0.7);
\draw[->, double] (1.3, 0.3) -- (1.7, 0.7);
\end{tikzpicture}
\]
and $\cell[2;0,0] \otimes \partial\cell[1;0]$ is generated by the nerves of
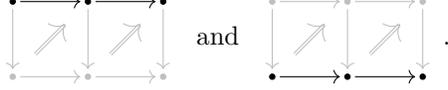
\[
\begin{tikzpicture}[baseline = 12]
\filldraw[gray!50!white]
(0,0) circle [radius = 1pt]
(1,0) circle [radius = 1pt]
(2,0) circle [radius = 1pt];
\filldraw
(0,1) circle [radius = 1pt]
(1,1) circle [radius = 1pt]
(2,1) circle [radius = 1pt];
\draw[->, gray!50!white] (0.1,0) -- (0.9,0);
\draw[->, gray!50!white] (1.1,0) -- (1.9,0);
\draw[->] (0.1,1) -- (0.9,1);
\draw[->] (1.1,1) -- (1.9,1);
\draw[->, gray!50!white] (0, 0.9) -- (0, 0.1);
\draw[->, gray!50!white] (1, 0.9) -- (1, 0.1);
\draw[->, gray!50!white] (2, 0.9) -- (2, 0.1);
\draw[->, double, gray!50!white] (0.3, 0.3) -- (0.7, 0.7);
\draw[->, double, gray!50!white] (1.3, 0.3) -- (1.7, 0.7);
\end{tikzpicture}\hspace{10pt} \text {and} \hspace{10pt}
\begin{tikzpicture}[baseline = 12]
\filldraw
(0,0) circle [radius = 1pt]
(1,0) circle [radius = 1pt]
(2,0) circle [radius = 1pt];
\filldraw[gray!50!white]
(0,1) circle [radius = 1pt]
(1,1) circle [radius = 1pt]
(2,1) circle [radius = 1pt];
\draw[->] (0.1,0) -- (0.9,0);
\draw[->] (1.1,0) -- (1.9,0);
\draw[->, gray!50!white] (0.1,1) -- (0.9,1);
\draw[->, gray!50!white] (1.1,1) -- (1.9,1);
\draw[->, gray!50!white] (0, 0.9) -- (0, 0.1);
\draw[->, gray!50!white] (1, 0.9) -- (1, 0.1);
\draw[->, gray!50!white] (2, 0.9) -- (2, 0.1);
\draw[->, double, gray!50!white] (0.3, 0.3) -- (0.7, 0.7);
\draw[->, double, gray!50!white] (1.3, 0.3) -- (1.7, 0.7);
\end{tikzpicture}\hspace{5pt}.
\]
We categorise the non-degenerate cells in $(\cell[2;0,0] \otimes \cell[1;0]) \setminus X$ into six kinds according to their ``silhouette''.
The cells
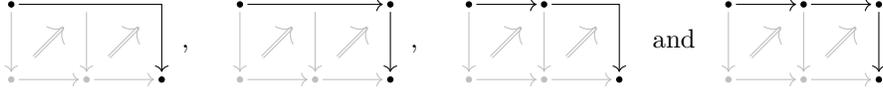
\[
\begin{tikzpicture}[baseline = 12]
\filldraw[gray!50!white]
(0,0) circle [radius = 1pt]
(1,0) circle [radius = 1pt];
\filldraw
(0,1) circle [radius = 1pt]
(2,0) circle [radius = 1pt];
\draw[->, gray!50!white] (0.1,0) -- (0.9,0);
\draw[->, gray!50!white] (1.1,0) -- (1.9,0);
\draw[->] (0.1,1) -- (2,1) -- (2,0.1);
\draw[->, gray!50!white] (0, 0.9) -- (0, 0.1);
\draw[->, gray!50!white] (1, 0.9) -- (1, 0.1);
\draw[->, double, gray!50!white] (0.3, 0.3) -- (0.7, 0.7);
\draw[->, double, gray!50!white] (1.3, 0.3) -- (1.7, 0.7);
\end{tikzpicture}\hspace{5pt}, \hspace{15pt}
\begin{tikzpicture}[baseline = 12]
\filldraw[gray!50!white]
(0,0) circle [radius = 1pt]
(1,0) circle [radius = 1pt];
\filldraw
(0,1) circle [radius = 1pt]
(2,1) circle [radius = 1pt]
(2,0) circle [radius = 1pt];
\draw[->, gray!50!white] (0.1,0) -- (0.9,0);
\draw[->, gray!50!white] (1.1,0) -- (1.9,0);
\draw[->] (0.1,1) -- (1.9,1);
\draw[->] (2,0.9) -- (2,0.1);
\draw[->, gray!50!white] (0, 0.9) -- (0, 0.1);
\draw[->, gray!50!white] (1, 0.9) -- (1, 0.1);
\draw[->, double, gray!50!white] (0.3, 0.3) -- (0.7, 0.7);
\draw[->, double, gray!50!white] (1.3, 0.3) -- (1.7, 0.7);
\end{tikzpicture}\hspace{5pt}, \hspace{15pt}
\begin{tikzpicture}[baseline = 12]
\filldraw[gray!50!white]
(0,0) circle [radius = 1pt]
(1,0) circle [radius = 1pt];
\filldraw
(1,1) circle [radius = 1pt]
(0,1) circle [radius = 1pt]
(2,0) circle [radius = 1pt];
\draw[->, gray!50!white] (0.1,0) -- (0.9,0);
\draw[->, gray!50!white] (1.1,0) -- (1.9,0);
\draw[->] (0.1,1) -- (0.9,1);
\draw[->] (1.1,1) -- (2,1) -- (2,0.1);
\draw[->, gray!50!white] (0, 0.9) -- (0, 0.1);
\draw[->, gray!50!white] (1, 0.9) -- (1, 0.1);
\draw[->, double, gray!50!white] (0.3, 0.3) -- (0.7, 0.7);
\draw[->, double, gray!50!white] (1.3, 0.3) -- (1.7, 0.7);
\end{tikzpicture}\hspace{10pt} \text {and} \hspace{10pt}
\begin{tikzpicture}[baseline = 12]
\filldraw[gray!50!white]
(0,0) circle [radius = 1pt]
(1,0) circle [radius = 1pt];
\filldraw
(1,1) circle [radius = 1pt]
(0,1) circle [radius = 1pt]
(2,1) circle [radius = 1pt]
(2,0) circle [radius = 1pt];
\draw[->, gray!50!white] (0.1,0) -- (0.9,0);
\draw[->, gray!50!white] (1.1,0) -- (1.9,0);
\draw[->] (0.1,1) -- (0.9,1);
\draw[->] (1.1,1) -- (1.9,1);
\draw[->] (2,0.9) -- (2,0.1);
\draw[->, gray!50!white] (0, 0.9) -- (0, 0.1);
\draw[->, gray!50!white] (1, 0.9) -- (1, 0.1);
\draw[->, double, gray!50!white] (0.3, 0.3) -- (0.7, 0.7);
\draw[->, double, gray!50!white] (1.3, 0.3) -- (1.7, 0.7);
\end{tikzpicture}
\]
have the same silhouette ``$\begin{tikzpicture}[baseline = 1, scale = 0.25] \draw (0,1) -- (2,1) -- (2,0);\end{tikzpicture}$''.
Similarly there are four cells of silhouette ``$\begin{tikzpicture}[baseline = 1, scale = 0.25] \draw (0,1) -- (1,1) -- (1,0) -- (2,0);\end{tikzpicture}$'' and four of silhouette ``$\begin{tikzpicture}[baseline = 1, scale = 0.25] \draw (0,1) -- (0,0) -- (2,0);\end{tikzpicture}$''.
There are two cells
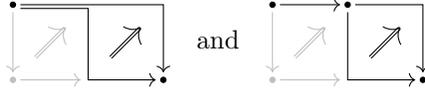
\[
\begin{tikzpicture}[baseline = 12]
\filldraw
(2,0) circle [radius = 1pt]
(0,1) circle [radius = 1pt];
\filldraw[gray!50!white]
(0,0) circle [radius = 1pt];
\draw[->, gray!50!white] (0.1,0) -- (0.9,0);
\draw[->] (0.1,1) -- (2,1) -- (2,0.1);
\draw[->] (0.1,0.95) -- (1,0.95) -- (1,0) -- (1.9,0);
\draw[->, gray!50!white] (0, 0.9) -- (0, 0.1);
\draw[->, double, gray!50!white] (0.3, 0.3) -- (0.7, 0.7);
\draw[->, double] (1.3, 0.3) -- (1.7, 0.7);
\end{tikzpicture} \hspace{10pt} \text {and} \hspace{10pt}
\begin{tikzpicture}[baseline = 12]
\filldraw
(2,0) circle [radius = 1pt]
(1,1) circle [radius = 1pt]
(0,1) circle [radius = 1pt];
\filldraw[gray!50!white]
(0,0) circle [radius = 1pt];
\draw[->, gray!50!white] (0.1,0) -- (0.9,0);
\draw[->] (1,0.9) -- (1,0) -- (1.9,0);
\draw[->] (0.1,1) -- (0.9,1);
\draw[->] (1.1,1) -- (2,1) -- (2,0.1);
\draw[->, gray!50!white] (0, 0.9) -- (0, 0.1);
\draw[->, double, gray!50!white] (0.3, 0.3) -- (0.7, 0.7);
\draw[->, double] (1.3, 0.3) -- (1.7, 0.7);
\end{tikzpicture}
\]
of silhouette ``$\begin{tikzpicture}[baseline = 1, scale = 0.25] \filldraw (0,1) -- (1,1) (1,1) -- (2,1) -- (2,0) -- (1,0) -- cycle;\end{tikzpicture}$'', and similarly for ``$\begin{tikzpicture}[baseline = 1, scale = 0.25] \filldraw (2,0) -- (3,0) (1,1) -- (2,1) -- (2,0) -- (1,0) -- cycle;\end{tikzpicture}$''.
Finally, the cells
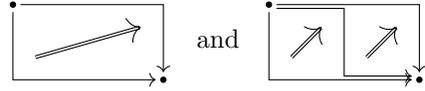
\[
\begin{tikzpicture}[baseline = 12]
\filldraw
(2,0) circle [radius = 1pt]
(0,1) circle [radius = 1pt];
\draw[->] (0.1,1) -- (2,1) -- (2,0.1);
\draw[->] (0,0.9) -- (0,0)-- (1.9,0);
\draw[->, double] (0.3, 0.3) -- (1.7, 0.7);
\end{tikzpicture} \hspace{10pt} \text {and} \hspace{10pt}
\begin{tikzpicture}[baseline = 12]
\filldraw
(2,0) circle [radius = 1pt]
(0,1) circle [radius = 1pt];
\draw[->] (0.1,1) -- (2,1) -- (2,0.1);
\draw[->] (0,0.9) -- (0,0)-- (1.9,0);
\draw[->] (0.1,0.95) -- (1,0.95) -- (1,0.05) -- (1.9,0.05);
\draw[->, double] (0.3, 0.3) -- (0.7, 0.7);
\draw[->, double] (1.3, 0.3) -- (1.7, 0.7);
\end{tikzpicture}
\]
have silhouette ``$\begin{tikzpicture}[baseline = 1, scale = 0.25] \filldraw (0,0) -- (2,0) -- (2,1) -- (0,1) -- cycle;\end{tikzpicture}$''.
We can associate a cut-point (= a point that disconnects the shape if removed) to each silhouette except for the last one as follows:
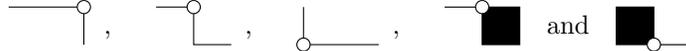
\[
\begin{tikzpicture}[baseline = 4, scale = 0.5]
\draw (0,1) -- (2,1) -- (2,0);
\draw[fill = white]
(2,1) circle [radius = 5pt];
\end{tikzpicture} \hspace{5pt},\hspace{15pt}
\begin{tikzpicture}[baseline = 4, scale = 0.5]
\draw (0,1) -- (1,1) -- (1,0) -- (2,0);
\draw[fill = white]
(1,1) circle [radius = 5pt];
\end{tikzpicture} \hspace{5pt},\hspace{15pt}
\begin{tikzpicture}[baseline = 4, scale = 0.5]
\draw (0,1) -- (0,0) -- (2,0);
\draw[fill = white]
(0,0) circle [radius = 5pt];
\end{tikzpicture} \hspace{5pt},\hspace{15pt}
\begin{tikzpicture}[baseline = 4, scale = 0.5]
\filldraw (0,1) -- (1,1) (1,1) -- (2,1) -- (2,0) -- (1,0) -- cycle;
\draw[fill = white]
(1,1) circle [radius = 5pt];
\end{tikzpicture} \hspace{10pt} \text {and} \hspace{10pt}
\begin{tikzpicture}[baseline = 4, scale = 0.5]
\filldraw (2,0) -- (3,0) (1,1) -- (2,1) -- (2,0) -- (1,0) -- cycle;
\draw[fill = white]
(2,0) circle [radius = 5pt];
\end{tikzpicture}
\]
Observe that the set of non-degenerate cells of these ``cuttable'' silhouettes can then be partitioned into pairs of the form $\bigl\{\phi, \phi \cdot \delta_h^{k_\phi}\bigr\}$ where the $k_\phi$-th vertex of $\phi$ is the cut-point associated to the silhouette of $\phi$.
We can glue such $\phi$ to $X$ along $\horn_h^{k_\phi}$ in increasing order of $\dim\phi$, and then glue the above $(1;2)$-cell of silhouette ``$\begin{tikzpicture}[baseline = 1, scale = 0.25] \filldraw (0,0) -- (2,0) -- (2,1) -- (0,1) -- cycle;\end{tikzpicture}$'' along $\horn_v^{1;1}[1;2]$.
This exhibits the inclusion $X \incl \cell[2;0,0] \otimes \cell[1;0]$ as a member of $\celll(\H_h \cup \H_v)$.

\subsection{Silhouettes and cut-points}\label{silhoette and cutpoints}
We will formalise the notions of \emph{silhouette} and \emph{cut-point} which were vaguely defined in the previous subsection.
Fix $\theta_1,\dots,\theta_a \in \cell$.
\begin{definition}
	A \emph{silhouette} $\sigma$ in $\tensor_a\bigl(\cell^{\theta_1},\dots,\cell^{\theta_a}\bigr)$ is a $(1;1)$-cell regarded as a pair of $(1;0)$-cells $\sigma = (\sigma_0,\sigma_1)$ where $\sigma_0 = \sigma \cdot \eta_v^0$ is the source and $\sigma_1 = \sigma \cdot \eta_v^1$ is the target.\footnote{We are making this distinction between a silhouette and a $(1;1)$-cell mainly so that \cref{cuttable silhouette,cut of silhouette,cuttable cell} do not cause ambiguity.}
	We write $\pre^\sigma_0$ and $\pre^\sigma_1$ for the underlying shuffles of $\sigma_0$ and $\sigma_1$ respectively.
\end{definition}

For example, the following picture depicts a silhouette in $\cell^{4;\zzero} \otimes \cell^{2;\zzero}$:
\begin{equation}\label{example silhouette}
\begin{tikzpicture}[baseline = 25]
\filldraw
(0,2) -- (1,2) -- (1,1) -- (4,1) -- (4,0) -- (3,0) -- (3,1) -- (0,1) -- cycle;

\filldraw[gray!50!white]
(2,2) circle [radius = 1pt]
(3,2) circle [radius = 1pt]
(4,2) circle [radius = 1pt]
(0,0) circle [radius = 1pt]
(1,0) circle [radius = 1pt];

\foreach \x in {1,2,3}
\draw[->, gray!50!white] (\x+0.1,2) -- (\x+0.9,2);
\foreach \x in {1,2,3}
\draw[->, gray!50!white] (\x + 1, 1.9) -- (\x + 1, 1.1);
\foreach \x in {1,2,3}
\draw[->, double, gray!50!white] (\x + 0.3, 1.3) -- (\x + 0.7, 1.7);

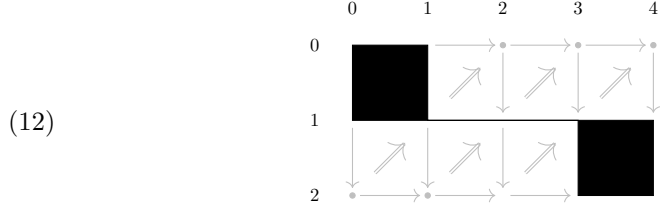
\foreach \x in {0,1,2}
\draw[->, gray!50!white] (\x,0.9) -- (\x , 0.1);
\foreach \x in {0,1,2}
\draw[->, gray!50!white] (\x + 0.1, 0) -- (\x + 0.9, 0);
\foreach \x in {0,1,2}
\draw[->, double, gray!50!white] (\x + 0.3, 0.3) -- (\x + 0.7, 0.7);

\foreach \x in {0,1,2,3,4}
\node[scale = 0.7] at (\x,2.5) {$\x$};
\foreach \x in {0,1,2}
\node[scale = 0.7] at (-0.5,2-\x) {$\x$};
\end{tikzpicture}
\end{equation}

For each $\ss, \tt \in \ttensor_a(\ttheta)$, we put a partial order on the set of silhouettes with endpoints $\ss,\tt$ so that $\sigma \le \tau$ holds if and only if
\[
\tau_0 \le \sigma_0 \le \sigma_1 \le \tau_1
\]
holds in the poset $\ttensor_a(\ttheta)(\ss,\tt)$.
This should be thought of as the containment relation between the silhouettes.

\begin{definition}\label{cuttable silhouette}
	Let $\sigma$ be a silhouette  with endpoints $\ss,\tt$.
	Then a \emph{cut-point} in $\sigma$ is an object $\xx$ with $\ss \neq \xx \neq \tt$ such that both $\sigma_0$ and $\sigma_1$ visit $\xx$.
	We call a silhouette \emph{cuttable} if it admits a cut-point.
\end{definition}
For example, the silhouette (\ref{example silhouette}) has cut-points $(1,1)$, $(2,1)$ and $(3,1)$.
The following proposition follows from \cref{visits}.
\begin{proposition}\label{cutpoint characterisation}
	Let $\sigma$ be a silhouette in $\ttensor_a(\ttheta)$ with endpoints $\ss,\tt$ and let $\xx \in \ttensor_a(\ttheta)$.
	Then $\xx$ is a cut-point in $\sigma$ if and only if:
	\begin{itemize}
		\item $s_i \le x_i \le t_i$ for each $i$ (which implies $S(\ss,\tt) = S(\ss,\xx) \cup S(\xx,\tt)$);
		\item $\ss \neq \xx \neq \tt$; and
		\item both $(i|k) \pre^\sigma_0 (j|\ell)$ and $(i|k) \pre^\sigma_1 (j|\ell)$ hold for any $(i|k)\in S(\ss,\xx)$ and $(j|\ell) \in S(\xx,\tt)$.
	\end{itemize}
\end{proposition}

\begin{definition}
	A cut-point $\xx$ in a silhouette $\sigma$ is \emph{right-angled} if for any $i$ with $s_i < x_i < t_i$, either:
	\begin{itemize}
		\item $(i|x_i+1)$ is not the immediate $\pre^\sigma_0$-successor of $(i|x_i)$; or
		\item $(i|x_i+1)$ is not the immediate $\pre^\sigma_1$-successor of $(i|x_i)$.
	\end{itemize}
\end{definition}

To continue our example (\ref{example silhouette}), the cut-point $(2,1)$ is not right-angled since $(1|3)$ is the immediate successor of $(1|2)$ with respect to both $\pre^\sigma_0$ and $\pre^\sigma_1$.
The other two cut-points $(1,1)$ and $(3,1)$ are right-angled.

\begin{definition}
	A silhouette in $\tensor_a\bigl(\cell^{\theta_1},\dots,\cell^{\theta_a}\bigr)$ is said to be \emph{non-linear} if it has endpoints $\ss$ and $\tt$ such that $s_i < t_i$ for at least two $i$'s.
\end{definition}

\begin{lemma}\label{right-angled cut-point}
	Let $\sigma$ be a non-linear, cuttable silhouette.
	Then $\sigma$ admits a right-angled cut-point.
\end{lemma}

\begin{proof}
	Let $\ss$ and $\tt$ denote the endpoints of $\sigma$.
	We will first treat the case where $\sigma_0$ and $\sigma_1$ visit exactly the same set of objects.
	Note that in this case any object that $\sigma_0$ visits is a cut-point in $\sigma$.
	By non-linearity, we must have $(j|\ell),(j'|\ell')\in S(\ss,\tt)$ with $j \neq j'$ such that $(j'|\ell')$ is the immediate $\pre^\sigma_0$-successor of $(j|\ell)$.
	Then the object $\xx$ defined by
	\[
	x_i = \min\bigl(\bigl\{(i|k) \in S(\ss,\tt): (i|k) \pre^\sigma_0 (j|\ell)\bigr\} \cup \{s_i\}\bigr)
	\]
	is a right-angled cut-point.
	
	In the other case, there must be a cut-point $\xx$ such that $\sigma_0$ visits a non-cut-point object $\yy$ with $\ss \neq \yy \neq \tt$ immediately before or immediately after $\xx$.
	Such $\xx$ then is necessarily right-angled.
\end{proof}
Note that for any silhouette $\sigma$, the set of cut-points in $\sigma$ admits a total order given by $\xx \le \yy$ if and only if $x_i \le y_i$ for each $i$.

\begin{definition}\label{cut of silhouette}
	If $\sigma$ is a non-linear, cuttable silhouette, then we write $\cut(\sigma)$ for the first right-angled cut-point in $\sigma$ (whose existence is guaranteed by \cref{right-angled cut-point}).
\end{definition}

\subsection{Silhouettes of cells}
\begin{definition}
	For any $(n;\qq)$-cell $\phi$ in $\tensor_a\bigl(\cell^{\theta_1},\dots,\cell^{\theta_a}\bigr)$, the \emph{silhouette of $\phi$} is
	\[
	\sil(\phi) \defeq \Bigl(\phi \cdot \bigl[\{0,n\};\{0\}, \dots, \{0\}\bigr],~\phi \cdot \bigl[\{0,n\};\{q_1\}, \dots, \{q_n\}\bigr]\Bigr).
	\]
\end{definition}
For example, if $\phi$ is the $(3;1,0,1)$-cell
\begin{equation}\label{example cell}
\begin{tikzpicture}[baseline = 25]
\filldraw
(0,2) circle [radius = 1pt]
(1,1) circle [radius = 1pt]
(3,1) circle [radius = 1pt]
(4,0) circle [radius = 1pt];
\filldraw[gray!50!white]
(2,2) circle [radius = 1pt]
(3,2) circle [radius = 1pt]
(4,2) circle [radius = 1pt]
(0,0) circle [radius = 1pt]
(1,0) circle [radius = 1pt]
(2,0) circle [radius = 1pt];

\draw[->] (0.1,2) -- (1,2) -- (1,1.1);
\draw[->] (0,1.9) -- (0,1) -- (0.9,1);
\draw[->] (1.1,1) -- (2.9,1);
\draw[->] (3.1,1) -- (4,1) -- (4,0.1);
\draw[->] (3,0.9) -- (3,0) -- (3.9,0);

\draw[->, double] (0.3,1.3) -- (0.7,1.7);
\draw[->, double] (2.3,0.3) -- (2.7,0.7);
\draw[->, double] (3.3,0.3) -- (3.7,0.7);

\foreach \x in {1,2,3}
\draw[->, gray!50!white] (\x+0.1,2) -- (\x+0.9,2);
\foreach \x in {1,2,3}
\draw[->, gray!50!white] (\x + 1, 1.9) -- (\x + 1, 1.1);
\foreach \x in {1,2,3}
\draw[->, double, gray!50!white] (\x + 0.3, 1.3) -- (\x + 0.7, 1.7);

\foreach \x in {0,1,2}
\draw[->, gray!50!white] (\x,0.9) -- (\x , 0.1);
\foreach \x in {0,1,2}
\draw[->, gray!50!white] (\x + 0.1, 0) -- (\x + 0.9, 0);
\foreach \x in {0,1,2}
\draw[->, double, gray!50!white] (\x + 0.3, 0.3) -- (\x + 0.7, 0.7);

\foreach \x in {0,1,2,3,4}
\node[scale = 0.7] at (\x,2.5) {$\x$};
\foreach \x in {0,1,2}
\node[scale = 0.7] at (-0.5,2-\x) {$\x$};
\end{tikzpicture}
\end{equation}
in $\cell^{4;\zzero} \otimes \cell^{2;\zzero}$ then $\sil(\phi)$ is the silhouette (\ref{example silhouette}).

The following proposition is straightforward to prove using \cref{special faces}.
\begin{proposition}\label{silhouette of face}
	Let $\phi$ be a non-degenerate cell in $\tensor_a\bigl(\cell^{\theta_1},\dots,\cell^{\theta_a}\bigr)$.
	Then a face of $\phi$ has the same silhouette as that of $\phi$ if and only if it is an inner face.
\end{proposition}

\begin{definition}\label{cuttable cell}
	We say a non-degenerate, non-linear cell $\phi : \nq \to\ttensor_a(\ttheta)$ is:
	\begin{itemize}
		\item \emph{$\sil$-cuttable} if $\sil(\phi)$ is cuttable;
		\item \emph{cuttable} if it is $\sil$-cuttable and moreover there is $k \in \nq$ such that $\phi(k) = \cut\bigl(\sil(\phi)\bigr)$; and
		\item \emph{$\sil$-uncuttable} if $\sil(\phi)$ is not cuttable.
	\end{itemize}
If $\phi$ is a cuttable cell, we write $\cut(\phi)$ for the necessarily unique $0 < k < n$ satisfying $\phi(k) = \cut\bigl(\sil(\phi)\bigr)$.
\end{definition}
\begin{remark}
	Note that \cref{cuttable cell} only concerns non-degenerate cells.
	Thus, whenever we speak of a ($\sil$-(un))cuttable cell, we are implicitly assuming that it is non-degenerate.
\end{remark}
\begin{proposition}\label{cuttable parent exists}
	Let $\chi$ be a $\sil$-cuttable cell in $\tensor_a\bigl(\cell^{\theta_1},\dots,\cell^{\theta_a}\bigr)$ that is not cuttable.
	Then there exists a unique cuttable cell $\phi$ such that $\chi$ is a $\cut(\phi)$-th horizontal face of $\phi$.
	
	Conversely, if $\phi$ is a cuttable $(n;\qq)$-cell and $\delta : [n-1;\pp] \to \nq$ is a $\cut(\phi)$-th horizontal face operator, then $\phi \cdot \delta$ is $\sil$-cuttable but not cuttable.
\end{proposition}
\begin{proof}
	The second part follows from \cref{silhouette of face}.
	We will prove the first part in the special case where $\theta_i = [n_i;\zzero]$ for each $i$ and $\chi$ is a $(1;q)$-cell.
	The general case can be treated similarly and is left to the reader.
	
	In this special case, $\chi$ is solely determined by its underlying shuffles $\prep$ on $S(\ss,\tt)$ where $\ss,\tt \in \ttensor_a(\ttheta)$ are the endpoints of $\chi$.
	Let $\xx = \cut(\sil(\chi))$ and suppose we are given $(i|k) \in S(\ss,\xx)$ and $(j|\ell) \in S(\xx,\tt)$.
	Then $\pre_0\;=\;\pre_0^{\sil(\chi)}$ and $\pre_q\;=\;\pre_1^{\sil(\chi)}$ by the definition of $\sil(\chi)$, hence we have $(i|k) \pre_0 (j|\ell)$ and $(i|k) \pre_q (j|\ell)$ by \cref{cutpoint characterisation}.
	Thus for any $0 \le p \le q$:
	\begin{itemize}
		\item if $i<j$ then we must have $(i|k) \pre_p (j|\ell)$ since $\pre_0\LHD\pre_p$;
		\item if $i>j$ then we must have $(i|k) \prep (j|\ell)$ for otherwise it contradicts our assumption that $\pre_p\LHD\pre_q$; and
		\item if $i=j$ then $(i|k)\prez(j|\ell)$ implies $k < \ell$ since $\prez$ is a shuffle, which in turn implies $(i|k) \prep (i|\ell)$ since $\prep$ is a shuffle.
	\end{itemize}
This shows that $(i|k) \prep (j|\ell)$ holds for any $(i|k) \in S(\ss,\xx)$, $(j|\ell) \in S(\xx,\tt)$ and $0 \le p \le q$.

Define two equivalence relations $\sim_1$, $\sim_2$ on the set $[q]$ so that:
\begin{itemize}
	\item $p\sim_1p'$ if and only if $\pre_p$ and $\pre_{p'}$ restrict to the same shuffle on $S(\ss,\xx)$; and
	\item $p\sim_2p'$ if and only if $\pre_p$ and $\pre_{p'}$ restrict to the same shuffle on $S(\xx,\tt)$.
\end{itemize}
Then the desired cuttable cell $\phi$ is the obvious $(2;q_1,q_2)$-cell where $[q_1] \cong [q]/\hspace{-3pt}\sim_1$ and $[q_2] \cong [q]/\hspace{-3pt}\sim_2$.
\end{proof}

\begin{definition}
	In the situation of \cref{cuttable parent exists}, we say $\phi$ is the \emph{cuttable parent} of $\chi$.
\end{definition}

\begin{lemma}\label{cuttable parent lemma}
	Let $f_i : X^i \to \cell^{\theta_i}$ be a monomorphism in $\hattheta$ for each $i$, and let $\chi$ be a $\sil$-cuttable cell in $\tensor_a\bigl(\cell^{\theta_1},\dots,\cell^{\theta_a}\bigr)$ that is not cuttable.
	Then $\chi$ is in the image of the monomorphism
	\[
	\tensor_a(f_1,\dots,f_a) : \tensor_a(X^1, \dots, X^a) \to \tensor_a\bigl(\cell^{\theta_1},\dots,\cell^{\theta_a}\bigr)
	\]
	if and only if the cuttable parent of $\chi$ is in the image.
\end{lemma}
\begin{proof}
	This follows from \cref{Gray mono}.
\end{proof}

\begin{remark}
	\cref{cuttable parent lemma} relies crucially on the fact that $\cut(\sigma)$ is right-angled.
	For example, if we had defined $\cut(\sil(\phi)) = (2,1)$ for the $(3;1,0,1)$-cell $\phi$ from (\ref{example cell}) then $\phi$ is in the image of
	\[
	\delta_h^2 \otimes \id : \cell^{3;\zzero} \otimes \cell^{2;\zzero} \to \cell^{4;\zzero} \otimes \cell^{2;\zzero}
	\]
	whereas its parent
	\[
	\begin{tikzpicture}[baseline = 25]
	\filldraw
	(0,2) circle [radius = 1pt]
	(1,1) circle [radius = 1pt]
	(2,1) circle [radius = 1pt]
	(3,1) circle [radius = 1pt]
	(4,0) circle [radius = 1pt];
	\filldraw[gray!50!white]
	(2,2) circle [radius = 1pt]
	(3,2) circle [radius = 1pt]
	(4,2) circle [radius = 1pt]
	(0,0) circle [radius = 1pt]
	(1,0) circle [radius = 1pt]
	(2,0) circle [radius = 1pt];
	
	\draw[->] (0.1,2) -- (1,2) -- (1,1.1);
	\draw[->] (0,1.9) -- (0,1) -- (0.9,1);
	\draw[->] (1.1,1) -- (1.9,1);
	\draw[->] (2.1,1) -- (2.9,1);
	\draw[->] (3.1,1) -- (4,1) -- (4,0.1);
	\draw[->] (3,0.9) -- (3,0) -- (3.9,0);
	
	\draw[->, double] (0.3,1.3) -- (0.7,1.7);
	\draw[->, double] (2.3,0.3) -- (2.7,0.7);
	\draw[->, double] (3.3,0.3) -- (3.7,0.7);
	
	\foreach \x in {1,2,3}
	\draw[->, gray!50!white] (\x+0.1,2) -- (\x+0.9,2);
	\foreach \x in {1,2,3}
	\draw[->, gray!50!white] (\x + 1, 1.9) -- (\x + 1, 1.1);
	\foreach \x in {1,2,3}
	\draw[->, double, gray!50!white] (\x + 0.3, 1.3) -- (\x + 0.7, 1.7);
	
	\foreach \x in {0,1,2}
	\draw[->, gray!50!white] (\x,0.9) -- (\x , 0.1);
	\foreach \x in {0,1,2}
	\draw[->, gray!50!white] (\x + 0.1, 0) -- (\x + 0.9, 0);
	\foreach \x in {0,1,2}
	\draw[->, double, gray!50!white] (\x + 0.3, 0.3) -- (\x + 0.7, 0.7);
	\end{tikzpicture}
	\]
	is not.
	On the other hand, our choice that $\cut(\sigma)$ be the first one among all right-angled cut-points is not essential.
	Any right-angled cut-point would suffice for our purposes, and we are choosing the first one purely for the sake of definiteness.
\end{remark}

\section{$\otimes_2$ is left Quillen}\label{section binary}
This section is devoted to proving the following theorem.
\begin{theorem}\label{binary tensor is left Quillen}
	The binary Gray tensor product functor $\otimes = \tensor_2$ is left Quillen.
\end{theorem}
\begin{proof}
By \cref{left Quillen characterisation}, it suffices to prove that all maps in $\I \hat \otimes\I$ are cofibrations and all maps in $\J \hat \otimes \I$ and $\I \hat \otimes \J$ are trivial cofibrations.
The first part is an instance of \cref{boundary}, and the second part follows from \cref{horizontal horn,vertical horn,vertical equivalence,horizontal equivalence} proved below (and their duals).
\end{proof}

\subsection{Inner horizontal horn inclusion $\hat \otimes$ boundary inclusion}
Fix objects $\mp,\nq \in \cell$ and $1 \le k \le m-1$.
The aim of this subsection is to prove the following lemma.
\begin{lemma}\label{horizontal horn}
	The map
	\[
	\bigl(\horn^k_h\mp \hookrightarrow \cell\mp\bigr) \hat \otimes \bigl(\partial\cell\nq \hookrightarrow \cell\nq\bigr)
	\]
	is in $\celll(\H_h \cup \H_v)$.
\end{lemma}
\begin{proof}
We will denote this map by $A \incl B$.
It is a monomorphism by \cref{celll,boundary} so we may regard $A$ as a cellular subset of $B = N\bigl(\mp \boxtimes\nq\bigr)$.
Since the case $\nq = [0]$ is trivial, we will assume $n \ge 1$.

Let $A' \subset B$ be the cellular subset generated by $A$ and the ($\sil$-)cuttable cells.
Note that any cell in $B \setminus A$ is non-linear.
Moreover, it follows from \cref{cuttable parent exists,cuttable parent lemma} that the set of non-degenerate cells in $A' \setminus A$ can be partitioned into subsets of the form
\[
\bigl\{\text {$\phi$ and all of its $\cut(\phi)$-th horizontal faces}\bigr\}
\]
where $\phi$ is a cuttable cell.
We prove that $A'$ may be obtained from $A$ by gluing the cuttable cells $\phi$ along the inner horn $\horn_h^{\cut(\phi)}$ in lexicographically increasing order of $\sil(\phi)$ and $\dim(\phi)$.
That is, given two cuttable cells $\chi$ and $\phi$, we glue $\chi$ before $\phi$ if:
\begin{itemize}
	\item $\sil(\chi) < \sil(\phi)$; or
	\item $\sil(\chi) = \sil(\phi)$ and $\dim(\chi) < \dim(\phi)$.
\end{itemize}
Fix a cuttable cell $\phi$ in $A' \setminus A$.
We must check that all hyperfaces of $\phi$ except for the $\cut(\phi)$-th horizontal ones are contained either in $A$ or in some cuttable $\chi$ satisfying one of the two conditions described above.
Indeed, all outer hyperfaces of $\phi$ have smaller silhouettes than $\phi$, and all inner hyperfaces $\chi$ of $\phi$ except for the $\cut(\phi)$-th horizontal ones are cuttable and satisfy $\sil(\chi) = \sil(\phi)$ and $\dim(\chi)<\dim(\phi)$.
Thus the inclusion $A \incl A'$ is in $\celll(\H_h)$.

Now consider a $\sil$-uncuttable cell $\phi$ in $B$ with endpoints $(0,0)$ and $(m,n)$ (which may or may not be contained in $A$).
Such $\phi$ is necessarily a $(1;r)$-cell for some $r \ge 1$.
Thus $\phi$ can be identified with a chain $f_0 < \dots < f_r$ in the hom-poset.
For each $0 \le s \le r$, we write $\pre_s$ for the underlying shuffle of $f_s$.

Let $k' \in [n]$ be the largest element such that $f_r$ visits $(k,k')$.
Note that we must have $k' < n$ for otherwise $(k,k') = (k,n)$ would be a cut-point in $\sil(\phi)$.
Define $s_\phi$ to be the largest element $s \in [r]$ such that
\[
(2|k'+1) \pres (1|k)
\]
holds; equivalently, $s_\phi$ is the largest $s$ such that $f_s$ does not visit $(k,k')$ (see \cref{picture horizontal horn}).
Such $s_\phi$ indeed exists for otherwise $(k,k')$ is a cut-point in $\sil(\phi)$.

\begin{figure}
	\[
	\begin{tikzpicture}[baseline = 40]
	
	\draw[line width = 2.5pt, gray!50!white] (2,3.3) -- (2,-0.3) (-0.3,2) -- (3.3,2);
	
	\filldraw[gray!50!white]
	(2,2) circle [radius = 3pt]
	(1,2) circle [radius = 3pt]
	(2,1) circle [radius = 3pt];
	
	\draw[gray!50!white]
	(2,2) .. controls (2.5,2) and (2.5,3) .. (3,3)
	(1,2) .. controls (1,1.3) and (0,1.2) .. (0,0.5)
	(2,1) .. controls (1.5,1) and (1.5,0) .. (1,0);
	
	\filldraw
	(0,3) circle [radius = 1pt]
	(3,0) circle [radius = 1pt];
	
	\draw[->] (0,2.9) -- (0,1) -- (2,1) -- (2,0) -- (2.9,0);
	\draw[->] (0.1,2.95) -- (0.95,2.95) -- (0.95,1.05) -- (2.05,1.05) -- (2.05,0.05) -- (2.9,0.05);
	\draw[->] (0.1,3) -- (1,3) -- (1,2) -- (3,2) -- (3,0.1);
	\draw[->] (0.1,3.05) -- (2,3.05) -- (2,2.05) -- (3.05,2.05) -- (3.05,0.1);
	
	\draw[->, double] (0.3,1.8) -- (0.7,2.2);
	\draw[->, double] (2.3,1.3) -- (2.7,1.7);
	\draw[->, double] (1.3,2.3) -- (1.7,2.7);
	
	\node[scale = 0.7] at (2,3.5) {$k$};
	\node[scale = 0.7] at (-0.5,2) {$k'$};
	\node[scale = 0.7] at (-0.2,1.5) {$f_0$};
	\node[scale = 0.7] at (1.2,1.5) {$f_{s_\phi}$};
	\node[scale = 0.7] at (1.5,3.2) {$f_r$};
	
	\node[scale = 0.7] at (0,0.35) {$(x,k')$};
	\node[scale = 0.7] at (3.3,3) {$(k,k')$};
	\node[scale = 0.7] at (0.7,0) {$(k,y)$};
	\end{tikzpicture}
	\hspace{10pt}\leadsto\hspace{10pt}
	\begin{tikzpicture}[baseline = 40]
	\draw[line width = 2.5pt, gray!50!white] (0,3) -- (1,3) -- (1,2) (2,1) -- (2,0) -- (3,0);
	
	\draw[gray!50!white]
	(1,2.5) .. controls (1.5,2.5) and (1.5,3) .. (2,3)
	(2,0.5) .. controls (1.5,0.5) and (0.5,0.7) .. (0.5,1);
	
	\filldraw
	(0,3) circle [radius = 1pt]
	(3,0) circle [radius = 1pt];
	
	\draw[->] (0.1,3) -- (1,3) -- (1,2) -- (2,2) -- (2,0) -- (2.9,0);
	
	\node[scale = 0.7] at (2.2,1.5) {$g$};
	\node[scale = 0.7] at (3.1,3) {$\pre_{s_\phi}$-initial segment};
	\node[scale = 0.7] at (0.5,1.2) {$\pre_{s_\phi}$-terminal segment};
	\end{tikzpicture}
	\]
	\caption{Example: $\mp = \nq = [3;\zzero]$ and $k = 2$.}\label{picture horizontal horn}
\end{figure}

We will construct the ``best approximation'' $g$ to $f_{s_\phi}$ that visits $(k,k')$.
Let $x \in [m]$ be the maximum such that $f_{s_\phi}$ visits $(x,k')$ and let $y \in [n]$ be the minimum such that $f_{s_\phi}$ visits $(k,y)$.
Then we must have $0 \le x < k$ and $k' < y \le n$.
Now let $g : (0,0) \to (m,n)$ be the 1-cell determined by the following conditions:
\begin{itemize}
	\item each of the projections $\mp \leftarrow \mp \boxtimes \nq \to \nq$ agrees on $g$ and $f_{s_\phi}$; and
	\item the underlying shuffle $\pre$ of $g$ is obtained by patching together the following (see \cref{picture horizontal horn}):
	\begin{itemize}
		\item the $\pre_{s_\phi}$-initial segment up to just before $(2|k'+1)$;
		\item the $\pre_{s_\phi}$-terminal segment starting just after $(1|k)$; and
		\item the interval
		\[
		(1|x+1) \pre (1|x+2) \pre \dots \pre (1|k) \pre (2|k'+1) \pre (2|k'+2) \pre \dots \pre (2|y).
		\]
	\end{itemize}
\end{itemize}
These data indeed specify a unique 1-cell by \cref{Gray pullback bar}, and moreover it is easy to see that $f_{s_\phi} < g \le f_{s_\phi+1}$ holds in the hom-poset.
Consider the following condition on $\phi$:
\begin{itemize}
    \item [(hh)] $f_{s_\phi+1} = g$.
\end{itemize}
(Here ``hh'' stands for ``horizontal horn''.)

It is obvious that the set of $\sil$-uncuttable cells in $B$ with endpoints $(0,0)$ and $(m,n)$ can be partitioned into pairs of the form $\bigl\{\phi,\phi\cdot\delta_v^{1;s_\phi+1}\bigr\}$ where $\phi$ satisfies (hh).
We now show that this pairing restricts to one on the set of non-degenerate cells in $B \setminus A'$.

\begin{claim*}
	A cell $\phi$ satisfying (hh) is contained in $A'$ (or equivalently in $A$) if and only if $\phi \cdot\delta_v^{1;s_\phi+1}$ is contained in $A'$ (or equivalently in $A$).
\end{claim*}
\begin{proof}[Proof of the claim]
    The ``only if'' part is obvious.
	For the ``if'' part, we first treat the case where $\phi \cdot\delta_v^{1;s_\phi+1}$ is contained in the cellular subset $\cell\mp \otimes \partial\cell\nq$.
	For most hyperface maps $\delta$ into $\nq$, if $\phi \cdot\delta_v^{1;s_\phi+1}$ is contained in the image of the induced map $\id \otimes \delta$ then we can apply \cref{Gray mono} twice to deduce that $\phi$ is in the image of same map, using the fact that the 1-cell $g$ constructed above is ``almost'' $f_{s_\phi}$.
	The only non-trivial sub-case is when $\phi \cdot\delta_v^{1;s_\phi+1}$ is in the image of
	\[
	\id \otimes \delta_h^{k';\langle\alpha,\alpha'\rangle} : \cell\mp \otimes \cell[n-1;\qqd] \to \cell\mp \otimes \cell\nq
	\]
	for some $(q_k,q_{k+1})$-shuffle $\langle\alpha,\alpha'\rangle$,
	which we can rule out (again using \cref{Gray mono}) since
	\[
	(2|k')\pre_r (1|k+1) \pre_r (2|k'+1)
	\]
	holds by our definition of $k'$.
	
	Next, suppose that $\phi \cdot\delta_v^{1;s_\phi+1}$ is contained in $\horn_h^k\mp \otimes \cell\nq$.
	Note that, by construction of $g$, if $g$ visits two distinct objects $(\ell,\ell')$ and $(\ell,\ell'')$ for some $\ell$ but $f_{s_\phi}$ does not then we must have $\ell = k$.
	Since all of the generating hyperfaces in $\horn_h^k\mp \otimes \cell\nq$ contain the object $k$, it follows from \cref{Gray mono} that $\phi$ is contained in $\horn_h^k\mp \otimes \cell\nq$.
\end{proof}

We prove that $B$ may be obtained from $A'$ by gluing those $(1;r)$-cells $\phi$ in $B \setminus A'$ satisfying (hh) along the horn $\horn_v^{1;s_\phi+1}[1;r]$ in lexicographically increasing order of $\sil(\phi)$, $\dim(\phi)$ and  $s_\phi$.
Note that if $\phi$ satisfies (hh) then $s_\phi+1 \neq r$ since $f_{s_\phi+1}$ visits $(k,k'+1)$ while $f_r$ does not.
Also we have $s_\phi+1 \neq 0$ since $s_\phi \ge 0$, thus this horn $\horn_v^{1;s_\phi+1}[1;r]$ is inner.

We must check that, for any such $\phi$, all of its hyperfaces except for the $(1;s_\phi+1)$-th vertical one are contained either in $A'$ or in some cell $\chi$ satisfying (hh) such that:
\begin{itemize}
    \item $\sil(\chi) < \sil(\phi)$;
    \item $\sil(\chi) = \sil(\phi)$ and $\dim(\chi) < \dim(\phi)$; or
    \item $\sil(\chi) = \sil(\phi)$, $\dim(\chi) = \dim(\phi)$ and $s_\chi < s_\phi$.
\end{itemize}
Indeed:
\begin{itemize}
	\item $\phi \cdot \delta_v^{1;0}$ and $\phi \cdot \delta_v^{1;r}$ have smaller silhouettes than $\sil(\phi)$;
	\item if $s_\phi \neq 0$ then $\phi \cdot \delta_v^{1;s_\phi}$:
	\begin{itemize}
		\item is contained in $A'$;
		\item satisfies (hh); or
		\item is of the form $\phi \cdot \delta_v^{1;s_\phi} = \chi \cdot \delta_v^{1;s_\chi+1}$ for some cell $\chi$ satisfying (hh) which necessarily has $\dim \chi = \dim \phi$ and $s_\chi = s_\phi - 1$; and
	\end{itemize}
	\item for any other value of $s$, the hyperface $\phi \cdot \delta_v^{1;s}$:
	\begin{itemize}
		\item is contained in $A'$; or
		\item satisfies (hh) and has dimension strictly smaller than $\dim(\phi)$.
	\end{itemize}
\end{itemize}
This completes the proof.
\end{proof}

\subsection{Inner vertical horn inclusion $\hat \otimes$ boundary inclusion}
Fix $\mp,\nq \in \cell$, $1 \le k \le m$ and $1 \le i \le p_k-1$.
In this subsection, we will prove the following lemma.
\begin{lemma}\label{vertical horn}
	The map
	\[
	\bigl(\horn_v^{k;i}\mp \hookrightarrow \cell\mp\bigr) \hat \otimes \bigl(\partial\cell\nq \hookrightarrow \cell\nq\bigr)
	\]
	is in $\celll(\H_h \cup \H_v)$.
\end{lemma}
\begin{proof}
We will regard this map as a cellular subset inclusion and denote it as $A \incl B$.
Since the case $\nq = [0]$ is trivial, we will assume $n \ge 1$.

Similarly to the proof of \cref{horizontal horn}, we can show that gluing the cuttable cells $\phi$ to $A$ along the inner horn $\horn^{\cut(\phi)}$ in lexicographically increasing order of $\sil(\phi)$ and $\dim(\phi)$ yields the cellular subset $A' \subset \cell\mp \otimes \cell\nq$ generated by $A$ and the ($\sil$-)cuttable cells.

\begin{temp}
    For any 1-cell $f : (0,0) \to (m,n)$ in $\mp \boxtimes \nq$ and for any $1 \le \ell \le m$, the composite
\[
\begin{tikzcd}
{[1;0]}
\arrow [r, "f"] &
{\mp \boxtimes \nq}
\arrow [r, "\pi_1"] &
{\mp}
\end{tikzcd}
\]
corresponds to a cellular operator $[\{0,m\};\aalpha] : [1;0] \to \mp$.
We will write $f \rest \ell$ for $\alpha_\ell(0) \in [p_\ell]$.
\end{temp}
Let $\phi$ be a non-degenerate $(1;r)$-cell in $B \setminus A'$ (which necessarily has endpoints $(0,0)$ and $(m,n)$) corresponding to 1-cells $f_0,\dots,f_r : (0,0) \to (m,n)$ with underlying shuffles $\pre_0,\dots,\pre_r$ respectively.
Let
\[
s_\phi \defeq \max\bigl\{s : f_s \rest k = i-1\bigr\}.
\]
To see that this is well-defined, observe that if $f_s \rest k \neq i-1$ for all $0 \le s \le r$ then $\phi$ is contained in the image of $\delta_v^{k;i-1}\otimes\id$ which contradicts our assumption that $\phi$ is not in $A'$.

We construct the ``best approximation'' $g$ to $f_{s_\phi}$ with $g \rest k = i$.
Let $g : (0,0) \to (m,n)$ be the 1-cell determined by the following conditions:
\begin{itemize}
	\item the second projection $\mp \boxtimes \nq \to \nq$ agrees on $f_{s_\phi}$ and $g$;
	\item $f_{s_\phi}$ and $g$ have the same underlying shuffle; and
	\item
	$g \rest \ell = \left\{\begin{array}{cl}
	i & \text{if $\ell = k$,}\\
	f_{s_\phi} \rest \ell & \text{otherwise.}
	\end{array}\right.$
\end{itemize}
Then clearly we have $f_{s_\phi} < g \le f_{s_\phi+1}$.
Consider the following condition on $\phi$:
\begin{itemize}
    \item [(vh)]$f_{s_\phi+1} = g$.
\end{itemize}
Note that if $\phi$ satisfies (vh) then $s_\phi+1 \neq r$ since $f_{s_\phi+1}\rest k = i$ while $f_r \rest k = p_k$.
Also we have $s_\phi+1 \neq 0$ since $s_\phi \ge 0$.

It can be easily checked using \cref{Gray mono} that the set of non-degenerate cells in $B \setminus A'$ can be partitioned into pairs of the form $\bigl\{\phi,\phi \cdot \delta_v^{1;s_\phi+1}\bigr\}$
where $\phi$ is a $(1;r)$-cell satisfying (vh).
We claim that $B$ may be obtained from $A'$ by gluing such $\phi$ along the inner horn $\horn_v^{1;s_\phi+1}[1;r]$ in lexicographically increasing order of $\sil(\phi)$, $\dim(\phi)$ and $s_\phi$.
Indeed, for any such $\phi$:
\begin{itemize}
	\item $\phi \cdot \delta_v^{1;0}$ and $\phi \cdot \delta_v^{1;r}$ have smaller silhouettes than $\sil(\phi)$;
	\item if $s_\phi \neq 0$ then $\phi \cdot \delta_v^{1;s_\phi}$ is:
	\begin{itemize}
		\item contained in $A'$; or
		\item of the form $\phi \cdot \delta_v^{1;s_\phi} = \chi \cdot \delta_v^{1;s_\chi+1}$ for some cell $\chi$ satisfying (vh) which necessarily has $\sil(\chi) = \sil(\phi)$, $\dim (\chi) = \dim (\phi)$ and $s_\chi = s_\phi - 1$; and
	\end{itemize}
	\item for any other value of $s$, the hyperface $\phi \cdot \delta_v^{1;s}$:
\begin{itemize}
	\item is contained in $A'$; or
	\item satisfies (vh) and has dimension strictly smaller than $\dim(\phi)$.
\end{itemize}
\end{itemize}
This completes the proof.
\end{proof}

\subsection{Vertical equivalence extension $\hat \otimes$ boundary inclusion}
Any unlabelled map of the form $\cell[1;0] \incl \cell[1;J]$ in this subsection is assumed to be $[\id;e]$, which looks like:
\[
\left\{
\begin{tikzpicture}[baseline = -3]
\filldraw
(0,0) circle [radius = 1pt]
(1,0) circle [radius = 1pt];
\draw[->] (0.1,0.1) .. controls (0.4,0.4) and (0.6,0.4) .. (0.9,0.1);
\draw[->, gray!50!white] (0.1,-0.1) .. controls (0.4,-0.4) and (0.6,-0.4) .. (0.9,-0.1);
\node[rotate = -90, gray!50!white] at (0.5,0) {$\cong$};
\node[scale = 0.7] at (0.5, 0.5) {$\lozenge$};
\node[scale = 0.7, gray!50!white] at (0.5, -0.5) {$\blacklozenge$};
\end{tikzpicture}\right\}
\incl
\left\{
\begin{tikzpicture}[baseline = -3]
\filldraw
(0,0) circle [radius = 1pt]
(1,0) circle [radius = 1pt];
\draw[->] (0.1,0.1) .. controls (0.4,0.4) and (0.6,0.4) .. (0.9,0.1);
\draw[->] (0.1,-0.1) .. controls (0.4,-0.4) and (0.6,-0.4) .. (0.9,-0.1);
\node[rotate = -90] at (0.5,0) {$\cong$};
\node[scale = 0.7] at (0.5, 0.5) {$\lozenge$};
\node[scale = 0.7] at (0.5, -0.5) {$\blacklozenge$};
\end{tikzpicture}\right\}
\]
Fix $\nq \in \cell$.
We will prove the following lemma in this subsection.
\begin{lemma}\label{vertical equivalence}
	The map
	\[
	\bigl(\cell[1;0] \incl \cell[1;J]\bigr) \hat \otimes \bigl(\partial\cell\nq \incl \cell\nq\bigr)
	\]
	is a trivial cofibration.
\end{lemma}

We will first analyse the Gray tensor product $\cell[1;J] \otimes \cell\nq$.
Let $\Jv$ denote the 2-category whose nerve is $\cell[1;J]$.
More precisely, its object set is $\{0,1\}$ and its hom-categories are
\[
\begin{split}
\Jv(0,0) &= [0],\\
\Jv(1,1) &= [0],\\
\Jv(0,1) &= \{\lozenge \cong \blacklozenge\},\\
\Jv(1,0) &= \varnothing.
\end{split}
\]
The following lemma can be proved in essentially the same way as \cref{Gray pullback bar}.
\begin{lemma}\label{vertical equivalence ttensor}
	The square
	\[
	\begin{tikzcd}[row sep = large]
	{\Jv \boxtimes \nq}
	\arrow [r]
	\arrow [d, "{\langle\pi_1,\pi_2\rangle}", swap] &
	{[1;0] \boxtimes \nq}
	\arrow [d, "{\langle\pi_1,\pi_2\rangle}"] \\
	{\Jv \times \nq}
	\arrow [r] &
	{[1;0] \times \nq}
	\end{tikzcd}
	\]
	is a pullback in $\twoCat$, where the horizontal maps are induced by the unique identity-on-objects 2-functor $\Jv \to [1;0]$.
\end{lemma}

For any 2-categories $\A$ and $\B$, a $\zeta$-cell in the Gray tensor product $N\A \otimes N\B$ is represented (non-uniquely) by $\theta_1,\theta_2 \in \cell$ together with three 2-functors
\[
\begin{gathered}
\phi : \zeta \to \theta_1 \boxtimes \theta_2,\\
\chi_1 : \theta_1 \to \A,\\
\chi_2 : \theta_2 \to \B.
\end{gathered}
\]
Such 2-functors may be combined into a single 2-functor
\[
\begin{tikzcd}
\zeta
\arrow [r, "\phi"] &
\theta_1 \boxtimes \theta_2
\arrow [r, "\chi_1 \boxtimes \chi_2"] &
\A \boxtimes \B
\end{tikzcd}
\]
which corresponds to a $\zeta$-cell in $N(\A\boxtimes\B)$.
This defines a comparison map
\[
N\A \otimes N\B \to N(\A\boxtimes\B).
\]
\begin{lemma}\label{vertical equivalence tensor}
	The comparison map
	\[
	\cell[1;J] \otimes \cell\nq \to N\bigl(\Jv \boxtimes \nq\bigr)
	\]
	is invertible.
\end{lemma}
\begin{proof}
	Observe that $\cell[1;J]$ may be obtained from $\cell[0] \amalg \cell[0]$ by gluing two copies of $\cell[1;r]$ along the boundary for each $r \ge 0$ in increasing order of $r$.
	Since the functor $\otimes$ preserves colimits in each variable, it follows that $\cell[1;J] \otimes \cell\nq$ may be obtained from $\cell\nq \amalg \cell\nq$ by gluing two copies of $\cell[1;r] \otimes \cell\nq$ along $\partial \cell[1;r] \otimes \cell\nq$ for each $r \ge 0$.
	This presentation of $\cell[1;J] \otimes \cell\nq$ can be made more explicit using \cref{Gray mono}, and comparing it to \cref{vertical equivalence ttensor} yields the desired result.
\end{proof}

\begin{proof}[Proof of \cref{vertical equivalence}]
    We will regard the map
    \[
	\bigl(\cell[1;0] \incl \cell[1;J]\bigr) \hat \otimes \bigl(\partial\cell\nq \incl \cell\nq\bigr)
	\]
	as a cellular subset inclusion and denote it by $A \incl B$.
	Let
	\[
	P : \cell[1;J] \otimes \cell\nq \to \cell[1;0] \otimes \cell\nq
	\]
	be the map induced by the unique map $\cell[1;J] \to \cell[1;0]$ that is bijective on 0-cells.
	Given any cell $\phi$ in $B$, we will write $\sil(\phi)$ for $\sil(P(\phi))$ and say $\phi$ is \emph{non-linear} or ($\sil$-(\emph{un}))\emph{cuttable} if $P(\phi)$ is so.
	If $\phi$ is a non-linear cuttable cell, we write $\cut(\phi)$ for $\cut(P(\phi))$.
	
	Let $A' \subset B$ be the cellular subset generated by $A$ and the ($\sil$-)cuttable cells.
	Then one can prove, using the obvious analogues of \cref{right-angled cut-point,cuttable parent exists}, that the inclusion $A \incl A'$ is in $\celll(\H_h)$.
	
	Now consider a non-degenerate cells $\phi$ in $B \setminus A'$, which is necessarily a $\sil$-uncuttable $(1;r)$-cell for some $r$.
	Via \cref{vertical equivalence ttensor,vertical equivalence tensor}, we may regard $\phi$ as consisting of:
	\begin{itemize}
		\item a chain $f_0\le\dots\le f_r$ of 1-cells $(0,0) \to (1,n)$ in $[1;0] \boxtimes \nq$; and
		\item a sequence $(\epsilon_0,\dots,\epsilon_r)$ in $\{\lozenge,\blacklozenge\}$.
	\end{itemize}
Since $\phi$ is not contained in $A'$, we must have $\epsilon_s = \blacklozenge$ for at least one $s$.
Thus
\[
s_\phi \defeq \max\{s : \epsilon_s = \blacklozenge\}
\]
is well-defined.
Consider the following condition on $\phi$:
\begin{itemize}
    \item [(ve)] $s_\phi < r$ and $f_{s_\phi+1} = f_{s_\phi}$.
\end{itemize}
Note that, since we are assuming $\phi$ to be ($\sil$-uncuttable and hence) non-degenerate, (ve) implies $\epsilon_{s_\phi+1} = \lozenge$.
It also implies $r \ge 2$ for otherwise $\phi$ would be $\sil$-cuttable.
	
Clearly the set of non-degenerate cells in $B \setminus A'$ can be partitioned into pairs of the form $\bigl\{\phi, \phi \cdot \delta_v^{1;s_\phi+1}\bigr\}$ where $\phi$ is a $(1;r)$-cell satisfying (ve).
We claim that $B$ may be obtained from $A'$ by gluing such $\phi$ along the horn $\horn_v^{1;s_\phi+1}[1;r]$ in lexicographically increasing order of $\dim(\phi)$ and $s_\phi$.
Indeed, for any $(1;r)$-cell $\phi$ satisfying (ve):
\begin{itemize}
	\item $\phi \cdot \delta_v^{1;s_\phi}$:
	\begin{itemize}
	\item is contained in $A'$;
	\item is degenerate;
	\item satisfies (ve); or
	\item is of the form $\phi \cdot \delta_v^{1;s_\phi} = \chi \cdot \delta_v^{1;s_\chi+1}$ for some cell $\chi$ satisfying (ve) which necessarily has $\dim(\chi) = \dim(\phi)$ and $s_\chi < s_\phi$; and
	\end{itemize}
	\item $\phi \cdot \delta_v^{1;s}$ where $s_\phi \neq s \neq s_\phi+1$ is:
	\begin{itemize}
		\item contained in $A'$; or 
		\item a (possibly trivial) degeneracy of some cell $\chi$ satisfying (ve) which necessarily has $\dim(\chi) < \dim(\phi)$.
	\end{itemize}
\end{itemize}
The horn $\horn_v^{1;s_\phi+1}[1;r]$ is not necessarily inner since $s_\phi+1$ may be equal to $r$.
Nevertheless, in that case the outer horn is a \emph{special} one in the sense that the composite map
\[
\begin{tikzcd}
{\horn_v^{1;s_\phi+1}[1;r]} = {\horn_v^{1;r}[1;r]}
\arrow [r, hook] &
{\cell[1;r]}
\arrow [r, "\phi"] &
{\cell[1;J] \otimes \cell\nq}
\end{tikzcd}
\]
can be extended to one from $\hornt_v^{1;r}[1;r]$ as defined in \cref{section special}.
Moreover, the images of the cells in $\hornt_v^{1;r}[1;r] \setminus \horn_v^{1;r}[1;r]$ are cuttable and hence contained in $A'$.
Since the special outer horn inclusions $\hornt^{1;r}_v[1;r] \incl \cellt^{1;r}[1;r]$ are trivial cofibrations by the dual of \cref{special vertical}, we can deduce that the inclusion $A' \incl B$ is a trivial cofibration.
This completes the proof.
\end{proof}

\subsection{Horizontal equivalence extension $\hat \otimes$ boundary inclusion}
Recall that the monomorphism $e : \cell[0] \incl J$ is (isomorphic to) the nerve of the inclusion
\[
\{\lozenge\} \incl \{\lozenge\cong\blacklozenge\} = \Jh
\]
into the chaotic category $\Jh$ on two objects.
We will prove the following lemma in this subsection.
\begin{lemma}\label{horizontal equivalence}
	The map
	\[
	\bigl(\cell[0] \overset{e}{\incl} J\bigr) \hat \otimes \bigl(\partial \cell\nq \incl \cell\nq\bigr)
	\]
	is a trivial cofibration for any $\nq \in \cell$.
\end{lemma}

First we analyse the Gray tensor product $J \otimes \cell[1;q]$ for $q \ge 0$.
Consider the (2-categorical) Gray tensor product $\Jh \boxtimes [1;q]$, whose object set is $\{\lozenge,\blacklozenge\} \times \{0,1\}$.
\begin{lemma}\label{horizontal equivalence ttensor}
	For any $\star,\star' \in \{\lozenge,\blacklozenge\}$ and for any $k,\ell \in \{0,1\}$, the hom-category of $\Jh \boxtimes [1;q]$ is given by
	\[
	\bigl(\Jh \boxtimes [1;q]\bigr)\bigl((\star,k),(\star',\ell)\bigr) \cong
	\left\{\begin{array}{cl}
	{[0]} & \text{if $k = \ell$,}\\
	{\{\cdot \cong \cdot\} \times [q]} & \text{if $k=0$ and $\ell = 1$,}\\
	\varnothing & \text{if $k = 1$ and $\ell = 0$.}
	\end{array}\right.
	\]
\end{lemma} 
\begin{proof}
	The proof is similar to that of \cref{Gray pullback bar}.
	The inverse to a generating 2-cell of the form
	\[
	\begin{tikzpicture}
	\node at (0,1.5) {$(\lozenge,0)$};
	\node at (0,0) {$(\lozenge,1)$};
	\node at (2,1.5) {$(\blacklozenge,0)$};
	\node at (2,0) {$(\blacklozenge,1)$};
	
	\draw[->] (0.5,0) -- (1.5,0);
	\draw[->] (0,1.2) -- node[scale = 0.7, left]{$(\lozenge,p)$} (0,0.3);
	\draw[->] (0.5,1.5) -- (1.5,1.5);
	\draw[->] (2,1.2) -- node[scale = 0.7, right]{$(\blacklozenge,p)$} (2,0.3);
	
	\draw[->, double] (0.6,0.45) -- (1.4,1.05);
	
	\node[scale = 0.8] at (0.8,0.9) {$\gamma$};
	\end{tikzpicture}
	\]
	is obtained by whiskering the 2-cell
	\[
	\begin{tikzpicture}
	\node at (0,1.5) {$(\lozenge,0)$};
	\node at (0,0) {$(\lozenge,1)$};
	\node at (2,1.5) {$(\blacklozenge,0)$};
	\node at (2,0) {$(\blacklozenge,1)$};
	
	\draw[->] (1.5,0) -- (0.5,0);
	\draw[->] (0,1.2) -- node[scale = 0.7, left]{$(\lozenge,p)$} (0,0.3);
	\draw[->] (1.5,1.5) -- (0.5,1.5);
	\draw[->] (2,1.2) -- node[scale = 0.7, right]{$(\blacklozenge,p)$} (2,0.3);
	
	\draw[->, double] (1.4,0.45) -- (0.6,1.05);
	
	\node[scale = 0.8] at (1.2,0.9) {$\gamma$};
	\end{tikzpicture}
	\]
	with the obvious 1-cells.
\end{proof}

\begin{lemma}\label{horizontal equivalence tensor}
	The comparison map
	\[
	J \otimes \cell[1;q] \to N(\Jh \boxtimes [1;q])
	\]
	is invertible for any $q \ge 0$.
\end{lemma}
\begin{proof}
For the sake of simplicity, we will only prove that the comparison map acts bijectively on the $(1;r)$-cells with endpoints $(\lozenge,0)$ and $(\blacklozenge,1)$; the general case can be treated similarly.
By \cref{horizontal equivalence ttensor}, such $(1;r)$-cells correspond to those sequences in $\{L,R\} \times [q]$ of length $r+1$ that are increasing in the second coordinate; here $L$ and $R$ correspond to 1-cells of the form
\[
\begin{tikzpicture}[baseline = 18]
\node at (0,1.5) {$(\lozenge,0)$};
\node at (0,0) {$(\lozenge,1)$};
\node at (2,0) {$(\blacklozenge,1)$};

\draw[->] (0.5,0) -- (1.5,0);
\draw[->] (0,1.2) -- (0,0.3);
\end{tikzpicture}
\hspace{15pt}\text{and}\hspace{15pt}
\begin{tikzpicture}[baseline = 18]
\node at (0,1.5) {$(\lozenge,0)$};
\node at (2,1.5) {$(\blacklozenge,0)$};
\node at (2,0) {$(\blacklozenge,1)$};

\draw[->] (0.5,1.5) -- (1.5,1.5);
\draw[->] (2,1.2) -- (2,0.3);
\end{tikzpicture}
\]
respectively.

Observe that $J$ has precisely two non-degenerate $(d;\zzero)$-cells $e_\lozenge^d, e_\blacklozenge^d$ for each $d \ge 0$ where $e_\star^d \cdot [\{0\}] = \star$ for $\star \in \{\lozenge,\blacklozenge\}$.
Thus $J \otimes \cell[1;q]$ may be obtained from $\varnothing$ by gluing two copies of $\cell[d;\zzero] \otimes \cell[1;q]$ along $\partial\cell[d;\zzero] \otimes \cell[1;q]$ in increasing order of $d$.
By \cref{Gray mono}, a sequence of 2-cells
\[
f_0 \Rightarrow \dots \Rightarrow f_r : (0,0) \to (d,1)
\]
in $[d;\zzero] \boxtimes [1;q]$ corresponds to a $(1;r)$-cell in
\[
\bigl(\cell[d;\zzero] \otimes \cell[1;q]\bigr) \setminus \bigl(\partial\cell[d;\zzero] \otimes \cell[1;q]\bigr)
\]
if and only if, for each $1 \le c \le d-1$, there exists $0 \le s \le r$ such that $f_s$ visits both $(c,0)$ and $(c,1)$.

There are four kinds of such cells, depending on whether $f_0$ visits $(0,1)$ and whether $f_r$ visits $(d,0)$ (see \cref{four kinds}).
The images of these cells under
\[
e^d_\star \otimes \id : \cell[d;\zzero]\otimes\cell[1;q] \to J \otimes \cell[1;q]
\]
have endpoints $(\lozenge,0)$ and $(\blacklozenge,1)$ if and only if $\star =\lozenge$ and $d = 2d'+1$ is odd.
Moreover, in such case $e^{2d'+1}_\lozenge \otimes \id$ sends these cells bijectively to those sequences in $\{L,R\} \times [q]$ for which ``$RL$'' appears exactly $d'$ times in their first projections (which are sequences in $\{L,R\}$).
This completes the proof.
\begin{figure}
\[
\begin{gathered}
\begin{tikzpicture}
\filldraw
(0,1) circle [radius = 1pt]
(5,0) circle [radius = 1pt];

\draw
(0,0.9) -- (0,-0.05) -- (2.5,-0.05)
(0.1,0.95) -- (1,0.95) -- (1,0) -- (2.5,0)
(0.1,1) -- (2,1) -- (2,0.05) -- (2.5,0.05)
(0.1,1.05) -- (2.5,1.05)
(0.1,1.1) -- (2.5,1.1);

\draw[->] (3.5,1.1) -- (5,1.1) -- (5,0.1);
\draw[->] (3.5,1.05) -- (4,1.05) -- (4,0.05) -- (4.9,0.05);
\draw[->] (3.5,0) -- (4.9,0);
\draw[->] (3.5,-0.05) -- (4.9,-0.05);

\foreach \x in {0,1,4}
\draw[->, double] (\x+0.3,0.3) -- (\x+0.7,0.7);

\foreach \x in {2.9,3,3.1}
\filldraw (\x,0.5) circle [radius = 0.2pt];
\end{tikzpicture}\\
\begin{tikzpicture}
\filldraw
(0,1) circle [radius = 1pt]
(5,0) circle [radius = 1pt];

\draw[gray!50!white]
(0,0.9) -- (0,-0.05) -- (2.5,-0.05);
\draw
(0.1,0.95) -- (1,0.95) -- (1,0) -- (2.5,0)
(0.1,1) -- (2,1) -- (2,0.05) -- (2.5,0.05)
(0.1,1.05) -- (2.5,1.05)
(0.1,1.1) -- (2.5,1.1);

\draw[->] (3.5,1.1) -- (5,1.1) -- (5,0.1);
\draw[->] (3.5,1.05) -- (4,1.05) -- (4,0.05) -- (4.9,0.05);
\draw[->] (3.5,0) -- (4.9,0);
\draw[->, gray!50!white] (3.5,-0.05) -- (4.9,-0.05);

\foreach \x in {1,4}
\draw[->, double] (\x+0.3,0.3) -- (\x+0.7,0.7);

\draw[->, double, gray!50!white] (0.3,0.3) -- (0.7,0.7);

\foreach \x in {2.9,3,3.1}
\filldraw (\x,0.5) circle [radius = 0.2pt];
\end{tikzpicture}\\
\begin{tikzpicture}
\filldraw
(0,1) circle [radius = 1pt]
(5,0) circle [radius = 1pt];

\draw
(0,0.9) -- (0,-0.05) -- (2.5,-0.05)
(0.1,0.95) -- (1,0.95) -- (1,0) -- (2.5,0)
(0.1,1) -- (2,1) -- (2,0.05) -- (2.5,0.05)
(0.1,1.05) -- (2.5,1.05);
\draw[gray!50!white]
(0.1,1.1) -- (2.5,1.1);

\draw[->, gray!50!white] (3.5,1.1) -- (5,1.1) -- (5,0.1);
\draw[->] (3.5,1.05) -- (4,1.05) -- (4,0.05) -- (4.9,0.05);
\draw[->] (3.5,0) -- (4.9,0);
\draw[->] (3.5,-0.05) -- (4.9,-0.05);

\foreach \x in {0,1}
\draw[->, double] (\x+0.3,0.3) -- (\x+0.7,0.7);

\draw[->, double, gray!50!white] (4.3,0.3) -- (4.7,0.7);

\foreach \x in {2.9,3,3.1}
\filldraw (\x,0.5) circle [radius = 0.2pt];
\end{tikzpicture}\\
\begin{tikzpicture}
\filldraw
(0,1) circle [radius = 1pt]
(5,0) circle [radius = 1pt];

\draw[gray!50!white]
(0,0.9) -- (0,-0.05) -- (2.5,-0.05);
\draw
(0.1,0.95) -- (1,0.95) -- (1,0) -- (2.5,0)
(0.1,1) -- (2,1) -- (2,0.05) -- (2.5,0.05)
(0.1,1.05) -- (2.5,1.05);
\draw[gray!50!white]
(0.1,1.1) -- (2.5,1.1);

\draw[->, gray!50!white] (3.5,1.1) -- (5,1.1) -- (5,0.1);
\draw[->] (3.5,1.05) -- (4,1.05) -- (4,0.05) -- (4.9,0.05);
\draw[->] (3.5,0) -- (4.9,0);
\draw[->, gray!50!white] (3.5,-0.05) -- (4.9,-0.05);

\foreach \x in {0,4}
\draw[->, double, gray!50!white] (\x+0.3,0.3) -- (\x+0.7,0.7);

\draw[->, double] (1.3,0.3) -- (1.7,0.7);

\foreach \x in {2.9,3,3.1}
\filldraw (\x,0.5) circle [radius = 0.2pt];
\end{tikzpicture}
\end{gathered}
\]
\caption{Cells in $\cell[d;\zzero] \otimes \cell[1;q]$}\label{four kinds}
\end{figure}
\end{proof}

\begin{remark}
	\cref{horizontal equivalence ttensor} can be generalised in the obvious way to general $\nq \in \cell$ (in place of $[1;q]$), but \cref{horizontal equivalence tensor} is no longer true if we replace $[1;q]$ by $\nq$ with $n \ge 2$.
	For example, consider the $(1;2)$-cell
	
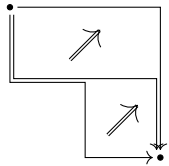
\[
	\begin{tikzpicture}
	\filldraw[gray!50!white]
	(0,0) circle [radius = 1pt];
	
	\draw[->, gray!50!white] (0.1,0) -- (0.9,0);
	\draw[->, gray!50!white] (0,0.9) -- (0,0.1);
	
	\draw[->, double, gray!50!white] (0.3,0.3) -- (0.7,0.7);
	
	\filldraw
	(0,2) circle [radius = 1pt]
	(2,0) circle [radius = 1pt];
	
	\draw[->] (0.1,2) -- (2,2) -- (2,0.1);
	\draw[->] (0.05,1.9) -- (0.05,1.05) -- (1.95,1.05) -- (1.95,0.1);
	\draw[->] (0,1.9) -- (0,1) -- (1,1) -- (1,0) -- (1.9,0);
	
	\draw[->, double] (1.3,0.3) -- (1.7,0.7);
	\draw[->, double] (0.8,1.3) -- (1.2,1.7);
	\end{tikzpicture}
	\]
	in $\cell[2;\zzero] \otimes \cell[2;\zzero]$.
	For each $\star \in \{\lozenge,\blacklozenge\}$, the image of this cell under $e^2_\star \otimes \id$ is an example of a non-degenerate cell in $J \otimes \cell[2;\zzero]$ that is sent to a degenerate one by the comparison map
	\[
	J \otimes \cell[2;\zzero] \to N\bigl(\Jh \boxtimes [2;\zzero]\bigr).
	\]
	
	In general, the cellular set $J \otimes \cell\nq$ does not seem to admit a simple description.
	Therefore the rest of our proof of \cref{horizontal equivalence} is not direct combinatorics, and instead it formalises the following idea.
	
	Informally speaking, the Gray tensor product $J \otimes X$ should represent invertible oplax natural transformations between $X$-shaped diagrams.
	But the invertibility implies that such transformations are in fact pseudo-natural.
	Thus the pseudo-variant of the Gray tensor product, which is modelled by the cartesian product $J \times X$ in this context, should be equivalent to $J \otimes X$.
	Therefore \cref{horizontal equivalence} should follow from the corresponding result for the cartesian product, which we already have since Ara's model structure is cartesian.
\end{remark}

We will make use of the following result of Campbell.

\begin{theorem}[{\cite[Theorem 10.11]{Campbell:nerve}}]\label{biequivalence}
	A 2-functor $\B \to \C$ is a biequivalence if and only if its nerve $N\B \to N\C$ is a weak equivalence of cellular sets.
\end{theorem}

\begin{proof}[Proof of \cref{horizontal equivalence}]
Fix $\nq \in \cell$, and let $\D$ be the full subcategory of the category of elements of $\spine\nq$ spanned by the non-degenerate cells.
Then $\D$ has an obvious Reedy category structure in which every map is degree-raising.
Since $\D$ has no degree-lowering maps, the diagonal functor $\hattheta \to \bigl[\D,\hattheta\bigr]$ is trivially right Quillen.
Now both composites
\[
\begin{tikzcd}[column sep = large, row sep = tiny]
\D
\arrow [r] &
\cell
\arrow [r, "\text{Yoneda}"] &
\hattheta
\arrow [r, "{J \otimes (-)}"] &
\hattheta \\
\D
\arrow [r] &
\cell
\arrow [r, "\text{Yoneda}"] &
\hattheta
\arrow [r, "{J \times (-)}"] &
\hattheta
\end{tikzcd}
\]
(where $\D \to \cell$ is the canonical projection) can be easily checked to be Reedy cofibrant by direct calculation.
Moreover, there is a natural transformation between them whose components are given by
\[
J \otimes \cell[0] \cong J \cong J \times \cell[0]
\]
for objects of degree $0$, and
\[
J \otimes \cell[1;q] \cong N\bigl(\Jh \boxtimes [1;q]\bigr) \to N\bigl(\Jh \times [1;q]\bigr) \cong J \times \cell[1;q]
\]
for objects of degree $q+1$ (with $q \in \{0,1\}$), where the middle map is the nerve of the obvious 2-functor.
It is easy to check that $\Jh \boxtimes [1;q] \to \Jh \times [1;q]$ is a biequivalence, hence its nerve is a weak equivalence by \cref{biequivalence}.
Thus by taking the colimit, we can conclude that $J \otimes \spine\nq \to J \times \spine\nq$ is a weak equivalence.
This map fits into the following commutative square:
\[
\begin{tikzcd}[row sep = large]
{J \otimes \spine\nq}
\arrow [r, "{\langle\pi_1,\pi_2\rangle}"]
\arrow [d, hook] &
{J \times \spine\nq}
\arrow [d, hook] \\
{J \otimes \cell\nq}
\arrow [r, "{\langle\pi_1,\pi_2\rangle}"] &
{J \times \cell\nq}
\end{tikzcd}
\]
The right vertical map is a trivial cofibration because Ara's model structure is cartesian \cite[Corollary 8.5]{Ara:nqcat}.
The left vertical map is also a trivial cofibration since the spine inclusion is in $\celll(\H_h \cup \H_v)$ as proved in \cite[Lemma 3.1]{Maehara:horns}, and
\[
J \otimes (\H_h \cup \H_v) \subset \celll(\H_u \cup \H_v)
\]
holds by \cref{horizontal horn,vertical horn}.
Therefore $\langle\pi_1,\pi_2\rangle : J \otimes \cell\nq \to J \times \cell\nq$ is a weak equivalence by the 2-out-of-3 property.

Finally, we prove the statement of the lemma by induction on $\dim\nq$.
The base case $\nq = [0]$ is trivial.
For the inductive step, consider the following commutative diagram:
\[
\begin{tikzcd}[column sep = tiny, row sep = large]
{\cell[0] \otimes \cell\nq}
\arrow [r, "{\langle\pi_1,\pi_2\rangle}"]
\arrow [d, hook] &
{\cell[0] \times \cell\nq}
\arrow [dd, hook, "e \times \id" right] \\
\bigl(\cell[0] \otimes \cell\nq\bigr) \cup \bigl(J \otimes \partial\cell\nq\bigr)
\arrow [d, hook] & \\
{J \otimes \cell\nq}
\arrow [r, "{\langle\pi_1,\pi_2\rangle}"] &
{J \times \cell\nq}
\end{tikzcd}
\]
Here the upper horizontal map is an isomorphism, the lower horizontal map is a weak equivalence as we have just proved, and the right vertical map is a trivial cofibration since Ara's model structure is cartesian.
Moreover, the upper left vertical map can be obtained by composing pushouts of maps of the form
\[
\bigl(\cell[0] \overset{e}{\incl} J\bigr) \hat \otimes \bigl(\partial \cell\mp \incl \cell\mp\bigr)
\]
with $\dim\mp < \dim\nq$, hence it is a trivial cofibration by the inductive hypothesis.
Thus the desired result follows by the 2-out-of-3 property.
\end{proof}

\section{Monoidal structure up to homotopy}\label{section assoc}
In this section, we will prove that the Gray tensor product forms part of a ``homotopical'' monoidal structure on $\hattheta$ in a suitable sense.
Let us first illustrate why it is not a genuine monoidal structure, or more specifically, how it fails to be associative up to isomorphism.
One would expect the Gray tensor product of three copies of $\cell[1;0]$ to ``be'' the commutative cube:
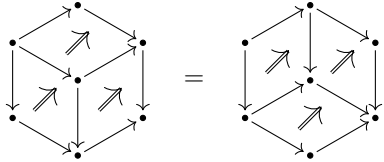
\[
\begin{tikzpicture}[baseline = -2]
\filldraw
(150:1) circle [radius = 1pt]
(90:1) circle [radius = 1pt]
(30:1) circle [radius = 1pt]
(-30:1) circle [radius = 1pt]
(-90:1) circle [radius = 1pt]
(-150:1) circle [radius = 1pt]
(0,0) circle [radius = 1pt];

\draw[->] (150:1) + (30:0.1) --+ (30:0.9);
\draw[->] (90:1) + (-30:0.1) --+ (-30:0.9);
\draw[->] (30:1) + (-90:0.1) --+ (-90:0.9);
\draw[->] (150:1) + (-90:0.1) --+ (-90:0.9);
\draw[->] (-150:1) + (-30:0.1) --+ (-30:0.9);
\draw[->] (-90:1) + (30:0.1) --+ (30:0.9);

\draw[->] (150:1) + (-30:0.1) --+ (-30:0.9);
\draw[->] (0:0) + (-90:0.1) --+ (-90:0.9);
\draw[->] (0:0) + (30:0.1) --+ (30:0.9);

\draw[->,double] (-150:0.5) + (-0.15,-0.15) --+ (0.15,0.15);
\draw[->,double] (90:0.5) + (-0.15,-0.15) --+ (0.15,0.15);
\draw[->,double] (-30:0.5) + (-0.15,-0.15) --+ (0.15,0.15);
\end{tikzpicture}
\hspace{10pt} = \hspace{10pt}
\begin{tikzpicture}[baseline = -2]
\filldraw
(150:1) circle [radius = 1pt]
(90:1) circle [radius = 1pt]
(30:1) circle [radius = 1pt]
(-30:1) circle [radius = 1pt]
(-90:1) circle [radius = 1pt]
(-150:1) circle [radius = 1pt]
(0,0) circle [radius = 1pt];

\draw[->] (150:1) + (30:0.1) --+ (30:0.9);
\draw[->] (90:1) + (-30:0.1) --+ (-30:0.9);
\draw[->] (30:1) + (-90:0.1) --+ (-90:0.9);
\draw[->] (150:1) + (-90:0.1) --+ (-90:0.9);
\draw[->] (-150:1) + (-30:0.1) --+ (-30:0.9);
\draw[->] (-90:1) + (30:0.1) --+ (30:0.9);

\draw[->] (0:0) + (-30:0.1) --+ (-30:0.9);
\draw[->] (90:1) + (-90:0.1) --+ (-90:0.9);
\draw[->] (-150:1) + (30:0.1) --+ (30:0.9);

\draw[->,double] (150:0.5) + (-0.15,-0.15) --+ (0.15,0.15);
\draw[->,double] (-90:0.5) + (-0.15,-0.15) --+ (0.15,0.15);
\draw[->,double] (30:0.5) + (-0.15,-0.15) --+ (0.15,0.15);
\end{tikzpicture}
\]
Indeed, the ``total'' tensor product $B = \tensor_3\bigl(\cell[1;0],\cell[1;0],\cell[1;0]\bigr)$ is by definition the nerve of this 2-category.
Now consider the nested tensor product
\[
A = \tensor_2\Bigl(\tensor_2\bigl(\cell[1;0],\cell[1;0]\bigr), \cell[1;0]\Bigr).
\]
The binary tensor product $\tensor_2\bigl(\cell[1;0],\cell[1;0]\bigr)$ is the nerve of the 2-category
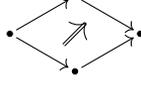
\[
\begin{tikzpicture}[baseline = -2]
\filldraw
(150:1) circle [radius = 1pt]
(90:1) circle [radius = 1pt]
(30:1) circle [radius = 1pt]
(0,0) circle [radius = 1pt];

\draw[->] (150:1) + (30:0.1) --+ (30:0.9);
\draw[->] (90:1) + (-30:0.1) --+ (-30:0.9);

\draw[->] (150:1) + (-30:0.1) --+ (-30:0.9);
\draw[->] (0:0) + (30:0.1) --+ (30:0.9);

\draw[->,double] (90:0.5) + (-0.15,-0.15) --+ (0.15,0.15);
\end{tikzpicture}
\]
which therefore has the following maximal non-degenerate cells:
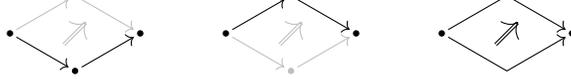
\[
\begin{tikzpicture}[baseline = -2]
\filldraw
(150:1) circle [radius = 1pt]
(0:0) circle [radius = 1pt]
(30:1) circle [radius = 1pt];
\filldraw[gray!50!white]
(90:1) circle [radius = 1pt];

\draw[->, gray!50!white] (150:1) + (30:0.1) --+ (30:0.9);
\draw[->, gray!50!white] (90:1) + (-30:0.1) --+ (-30:0.9);

\draw[->] (150:1) + (-30:0.1) --+ (-30:0.9);
\draw[->] (0:0) + (30:0.1) --+ (30:0.9);

\draw[->,double, gray!50!white] (90:0.5) + (-0.15,-0.15) --+ (0.15,0.15);
\end{tikzpicture}
\hspace{30pt}
\begin{tikzpicture}[baseline = -2]
\filldraw
(150:1) circle [radius = 1pt]
(90:1) circle [radius = 1pt]
(30:1) circle [radius = 1pt];
\filldraw[gray!50!white]
(0,0) circle [radius = 1pt];

\draw[->] (150:1) + (30:0.1) --+ (30:0.9);
\draw[->] (90:1) + (-30:0.1) --+ (-30:0.9);

\draw[->, gray!50!white] (150:1) + (-30:0.1) --+ (-30:0.9);
\draw[->, gray!50!white] (0:0) + (30:0.1) --+ (30:0.9);

\draw[->,double, gray!50!white] (90:0.5) + (-0.15,-0.15) --+ (0.15,0.15);
\end{tikzpicture}
\hspace{30pt}
\begin{tikzpicture}[baseline = -2]
\filldraw
(150:1) circle [radius = 1pt]
(30:1) circle [radius = 1pt];

\draw[->] (150:1) + (30:0.1) --++ (30:1) --+ (-30:0.9);
\draw[->] (150:1) + (-30:0.1) --++ (-30:1) --+ (30:0.9);

\draw[->, double] (90:0.5) + (-0.15,-0.15) --+ (0.15,0.15);
\end{tikzpicture}
\]
Thus $A$ may be obtained by pasting together the nerves of two copies of $[2;\zzero] \boxtimes [1;0]$ and one copy of $[1;1] \boxtimes [1;0]$ appropriately.
In fact, $A$ turns out to be (isomorphic to) the cellular subset of $B$ generated by the nerves of the following sub-2-categories:
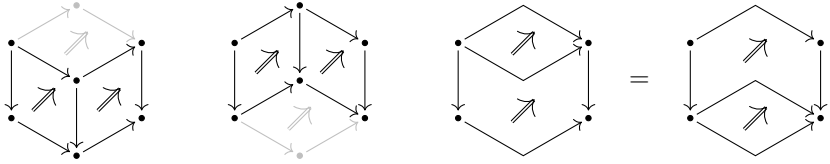
\[
\begin{tikzpicture}[baseline = -2]
\filldraw
(150:1) circle [radius = 1pt]
(30:1) circle [radius = 1pt]
(-30:1) circle [radius = 1pt]
(-90:1) circle [radius = 1pt]
(-150:1) circle [radius = 1pt]
(0,0) circle [radius = 1pt];
\filldraw[gray!50!white]
(90:1) circle [radius = 1pt];

\draw[->, gray!50!white] (150:1) + (30:0.1) --+ (30:0.9);
\draw[->, gray!50!white] (90:1) + (-30:0.1) --+ (-30:0.9);
\draw[->] (30:1) + (-90:0.1) --+ (-90:0.9);
\draw[->] (150:1) + (-90:0.1) --+ (-90:0.9);
\draw[->] (-150:1) + (-30:0.1) --+ (-30:0.9);
\draw[->] (-90:1) + (30:0.1) --+ (30:0.9);

\draw[->] (150:1) + (-30:0.1) --+ (-30:0.9);
\draw[->] (0:0) + (-90:0.1) --+ (-90:0.9);
\draw[->] (0:0) + (30:0.1) --+ (30:0.9);

\draw[->,double] (-150:0.5) + (-0.15,-0.15) --+ (0.15,0.15);
\draw[->,double, gray!50!white] (90:0.5) + (-0.15,-0.15) --+ (0.15,0.15);
\draw[->,double] (-30:0.5) + (-0.15,-0.15) --+ (0.15,0.15);
\end{tikzpicture}
\hspace{30pt}
\begin{tikzpicture}[baseline = -2]
\filldraw
(150:1) circle [radius = 1pt]
(90:1) circle [radius = 1pt]
(30:1) circle [radius = 1pt]
(-30:1) circle [radius = 1pt]
(0:0) circle [radius = 1pt]
(-150:1) circle [radius = 1pt];
\filldraw[gray!50!white]
(-90:1) circle [radius = 1pt];

\draw[->] (150:1) + (30:0.1) --+ (30:0.9);
\draw[->] (90:1) + (-30:0.1) --+ (-30:0.9);
\draw[->] (30:1) + (-90:0.1) --+ (-90:0.9);
\draw[->] (150:1) + (-90:0.1) --+ (-90:0.9);
\draw[->, gray!50!white] (-150:1) + (-30:0.1) --+ (-30:0.9);
\draw[->, gray!50!white] (-90:1) + (30:0.1) --+ (30:0.9);

\draw[->] (0:0) + (-30:0.1) --+ (-30:0.9);
\draw[->] (90:1) + (-90:0.1) --+ (-90:0.9);
\draw[->] (-150:1) + (30:0.1) --+ (30:0.9);

\draw[->,double] (150:0.5) + (-0.15,-0.15) --+ (0.15,0.15);
\draw[->,double, gray!50!white] (-90:0.5) + (-0.15,-0.15) --+ (0.15,0.15);
\draw[->,double] (30:0.5) + (-0.15,-0.15) --+ (0.15,0.15);
\end{tikzpicture}
\hspace{30pt}
\begin{tikzpicture}[baseline = -2]
\filldraw
(150:1) circle [radius = 1pt]
(30:1) circle [radius = 1pt]
(-30:1) circle [radius = 1pt]
(-150:1) circle [radius = 1pt];

\draw[->] (150:1) + (30:0.1) --++ (30:1) --+ (-30:0.9);
\draw[->] (150:1) + (-30:0.1) --++ (-30:1) --+ (30:0.9);
\draw[->] (150:1) + (-90:0.1) --+ (-90:0.9);
\draw[->] (-150:1) + (-30:0.1) -- (-90:1) --+ (30:0.9);
\draw[->] (30:1) + (-90:0.1) --+ (-90:0.9);

\draw[->,double] (-90:0.4) + (-0.15,-0.15) --+ (0.15,0.15);
\draw[->,double] (90:0.5) + (-0.15,-0.15) --+ (0.15,0.15);
\end{tikzpicture}
\hspace{10pt}=\hspace{10pt}
\begin{tikzpicture}[baseline = -2]
\filldraw
(150:1) circle [radius = 1pt]
(30:1) circle [radius = 1pt]
(-30:1) circle [radius = 1pt]
(-150:1) circle [radius = 1pt];

\draw[->] (150:1) + (30:0.1) --++ (30:1) --+ (-30:0.9);
\draw[->] (-150:1) + (30:0.1) --++ (30:1) --+ (-30:0.9);
\draw[->] (150:1) + (-90:0.1) --+ (-90:0.9);
\draw[->] (-150:1) + (-30:0.1) -- (-90:1) --+ (30:0.9);
\draw[->] (30:1) + (-90:0.1) --+ (-90:0.9);

\draw[->,double] (-90:0.5) + (-0.15,-0.15) --+ (0.15,0.15);
\draw[->,double] (90:0.4) + (-0.15,-0.15) --+ (0.15,0.15);
\end{tikzpicture}
\]
Informally speaking, a cell in $B$ is contained in $B \setminus A$ if and only if it remembers both of:
\begin{itemize}
	\item the decomposability of a 1-cell of shape
	$\begin{tikzpicture}[baseline = 12, scale = 0.7]
	\draw[->] (150:1) + (30:0.1) --++ (30:1) --+ (-30:0.9);
	\end{tikzpicture}$ or
	$\begin{tikzpicture}[baseline = 3,scale = 0.7]
	\draw[->] (150:1) + (-30:0.1) --++ (-30:1) --+ (30:0.9);
	\end{tikzpicture}$; and
	\item the existence of the top or bottom face of the cube.
\end{itemize}
For example, the following $(1;1)$-cell is not contained in $A$:
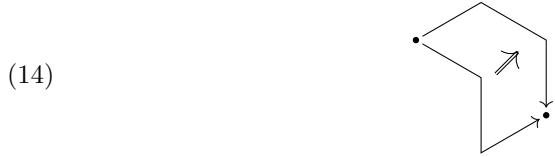
\begin{equation}\label{impurity example}
\begin{tikzpicture}[baseline = -2]
\filldraw
(150:1) circle [radius = 1pt]
(-30:1) circle [radius = 1pt];

\draw[->] (150:1) + (30:0.1) --++ (30:1) --++ (-30:1) --++ (-90:0.9);
\draw[->] (150:1) + (-30:0.1) --++ (-30:1) --++ (-90:1) --++ (30:0.9);

\draw[->,double] (30:0.4) + (-0.15,-0.15) --+ (0.15,0.15);

\filldraw[gray!50!white]
(-150:1) circle [radius = 1pt];

\draw[->, gray!50!white] (-150:1) + (-30:0.1) --+ (-30:0.9);
\draw[->, gray!50!white] (150:1) + (-90:0.1) --+ (-90:0.9);

\draw[->,double, gray!50!white] (-150:0.5) + (-0.15,-0.15) --+ (0.15,0.15);
\end{tikzpicture}
\end{equation}
Here we can ``see'' the top face of the cube, and moreover the vertical segment in the lower path ``remembers'' that $\begin{tikzpicture}[baseline = 3,scale = 0.7]
\draw[->] (150:1) + (-30:0.1) --++ (-30:1) --+ (30:0.9);
\end{tikzpicture}$ is decomposable.
This geometric intuition is formalised in \cref{mu}.
There is a similar description of the other nested tensor product \[
A' = \tensor_2\Bigl(\cell[1;0], \tensor_2\bigl(\cell[1;0],\cell[1;0]\bigr)\Bigr)
\]
as a cellular subset of $B$, and it is easy to see that $A \ncong A'$.

Now, although $A$ and $A'$ are not isomorphic to each other, they both admit an inclusion into $B$.
In general, we always have a comparison map from a nested tensor product to the corresponding total tensor product.
The functors $\tensor_a$ together with these comparison maps form a \emph{normal lax monoidal structure} on the category $\hattheta$ (\cref{lax monoidal}).
Moreover the (relative version of the) comparison maps are trivial cofibrations (\cref{hat mu}), hence the Gray tensor product is associative up to homotopy.

\subsection{Lax monoidal structure}
\begin{definition}
	A \emph{lax monoidal structure} on a category $\C$ consists of:
	\begin{itemize}
		\item a functor $\scalebox{1.2}{$\odot$}_a : \C^a \to \C$ for each $a \in \NN$;
		\item a natural transformation $\iota : \id_{\C} \to \scalebox{1.2}{$\odot$}_1$; and
		\item a natural transformation
		\[
		\mu_{b_1, \dots, b_a} : \scalebox{1.2}{$\odot$}_a\left(\scalebox{1.2}{$\odot$}_{b_1},\dots,\scalebox{1.2}{$\odot$}_{b_a}\right) \to \scalebox{1.2}{$\odot$}_{b_1+\dots+b_a}
		\]
		for each $a,b_1,\dots,b_a \in \NN$
	\end{itemize}
such that the following diagrams commute:
\[
\begin{tikzcd}[row sep = large]
\scalebox{1.2}{$\odot$}_a \arrow [r, "\iota\scalebox{1.2}{$\odot$}_a"] \arrow [dr, swap, "\id"] & \scalebox{1.2}{$\odot$}_1(\scalebox{1.2}{$\odot$}_a) \arrow [d, "\mu_a"] & \scalebox{1.2}{$\odot$}_a(\scalebox{1.2}{$\odot$}_1, \dots, \scalebox{1.2}{$\odot$}_1) \arrow [d, "\mu_{1,\dots,1}", swap] & \scalebox{1.2}{$\odot$}_a \arrow [l, swap, "{\scalebox{1.2}{$\odot$}_a(\iota,\dots,\iota)}"] \arrow [dl, "\id"]\\
& \scalebox{1.2}{$\odot$}_a & \scalebox{1.2}{$\odot$}_a &
\end{tikzcd}
\]
\[
\begin{tikzpicture}[scale = 1.5]
\node at (0,0) {$\scalebox{1.2}{$\odot$}_{c_{11} + \dots + c_{ab_a}}$};
\node at (-1,3) {$\scalebox{1.2}{$\odot$}_a \bigl(\scalebox{1.2}{$\odot$}_{b_1}(\scalebox{1.2}{$\odot$}_{c_{11}}, \dots, \scalebox{1.2}{$\odot$}_{c_{1b_1}}), \dots, \scalebox{1.2}{$\odot$}_{b_a}(\scalebox{1.2}{$\odot$}_{c_{a1}},\dots,\scalebox{1.2}{$\odot$}_{c_{ab_a}})\bigr)$};
\node at (-2,1) {$\scalebox{1.2}{$\odot$}_{b_1 + \dots + b_a}(\scalebox{1.2}{$\odot$}_{c_{11}}, \dots, \scalebox{1.2}{$\odot$}_{c_{ab_a}})$};
\node at (1.5,2) {$\scalebox{1.2}{$\odot$}_a(\scalebox{1.2}{$\odot$}_{c_{11}+\dots+c_{1b_1}}, \dots, \scalebox{1.2}{$\odot$}_{c_{a1}+\dots+c_{ab_a}})$};

\draw[->] (-0.6,2.8)--(0.6,2.2);
\draw[->] (-1.2,2.6)--(-1.8,1.4);
\draw[->] (-1.6,0.8)--(-0.4,0.2);
\draw[->] (0.8,1.6)--(0.2,0.4);

\node[scale = 0.8, fill = white] at (-1,0.5) {$\mu_{c_{11}, \dots, c_{ab_a}}$};
\node[scale = 0.8, fill = white] at (1,1) {$\mu_{c_{11} + \dots + c_{1b_1}, \dots, c_{a1} + \dots + c_{ab_a}}$};
\node[scale = 0.8, fill = white] at (-1.5,2) {$\mu_{b_1, \dots, b_a}$};
\node[scale = 0.8, fill = white] at (1,2.5) {$\scalebox{1.2}{$\odot$}_a(\mu_{c_{11},\dots,c_{1b_1}}, \dots, \mu_{c_{a1},\dots,c_{ab_a}})$};
\end{tikzpicture}
\]
Such a lax monoidal structure is called \emph{normal} if $\iota$ is invertible.
\end{definition}

\begin{remark}
	A lax monoidal structure on $\C$ is equivalently a lax algebra structure on $\C$ for the 2-monad on $\Cat$ whose strict algebras are the strict monoidal categories.
	It gives rise to a multicategory/coloured operad whose objects are those in $\C$ and whose $a$-ary maps $X_1,\dots,X_a \to Y$ are maps $Y \to \scalebox{1.2}{$\odot$}_a(X_1,\dots,X_a)$ in $\C$.
	However, mapping \emph{out} of tensor products does not yield a multicategory.
\end{remark}

\begin{proposition}\label{lax monoidal}
	The Gray tensor product functors $\tensor_a$ form part of a normal lax monoidal structure on $\hattheta$.
\end{proposition}
\begin{proof}
Since we chose $\tensor_1$ to be $\id_{\hattheta}$, we may take $\iota = \id_{\id_{\hattheta}}$.
The transformation $\mu$ is defined on the representables as follows.
Recall that each $\kappa$-cell in
\[
\tensor_a\left(\tensor_{b_1}\left(\cell^{\theta_{11}},\dots,\cell^{\theta_{1b_1}}\right),\dots,\tensor_{b_a}\left(\cell^{\theta_{a1}},\dots,\cell^{\theta_{ab_a}}\right)\right)
\]
is (non-uniquely) represented by $\zeta_1,\dots,\zeta_a \in \cell$ together with 2-functors
\[
\begin{split}
\phi &: \kappa \to \ttensor_a(\zeta_1,\dots,\zeta_a)\\
\phi_i &: \zeta_i \to \ttensor_{b_i}(\theta_{i1},\dots,\theta_{ib_i})
\end{split}
\]
for $1 \le i \le a$.
Then $\mu_{b_1,\dots,b_a}$ sends this cell to the $\kappa$-cell in
\[
\tensor_{b_1+\dots+b_a}\left(\cell^{\theta_{11}},\dots,\cell^{\theta_{ab_a}}\right)
\]
represented by the 2-functor
\[
\begin{tikzcd}[column sep = large]
\kappa \arrow [r, "\phi"] & \ttensor_a(\zeta_1,\dots,\zeta_a) \arrow [r, "{\ttensor_a(\phi_1,\dots,\phi_a)}"] & \ttensor_{b_1+\dots+b_a}(\theta_{11},\dots,\theta_{ab_a}).
\end{tikzcd}
\]
That $\mu$ is well-defined, natural, and satisfies the coherence conditions is all straightforward to check.
\end{proof}

\subsection{The comparison map $\mu$}
Fix $a,b_1,\dots,b_a \in \NN$, and let $b = \sum_{u=1}^a b_u$.
Fix $\theta_i \in \cell$ for $1 \le i \le b$.
We will show that
\[
\mu = (\mu_{b_1,\dots,b_a})_{\cell^{\theta_1},\dots,\cell^{\theta_b}} : A \to B
\]
is a monomorphism and moreover characterise its image, where
\begin{align*}
A &= \tensor_a\left(\tensor_{b_1}\left(\cell^{\theta_{1}},\dots,\cell^{\theta_{b_1}}\right),\dots,\tensor_{b_a}\left(\Theta_2^{\theta_{b-b_a+1}},\dots,\cell^{\theta_{b}}\right)\right),\\
B &= \tensor_b\left(\cell^{\theta_{1}},\dots,\cell^{\theta_{b}}\right) = N\bigl(\ttensor_b(\theta_{1},\dots,\theta_{b})\bigr).
\end{align*}

Let $\rho : \{1,\dots,b\} \to \{1,\dots,a\}$ denote the unique function such that
\[
\sum_{u < \rho(i)}b_u < i \le \sum_{u \le \rho(i)}b_u
\]
for each $1 \le i \le b$.
Informally speaking, for each $1 \le i \le b$, the $i$-th factor $\cell^{\theta_i}$ is contained in the $\rho(i)$-th ``subtensor''.

\begin{definition}\label{pure}
	Let $\phi$ be a $(1;q)$-cell in $B$ with endpoints $\ss,\tt$ and underlying shuffles $\prep$.
	We say $\phi$ is \emph{pure} if, for each pair $(i|k),(j|\ell) \in S(\ss,\tt)$ with $\rho(i) = \rho(j)$, at least one of the following holds:
	\begin{itemize}
		\item [(i)]
		$(i|k) \prep (j|\ell)$ for all $0 \le p \le q$;
		\item [(ii)]
		$(i|k) \sucp (j|\ell)$ for all $0 \le p \le q$; or
		\item [(iii)]
		for any $(m|n) \in S$ and for any $0 \le p \le q$, if
		\[
		(i|k) \prep (m|n) \prep (j|\ell) \hspace{10pt}\text{or}\hspace{10pt}
		(i|k) \sucp (m|n) \sucp (j|\ell)
		\]
		then $\rho(m) = \rho(i)$.
	\end{itemize}
	More generally, call an $(n;\qq)$-cell $\phi$ in $B$ \emph{pure} if, for each $1 \le k \le n$, the $(1;q_k)$-cell $\phi \cdot \eta_h^k$ is pure in the above sense.
	(See \cref{etah} for the definition of $\eta_h^k$.)
\end{definition}

	If we take $a = 2$, $b_1 = 2$, $b_2 = 1$ and $\theta_1 = \theta_2 = \theta_3 = [1;0]$ then we recover the example considered at the beginning of this section.
	In this case, the $(1;1)$-cell (\ref{impurity example}) which has
	\[
	\begin{gathered}
	(2|1) \pre_0 (3|1) \pre_0 (1|1),\\
	(1|1) \pre_1 (2|1) \pre_1 (3|1)
	\end{gathered}
	\]
	is not pure; consider $(i|k) = (1|1)$ and $(j|\ell) = (2|1)$.

The rest of this subsection is devoted to proving the following theorem.
\begin{theorem}\label{mu}
	The map $\mu : A \to B$ is a monomorphism, and its image consists precisely of the pure cells.
\end{theorem}

\begin{proof}
Every cell in the image $\mu(A)$ is pure by \cref{pure image}.
That every pure cell is in the image of $\mu(A)$ follows from \cref{homs suffice,factorising}.
Finally, the map $\mu$ is a monomorphism by \cref{mu mono}.
\end{proof}

\begin{lemma}\label{pure image}
	Every cell in the image $\mu(A)$ is pure.
\end{lemma}

Note that a $(1;q)$-cell $\phi$ in $B$ is contained in $\mu(A)$ if and only if it admits a factorisation of the form
\begin{equation}\label{mu factorisation}
\phi :
\begin{tikzcd}[column sep = large]
{[1;q]}
\arrow [r, "\chi"] &
\ttensor_a(\zeta_1,\dots,\zeta_a)
\arrow [r, "{\ttensor_a(\psi_1,\dots,\psi_a)}"] &
\ttensor_b(\theta_1, \dots, \theta_b).
\end{tikzcd}
\end{equation}

\begin{proof}
It suffices to check the $(1;q)$-cells.
So consider a $(1;q)$-cell $\phi$ with a factorisation (\ref{mu factorisation}).
Given $(i|k),(j|\ell) \in S$ with $\rho(i) = \rho(j) = u$, let $x,y \in \zeta_u$ be the unique objects such that
\[
\begin{gathered}
\pi_i \circ \psi_u(x-1)<k \le \pi_i \circ \psi_u(x),\\
\pi_j \circ \psi_u(y-1)< \ell \le \pi_j \circ \psi_u(y).
\end{gathered}
\]
Then we must have precisely one of the following:
\begin{itemize}
	\item $x<y$, in which case the pair $(i|k), (j|\ell)$ satisfies \cref{pure}(i);
	\item $x>y$, in which case the pair satisfies (ii); or
	\item $x = y$, in which case the pair satisfies (iii).
\end{itemize}
This completes the proof.
\end{proof}

\begin{lemma}\label{homs suffice}
	An $(n;\qq)$-cell $\phi$ in $B$ is contained in $\mu(A)$ if and only if $\phi \cdot \eta_h^k$ is contained in $\mu(A)$ for each $1 \le k \le n$.
\end{lemma}

\begin{proof}
In this proof, we say a factorisation of the form (\ref{mu factorisation}) is \emph{nice} if the composite
	\[
	\begin{tikzcd}
	{[1;q]}
	\arrow [r, "\chi"] &
	\ttensor_a(\zeta_1,\dots,\zeta_a)
	\arrow [r, "\pi_u"] &
	\zeta_u
	\end{tikzcd}
	\]
preserves the first and last objects for each $1 \le u \le a$.
Note that any factorisation of the form (\ref{mu factorisation}) can be made into a nice one by replacing each $\zeta_u$ by an appropriate horizontal face.
	
Now let $\phi$ be an $(n;\qq)$-cell such that each $\phi \cdot \eta_h^k$ admits a factorisation
	\[
	\begin{tikzcd}[column sep = large]
	{[1;q_k]}
	\arrow [r, "\chi^k"] &
	\ttensor_a(\zeta^k_1,\dots,\zeta^k_a)
	\arrow [r, "{\ttensor_a(\psi^k_1,\dots,\psi^k_a)}"] &
	\ttensor_b(\theta_1, \dots, \theta_b)
	\end{tikzcd}
	\]
which we may assume to be nice.
Then we can factorise $\phi$ as
	\[
	\phi :
	\begin{tikzcd}[column sep = large]
		{\nq}
		\arrow [r, "\chi"] &
		\ttensor_a(\zeta_1,\dots,\zeta_a)
		\arrow [r, "{\ttensor_a(\psi_1,\dots,\psi_a)}"] &
		\ttensor_b(\theta_1, \dots, \theta_b)
	\end{tikzcd}
	\]
	where
	$\zeta_u$ is obtained by concatenating $\zeta_u^k$'s, or more precisely by taking the colimit of
	
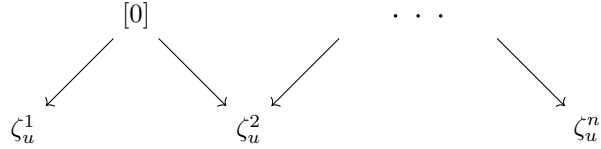
\[
	\begin{tikzpicture}[scale = 1.5, baseline = 12]
	\node at (1,1) {$[0]$};
	\node at (0,0) {$\zeta^1_u$};
	\node at (2,0) {$\zeta^2_u$};
	\node at (5,0) {$\zeta^n_u$};
	
	\draw[->] (0.8,0.8) -- (0.2,0.2);
	\draw[->] (1.2,0.8) -- (1.8,0.2);
	\draw[->] (2.8,0.8) -- (2.2,0.2);
	\draw[->] (4.2,0.8) -- (4.8,0.2);
	
	\foreach \x in {3.3,3.5,3.7}
	\filldraw
	(\x , 1) circle [radius = 0.3pt];
	\end{tikzpicture}
	\]
	in $\twoCat$, $\psi_u$ is the induced map from this colimit, and $\chi$ is obtained by taking the colimit of the top zigzag in the following diagram:
	
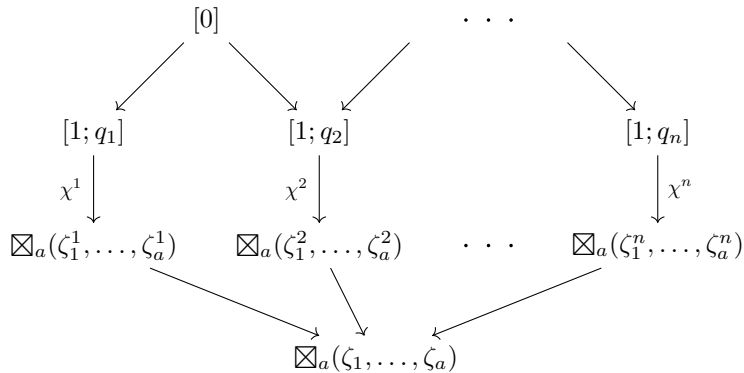
\[
	\begin{tikzpicture}[scale = 1.5]
	\node at (0,1) {$\ttensor_a(\zeta^1_1,\dots,\zeta^1_a)$};
	\node at (0,2) {$[1;q_1]$};
	\node at (1,3) {$[0]$};
	\node at (2.5,0) {$\ttensor_a(\zeta_1,\dots,\zeta_a)$};
	\node at (2,1) {$\ttensor_a(\zeta^2_1,\dots,\zeta^2_a)$};
	\node at (2,2) {$[1;q_2]$};
	\node at (5,2) {$[1;q_n]$};
	\node at (5,1) {$\ttensor_a(\zeta^n_1,\dots,\zeta^n_a)$};
	
	\foreach \x in {0,2,5}
	\draw[->] (\x ,1.8) -- (\x ,1.2);
	\draw[->] (0.8,2.8) -- (0.2,2.2);
	\draw[->] (1.2,2.8) -- (1.8,2.2);
	\draw[->] (2.8,2.8) -- (2.2,2.2);
	\draw[->] (4.2,2.8) -- (4.8,2.2);
	\draw[->] (0.5,0.8) -- (2,0.2);
	\draw[->] (2.1,0.8) -- (2.4,0.2);
	\draw[->] (4.5,0.8) -- (3,0.2);
	
	\foreach \x in {3.3,3.5,3.7}
	\filldraw
	(\x , 3) circle [radius = 0.3pt]
	(\x , 1) circle [radius = 0.3pt];
	
	\node[scale = 0.8] at (-0.2,1.5) {$\chi^1$};
	\node[scale = 0.8] at (1.8,1.5) {$\chi^2$};
	\node[scale = 0.8] at (5.2,1.5) {$\chi^n$};
	\end{tikzpicture}
	\]
	Thus $\phi$ is in $\mu(A)$.
\end{proof}

\begin{lemma}\label{factorising}
	Every pure $(1;q)$-cell in $B$ is contained in $\mu(A)$.
\end{lemma}
We will prove this lemma by constructing a factorisation of the form (\ref{mu factorisation}) for each pure $(1;q)$-cell $\phi$.
The intuition behind the construction below is as follows.
First, we observe that the condition (iii) in \cref{pure} tells us which elements of
\[
S_u \defeq \{(i|k) \in S (\ss,\tt): \rho(i) = u\}
\]
(where $\ss,\tt$ are the endpoints of $\phi$) can be ``bundled together'', and moreover the purity of $\phi$ implies that the collection of these bundles (for fixed $u$) admits a canonical ordering.
The horizontal component of each $\zeta_u$ is then set to be the indexing total order for this collection, whereas the vertical components of $\zeta_u$ are all $[q]$.
The first factor $\chi$ is then essentially determined by how the bundles coming from different $u$'s are ordered with respect to each other (by the underlying shuffles of $\phi$), and each $\psi_u$ is essentially determined by how the elements are ordered within each bundle in $S_u$.

\begin{proof}
Let $\phi$ be a pure $(1;q)$-cell in $B$ with endpoints $\ss,\tt$.
Let $\prez, \dots, \preq$ be the underlying shuffles of $\phi$ on the set $S = S(\ss,\tt)$.
Define a binary relation $\sim$ on $S$ so that $(i|k)\sim(j|\ell)$ if and only if
\begin{itemize}
	\item $\rho(i) = \rho(j)$; and
	\item for any $(m|n) \in S$ and for any $0 \le p \le q$, if
	\[
	(i|k) \prep (m|n) \prep (j|\ell) \hspace{10pt}\text{or}\hspace{10pt}
	(i|k) \sucp (m|n) \sucp (j|\ell)
	\]
	then $\rho(m) = \rho(i)$.
\end{itemize}
(The second clause is precisely \cref{pure}(iii).)
It is straightforward to check that $\sim$ is an equivalence relation.
We will write $[i|k]$ for the $\sim$-class containing $(i|k) \in S$.

For each $1 \le u \le a$, let $S_u \defeq \{(i|k) \in S : \rho(i) = u\}$ and define a binary relation $\le_u$ on the quotient $T_u \defeq S_u /\hspace{-3pt}\sim$ so that $[i|k] \le_u [j|\ell]$ if and only if
\begin{itemize}
	\item $(i|k) \sim (j|\ell)$; or
	\item $(i|k) \prep (j|\ell)$ for all $0 \le p \le q$.
\end{itemize}
Before checking that $\le_u$ is well-defined, notice that if $(i|k) \nsim (j|\ell)$ then the purity of $\phi$ implies that we have either $(i|k) \prep (j|\ell)$ for all $p$, or $(i|k) \sucp (j|\ell)$ for all $p$.
Thus $\le_u$ can be equivalently defined as: $[i|k] \le_u [j|\ell]$ if and only if
\begin{itemize}
	\item $(i|k) \sim (j|\ell)$; or
	\item $(i|k) \prep (j|\ell)$ for \emph{some} $0 \le p \le q$,
\end{itemize}
or alternatively, for any \emph{fixed} $0 \le p \le q$, we can define: $[i|k] \le_u [j|\ell]$ if and only if
\begin{itemize}
	\item $(i|k) \sim (j|\ell)$; or
	\item $(i|k) \prep (j|\ell)$.
\end{itemize}
It is easy to see from the third definition that, assuming it is well-defined, $\le_u$ is a total order
\[
S_{u,1} \le_u S_{u,2} \le_u \dots \le_u S_{u,z_u}
\]
on $T_u$ where $z_u \defeq |T_u|$ and each $S_{u,v} \subset S_u$ is a $\sim$-class.

To see that $\le_u$ is indeed well-defined, consider two $\sim$-related pairs $(i|k) \sim (i'|k')$ and $(j|\ell) \sim(j'|\ell')$ in $S_u$.
If $(i|k) \sim (j|\ell)$ then $(i'|k') \sim (j'|\ell')$ by the transitivity of $\sim$.
So consider the case where $(i|k) \nsim (j|\ell)$.
Making use of the first and second definitions of $\le_u$, it suffices to prove that
\[
\forall p \bigl[(i|k) \prep (j|\ell)\bigr] \Longrightarrow \exists p \bigl[(i'|k') \prep (j'|\ell')\bigr].
\]
So assume that $(i|k) \prep (j|\ell)$ for all $p$.
Then $(i|k) \nsim (j|\ell)$ implies that there exist $0 \le p \le q$ and $(m|n) \in S$ such that $\rho(m) \neq u$ and $(i|k) \prep (m|n) \prep (j|\ell)$.
Since $(i|k) \sim (i'|k')$ and $(j|\ell) \sim (j'|\ell')$, we can then infer $(i'|k') \prep (m|n) \prep (j'|\ell')$ as desired.

For each $1 \le u \le a$, let $\zeta_u \defeq \bigl[z_u;q,\dots,q\bigr] \in \cell$.
Then we can specify a $(1;q)$-cell
\[
\bar \chi : [1;q] \to \ttensor_a(\bar \zeta_1, \dots, \bar \zeta_a)
\]
with endpoints $\zzero,\zz$ by specifying shuffles $\unlhd_0 \LHD \dots \LHD \unlhd_q$ on $T \defeq S /\hspace{-3pt}\sim\;= \coprod_u T_u$.
Here a \emph{shuffle} $\unlhd$ on $T$ is a total order on $T$ such that $[i|k] \unlhd [j|\ell]$ for any $(i|k),(j|\ell) \in S_u$ with $[i|k] \le_u [j|\ell]$, and $\unlhd \LHD \unlhd'$ if $[i|k] \unlhd [j|\ell]$ implies $[i|k] \unlhd' [j|\ell]$ for any $(i|k) \in S_u$, $(j|\ell) \in S_v$ with $u < v$.

For each $0 \le p \le q$, define a binary relation $\unlhd_p$ on $T$ so that $[i|k] \unlhd_p [j|\ell]$ if and only if $(i|k) \sim (j|\ell)$ or $(i|k) \prep (j|\ell)$.
Note that this agrees with the third definition of $\le_u$ on each $T_u$ (and hence, assuming it is a well-defined total order, $\unlhd_p$ is a shuffle).
Thus to check that $\unlhd_p$ is well-defined, we only need to consider two $\sim$-related pairs $(i|k)\sim(i'|k')$ and $(j|\ell)\sim(j'|\ell')$ such that $\rho(i) \neq \rho(j)$.
In this case, it follows from our definition of $\sim$ that $(i|k) \prep (j|\ell)$ implies $(i'|k')\prep(j'|\ell')$.
Hence $\unlhd_p$ is indeed well-defined, and moreover it is a total order since $\prep$ is so.
Furthermore, it is easy to check that $\prep \LHD \pre_{p'}$ implies $\unlhd_p \LHD \unlhd_{p'}$ for any $0 \le p \le p' \le q$.
Thus we obtain the desired map $\bar \chi$.
There is a map
\[
\chi_u = \bigl[\{0,z_u\};\id,\dots,\id\bigr] : [1;q] \to \zeta_u
\]
in $\cell$ for each $1 \le u \le a$, and these maps induce $\chi$ as in
\[
\begin{tikzcd}[row sep = large]
{[1;q]}
\arrow [drr, "\bar \chi", bend left = 20]
\arrow [ddr, "{\langle \chi_1, \dots, \chi_a \rangle}", swap, bend right, end anchor = north west]
\arrow [dr, dashed, "\chi"] & & \\
& \ttensor_a(\zeta_1, \dots, \zeta_a)
\arrow [r]
\arrow [d]
\arrow [phantom, dr, "\lrcorner" very near start] &
\ttensor_a(\bar \zeta_1, \dots, \bar \zeta_a)
\arrow [d] \\
& \zeta_1 \times \dots \times \zeta_a
\arrow [r] &
\bar \zeta_1 \times \dots \times \bar \zeta_a
\end{tikzcd}
\]
where the inner square is the pullback square in \cref{Gray pullback bar}.

Now we construct the remaining part of the factorisation (\ref{mu factorisation}), namely
\[
\psi_u: \zeta_u \to \ttensor_{b_u}(\theta_{u,1},\dots,\theta_{u,b_u})
\]
for each $1 \le u \le a$, where $\theta_{u,i} \defeq \theta_{b_1+\dots+b_{u-1}+i}$ denotes the $i$-th factor in the $u$-th ``subtensor''.
First, define the object part of
\[
\bar \psi_u : \zeta_u \to \ttensor_{b_u}(\bar \theta_{u,1},\dots,\bar \theta_{u,b_u})
\]
by sending each $0 \le v \le z_u$ to the object whose $i$-th coordinate is given by
\[
\max\bigl(\bigl\{k: \exists v' \le v[(i|k)\in S_{u,v'}]\bigr\} \cup \{s_i\}\bigr).
\]
Its action on $\hom_{\zeta_u}(v-1,v) = [q]$ is given by restricting the $\prep$'s to $S_{u,v}$.

Fix $1 \le i \le b_u$.
Define the horizontal component of
\[
\psi_{u,i} : \zeta_u \to \theta_{u,i}
\]
by the same formula as above, \emph{i.e.}~it sends each $0 \le v \le |T_u|$ to
\[
\max\bigl(\bigl\{k: \exists v' \le v[(i|k)\in S_{u,v'}]\bigr\} \cup \{s_i\}\bigr).
\]
If $s_i < k \le t_i$ then the $k$-th vertical component of $\psi_{u,i}$ is that of
\[
\begin{tikzcd}
{[1;q]}
\arrow [r, "\phi"] &
B
\arrow [r, "\pi_i"] &
\theta_i.
\end{tikzcd}
\]
Finally, these maps induce $\psi_u$ as in
\[
\begin{tikzcd}[row sep = large]
\zeta_u
\arrow [drr, "\bar \psi_u", bend left = 20]
\arrow [ddr, "{\langle \psi_{u,1}, \dots, \psi_{u,b_u} \rangle}", swap, bend right, end anchor = north west]
\arrow [dr, dashed, "\psi_u"] & & \\
& \ttensor_{b_u}(\theta_{u,1}, \dots, \theta_{u,b_u})
\arrow [r]
\arrow [d]
\arrow [phantom, dr, "\lrcorner" very near start] &
\ttensor_a(\bar \theta_{u,1}, \dots, \bar \theta_{u,b_u})
\arrow [d] \\
& \theta_{u,1} \times \dots \times \theta_{u,b_u}
\arrow [r] &
\bar \theta_{u,1} \times \dots \times \bar \theta_{u,b_u}
\end{tikzcd}
\]
and one can check that
\[
\begin{tikzcd}[column sep = large]
{[1;q]}
\arrow [r, "\chi"] &
\ttensor_a(\zeta_1,\dots,\zeta_a)
\arrow [r, "{\ttensor_a(\psi_1,\dots,\psi_a)}"] &
\ttensor_b(\theta_1, \dots, \theta_b).
\end{tikzcd}
\]
is indeed a factorisation of $\phi$.
\end{proof}

\begin{lemma}\label{mu mono}
	The map $\mu : A \to B$ is a monomorphism.
\end{lemma}
\begin{proof}
Consider an $(n;\qq)$-cell $\phi$ in the image of $\mu$.
The proof of \cref{factorising} constructs a factorisation of each $\phi \cdot \eta_h^k$, and then the proof of \cref{homs suffice} combines them into a factorisation
	\[
	\phi :
	\begin{tikzcd}[column sep = large]
	{\nq}
	\arrow [r, "\chi"] &
	\ttensor_a(\zeta_1,\dots,\zeta_a)
	\arrow [r, "{\ttensor_a(\psi_1,\dots,\psi_a)}"] &
	\ttensor_b(\theta_1, \dots, \theta_b)
	\end{tikzcd}
	\]
of $\phi$.
We wish to prove that $(\chi,\psi_1,\dots,\psi_a)$ represents a unique cell in
\[
A = \tensor_a\left(\tensor_{b_1}\left(\cell^{\theta_1},\dots,\cell^{\theta_{b_1}}\right),\dots,\tensor_{b_a}\left(\cell^{\theta_{b-b_a+1}},\dots,\cell^{\theta_{b_a}}\right)\right)
\]
that is sent to $\phi$ by $\mu$.
So suppose that $(\chi',\psi'_1,\dots,\psi'_a)$ also represents such a cell in $A$, \emph{i.e.}~that
	\[
	\phi :
	\begin{tikzcd}[column sep = large]
	{\nq}
	\arrow [r, "\chi'"] &
	\ttensor_a(\zeta'_1,\dots,\zeta'_a)
	\arrow [r, "{\ttensor_a(\psi'_1,\dots,\psi'_a)}"] &
	\ttensor_b(\theta_1, \dots, \theta_b)
	\end{tikzcd}
	\]
is another factorisation of $\phi$.
If $\psi'_u$ can be factored as $\psi'_u = \delta_u \circ \sigma_u$ then 
\[
\bigl(\ttensor_a(\sigma_1,\dots,\sigma_a)\circ\chi';\delta_1,\dots,\delta_a\bigr)
\]
represents the same cell in $A$, so we may assume without loss of generality that each $\psi'_u$ is a non-degenerate cell in (the nerve of) $\ttensor_{b_u}(\theta_{u,1},\dots,\theta_{u,b_u})$.
This implies that $\ttensor_a(\psi'_1,\dots,\psi'_a)$ is a monomorphism.
Now for each $1 \le u \le a$ and $1 \le i \le n$, consider the diagram:
\[
\begin{tikzcd}
\zeta_u^k
\arrow [r, hook]
\arrow [ddr, dashed, "\omega_u^k", swap] &
\zeta_u
\arrow [dr, "\psi_u", end anchor = north west]
\arrow [dd, "\omega_u", dashed] & \\
& & \ttensor_{b_u}(\theta_{u,1},\dots,\theta_{u,b_u})\\
& \zeta'_u
\arrow [ur, "\psi'_u", swap, end anchor = south west] &
\end{tikzcd}
\]
By construction of $\psi_u^k$, the image of $\psi_u'$ contains all of the objects in the image of $\psi_u^k$.
So, at least on the object level, there is $\omega_u^k$ as indicated above that renders the perimeter commutative.
We can upgrade it to a morphism in $\cell$ by setting its $v$-th vertical component to be that of
\[
\begin{tikzcd}
{\nq}
\arrow [r, "\chi'"] &
\ttensor_a(\zeta'_1,\dots,\zeta'_a)
\arrow [r, "\pi_u"] &
\zeta'_u
\end{tikzcd}
\]
for $\pi_u \circ \chi'(k-1) < v \le \pi_u \circ \chi'(k)$.
Since $\zeta_u$ is the colimit of $\zeta_u^k$'s, these induce a unique map $\omega_u$ as indicated.
Now in the diagram
\[
\begin{tikzcd}
& \ttensor_a(\zeta_1,\dots,\zeta_a)
\arrow [dr, "{\ttensor_a(\psi_1,\dots,\psi_a)}", end anchor = north west]
\arrow [dd, "{\ttensor_a(\omega_1,\dots,\omega_a)}" description] & \\
{\nq}
\arrow [ur, "\chi"]
\arrow [dr, "\chi'", swap]
& & \ttensor_{b}(\theta_1,\dots,\theta_b)\\
& \ttensor_a(\zeta'_1,\dots,\zeta'_a)
\arrow [ur, "{\ttensor_a(\psi'_1,\dots,\psi'_a)}", swap, end anchor = south west] &
\end{tikzcd}
\]
the perimeter commutes since both of the two paths compose to $\phi$, and the right triangle commutes by construction of $\omega_u$.
Moreover we know that the lower right map is a monomorphism, so the left triangle also commutes.
This shows that $(\chi',\psi'_1,\dots,\psi'_a)$ and  $(\chi,\psi_1,\dots,\psi_a)$ represent the same cell in $A$, as desired.
\end{proof}

\subsection{The Leibniz comparison map $\hat \mu$}\label{hat mu subsection}
Fix $a, b_1, \dots, b_a \in \NN$ and let $b = \sum_{u=1}^a b_u$.
Note that the natural transformation
\[
\mu : \tensor_a\bigl(\tensor_{b_1},\dots,\tensor_{b_a}\bigr) \to \tensor_b
\]
may be regarded as a $(b+1)$-ary functor
\[
M : \underbrace{\hattheta \times \dots \times \hattheta}_{b \text{ times}} \times\; \two \to \hattheta.
\]
\begin{definition}
    We define the \emph{Leibniz comparison map} $\hat \mu$ to be the $b$-ary functor
\[
\hat \mu \defeq \hat M \bigl(- , \dots, -, 0 \to 1\bigr) : \hattheta^\two \times \dots \times  \hattheta^\two \to \hattheta^\two.
\]
\end{definition}

The aim of this subsection is to prove the following theorem.
\begin{theorem}\label{hat mu}
	For any monomorphisms $f_1,\dots,f_b$ in $\hattheta$, the Leibniz comparison map $\hat \mu(f_1,\dots,f_b)$ is in $\celll(\H_h \cup \H_v)$.
\end{theorem}
\begin{proof}
	By \cref{celll,mono}, it suffices to prove the special case where each $f_i$ is the boundary inclusion into a representable cellular set.
	This follows from \cref{nu mono,nu 1,nu 2,nu 3} proved below.
\end{proof}
The following corollary of \cref{hat mu} states that the Gray tensor product is associative up to homotopy.
In particular, our lax monoidal structure is \emph{homotopical} in the sense of Heuts, Hinich and Moerdijk \cite[\textsection 6.3]{Heuts;Hinich;Moerdijk} except that ours is not symmetric.
\begin{corollary}\label{associativity}
	For any $X^1, \dots, X^b \in \hattheta$, the component
	\[
	\tensor_a\bigl(\tensor_{b_1}(X^1,\dots,X^{b_1}),\dots,\tensor_{b_a}(X^{b-b_a+1},\dots,X^b)\bigr) \to \tensor_b(X^1,\dots,X^b)
	\]
	of $\mu$ is in $\celll(\H_h \cup \H_v)$.
\end{corollary}
\begin{proof}
	Apply \cref{hat mu} to the empty inclusions $f_i : \varnothing \incl X^i$.
\end{proof}

Now we complete the proof of \cref{hat mu}.
Fix $\theta_1, \dots, \theta_b \in \cell$, and let $\nu : A^0 \to B$ denote the Leibniz comparison map
\[
\nu \defeq \hat \mu \bigl(\partial\cell^{\theta_1} \incl \cell^{\theta_1}, \dots, \partial\cell^{\theta_b} \incl \cell^{\theta_b}\bigr).
\]

\begin{lemma}\label{nu mono}
	The map $\nu$ is a monomorphism.
\end{lemma}
\begin{proof}
By \cref{Leibniz mono}, it suffices to prove that the functor
\[
G : \two^{b+1} \to \hattheta
\]
(defined as in \cref{Leibniz} with $F=M$) sends each square of the form (\ref{ij square}) to a pullback square of monomorphisms.
The case $i,j \le b$ was treated in \cref{boundary}, so we may assume $j =b+1$. 
Fix $1 \le i \le b$, and let
\[
\begin{tikzcd}
A'
\arrow [r]
\arrow [d] &
A
\arrow [d] \\
B'
\arrow [r] &
B
\end{tikzcd}
\]
be the image of the square (\ref{ij square}) under $G$.
The horizontal maps are monic by \cref{boundary}, and the right vertical one is monic by \cref{associativity}.
Moreover the commutativity of this square then implies that the left vertical map is also monic.

It remains to prove that this square is a pullback.
So consider a pure $(n;\qq)$-cell $\phi$ contained in the image of the map
\[
\ttensor_b(\theta_1,\dots,\theta_{i-1},\kappa,\theta_{i+1},\dots,\theta_b) \to \ttensor_b(\theta_1,\dots,\theta_b)
\]
induced by some hyperface $\delta : \kappa \to \theta_i$.
It is straightforward to check that, in the factorisation of $\phi$ constructed in the proof of \cref{mu mono}, the map $\psi_{\rho(i)}$ then factors through the obvious sub-2-category of the codomain determined by $\delta$.
Hence this factorisation specifies a cell in $A'$ as desired.
%
\end{proof}

Thus we may regard $\nu : A^0 \to B$ as a cellular subset inclusion.
By \cref{mu}, $A^0$ is generated by $A$ and the pure cells.
Let $A^1 \subset B$ be the cellular subset generated by $A^0$ and the ($\sil$-)cuttable cells.
\begin{lemma}\label{nu 1}
	The inclusion $A^0 \incl A^1$ is in $\celll(\H_h)$.
\end{lemma}
\begin{proof}
Observe that for any $\sil$-cuttable cell $\chi$ in $B$ that is not cuttable, $\chi$ is pure if and only if its cuttable parent is pure.
The rest of the proof is similar to the first part of the proof of \cref{horizontal horn}.
\end{proof}

Now consider a non-degenerate cell $\phi$ in $B \setminus A^1$.
Note that $\phi$ is necessarily a $(1;q)$-cell for some $q \ge 1$ with endpoints $\zzero,\tt$ where $t_i$ is the horizontal length of $\theta_i$ (\emph{i.e.}~$\bar \theta_i = [t_i;\zzero]$).
Let $S = S(\zzero,\tt)$ and let $\prez,\dots,\preq$ be the underlying shuffles of $\phi$.
Since $\phi$ is not pure, $\phi$ must contain an \emph{impurity} in the following sense.

\begin{definition}
	An \emph{upper impurity} in $\phi$ is a quadruple $\mathcal{U} = \langle(i|k),(j|\ell),(m|n),p\rangle$ consisting of $(i|k)$, $(j|\ell)$, $(m|n) \in S$ and $p \in [q]$ such that:
	\begin{itemize}
		\item $i<j$;
		\item $\rho(i) = \rho(j) \neq \rho(m)$;
		\item $(i|k) \suc_0 (j|\ell)$; and
		\item $(i|k) \prep (m|n) \prep (j|\ell)$.
	\end{itemize}
	(See \cref{figure impurity}.)
A \emph{lower impurity} in $\phi$ is a quadruple $\mathcal{L} = \langle(i|k),(j|\ell),(m|n),p\rangle$ consisting of $(i|k)$, $(j|\ell)$, $(m|n) \in S$ and $p \in [q]$ such that:
\begin{itemize}
	\item $i<j$;
	\item $\rho(i) = \rho(j) \neq \rho(m)$;
	\item $(i|k) \preq (j|\ell)$; and
	\item $(i|k) \sucp (m|n) \sucp (j|\ell)$.
\end{itemize}
We say $\phi$ is an \emph{upper} cell if it contains no lower impurities.
\end{definition}

\begin{figure}
    \[
	\begin{tikzpicture}
	\filldraw[gray!20!white]
	(-1.4,-0.3) -- (1.35,-0.3) -- (3.05,2.3) -- (-3.1,2.3) -- cycle;
	\filldraw[gray!70!white]
	(-0.5,1.7) rectangle (0.45,2.3);
	\draw (-5,-0.3) rectangle (5,0.3)
	(-5,1.7) rectangle (5,2.3);
	\node at (0,2) {$\dots\prep(i|k)\prep\dots\prep(m|n)\prep\dots\prep(j|\ell)\prep\dots$};
	\node at (0,0) {$\dots\prez(j|\ell)\prez\dots\prez(i|k)\prez\dots$};
	\node at (-6,0) {$\prez$};
	\node at (-6,2) {$\prep$};
	\end{tikzpicture}
	\]
    \caption{A typical upper impurity}\label{figure impurity}
\end{figure}

Let $A^2 \subset B$ be the cellular subset generated by $A^1$ and the upper cells.
Since any face of an upper cell is itself upper, any non-degenerate face in $A^2 \setminus A^1$ must be upper.

\begin{lemma}\label{nu 2}
	The inclusion $A^1 \incl A^2$ is in $\celll(\H_v)$.
\end{lemma}
\begin{proof}
Fix a non-degenerate $(1;q)$-cell $\phi$ in $A^2 \setminus A^1$ (which is necessarily upper).
Define a total order $\le$ on the set of upper impurities in $\phi$ so that
\[
\langle(i|k),(j|\ell),(m|n),p\rangle \le \langle(i'|k'),(j'|\ell'),(m'|n'),p'\rangle
\]
if and only if:
\begin{itemize}
	\item $p < p'$;
	\item $p = p'$ and $(i|k) >_{lex} (i'|k')$;
	\item $p=p'$, $(i|k) = (i'|k')$ and $(j|\ell) <_{lex} (j'|\ell')$; or
	\item $p=p'$, $(i|k) = (i'|k')$, $(j|\ell) = (j'|\ell')$ and $(m|n) \lelex (m'|n')$.
\end{itemize}
Here $\lelex$ denotes the lexicographical order so that $(i|k) \lelex (j|\ell)$ if and only if either:
\begin{itemize}
    \item $i<j$; or
    \item $i=j$ and $k \le \ell$.
\end{itemize}
This indeed defines a total order on the set of upper impurities in $\phi$, hence in particular we have a minimum impurity
\[
\mathcal{U}_\phi = \bigl\langle(i_\phi|k_\phi),(j_\phi|\ell_\phi),(m_\phi|n_\phi),p_\phi\bigr\rangle.
\]
Let $s_\phi\in [q]$ be the largest $s$ satisfying $(i_\phi|k_\phi) \sucs (j_\phi|\ell_\phi)$.
Note that we must have $s_\phi < p_\phi$ since $(i_\phi|k_\phi) \prepp (j_\phi|\ell_\phi)$ and $i_\phi < j_\phi$ imply $(i_\phi|k_\phi) \prep (j_\phi|\ell_\phi)$ for all $p \ge p_\phi$.
We will construct the ``best approximation'' $\pre$ to $\presp$ such that $(i_\phi|k_\phi) \pre (j_\phi|\ell_\phi)$ (in the sense of \cref{claim Gray 1} below).

Consider the partition $S = I_1 \cup I_2 \cup I_3 \cup I_4$ where
\begin{align*}
I_1 &= \bigl\{(x|y) \in S : (x|y) \prec_{s_\phi} (j_\phi|\ell_\phi)\bigr\}\\
I_2 &= \bigl\{(x|y) \in S : (j_\phi|\ell_\phi) \presp (x|y) \presp (i_\phi|k_\phi),\hspace{10pt}(x|y)\prepp(i_\phi|k_\phi)\bigr\}\\
I_3 &= \bigl\{(x|y) \in S : (j_\phi|\ell_\phi) \presp (x|y) \presp (i_\phi|k_\phi),\hspace{10pt}(j_\phi|\ell_\phi)\prepp(x|y)\bigr\}\\
I_4 &= \bigl\{(x|y) \in S : (i_\phi|k_\phi) \prec_{s_\phi} (x|y)\bigr\}.
\end{align*}
To see that this is indeed a partition of $S$, observe that if $(x|y)$ satisfies both
\[
(j_\phi|\ell_\phi) \presp (x|y) \presp (i_\phi|k_\phi)\hspace{10pt}\text{and}\hspace{10pt}(i_\phi|k_\phi)\prec_{p_\phi}(x|y)\prec_{p_\phi}(j_\phi|\ell_\phi)
\]
then we must have $i_\phi < x < j_\phi$ since $\presp \LHD \prepp$.
It follows that $\rho(x) = \rho(i_\phi)$.
But then either $\bigl\langle (i_\phi|k_\phi),(x|y),(m_\phi|n_\phi),{p_\phi} \bigr\rangle$ or $\bigl\langle (x|y),(j_\phi|\ell_\phi),(m_\phi|n_\phi),{p_\phi} \bigr\rangle$ is an upper impurity strictly smaller than $\mathcal{U}_\phi$, which contradicts our choice of $\mathcal{U}_\phi$.

\begin{figure}
	\[
	\begin{aligned}
	\presp \hspace{20pt}&
	\begin{tikzpicture}[baseline = 1]
	\draw[pattern = north east lines] (0,0) rectangle (2,0.3) (7,0) rectangle (5,0.3);
	\foreach \x in {2,3,4}
	\filldraw[gray] (\x,0) -- (\x +0.5,0) -- (\x + 0.5,0.3) -- (\x , 0.3) -- cycle;
	\foreach \x in {3,4,5}
	\filldraw[gray!50!white] (\x,0) -- (\x -0.5,0) -- (\x - 0.5,0.3) -- (\x , 0.3) -- cycle;
	
	\draw[gray]
	(2.25,0) .. controls (2.25,-0.3) and (3.2,-0.3) .. (3.2,-0.6)
	(4.25,0) .. controls (4.25,-0.3) and (3.3,-0.3) .. (3.3,-0.6)
	(3.25,0) -- (3.25,-0.6);
	
	\draw[gray!50!white]
	(2.75,0.3) .. controls (2.75,0.6) and (3.7,0.6) .. (3.7,0.9)
	(4.75,0.3) .. controls (4.75,0.6) and (3.8,0.6) .. (3.8,0.9)
	(3.75,0.3) -- (3.75,0.9);
	
	\node[scale = 0.7] at (1,0.5) {$I_1$};
	\node[scale = 0.7] at (3.75,1.1) {$I_2$};
	\node[scale = 0.7] at (3.25,-0.8) {$I_3$};
	\node[scale = 0.7] at (6,0.5) {$I_4$};
	\node[scale = 0.7] at (2,1) {$(j_\phi|\ell_\phi)$};
	\node[scale = 0.7] at (5,-0.7) {$(i_\phi|k_\phi)$};
	
	\draw[->] (2,0.8) -- (2,0.4);
	\draw[->] (5,-0.5) -- (5,-0.1);
	\end{tikzpicture}\\
	\pre \hspace{27pt}&
	\begin{tikzpicture}[baseline = 1]
	\draw[pattern = north east lines] (0,0) rectangle (2,0.3) (7,0) rectangle (5,0.3);
	\filldraw[gray] (3.5,0) rectangle (5,0.3);
	\filldraw[gray!50!white] (2,0) rectangle (3.5,0.3);	
	
	\node[scale = 0.7] at (3.6,1) {$(j_\phi|\ell_\phi)$};
	\node[scale = 0.7] at (3.4,-0.7) {$(i_\phi|k_\phi)$};
	
	\draw[->] (3.6,0.8) -- (3.55,0.4);
	\draw[->] (3.4,-0.5) -- (3.45,-0.1);
	\end{tikzpicture}
	\end{aligned}
	\]
	\caption{$\presp$ and $\pre$}
\end{figure}

Now define a total order $\pre$ on $S$ so that $(x|y) \pre (z|w)$ if and only if either
\begin{itemize}
	\item $(x|y) \in I_u$ and $(z|w) \in I_v$ for some $u<v$; or
	\item $(x|y),(z|w) \in I_u$ for some $u$ and $(x|y) \presp (z|w)$.
\end{itemize}
It is easy to check that $\pre$ is a shuffle using the fact that $\presp$ and $\prepp$ are so.

Observe that $(i_\phi|k_\phi)$ is the $\presp$-maximum element of $I_2$ and $(j_\phi|\ell_\phi)$ is the $\presp$-minimum element of $I_3$.
Therefore $(j_\phi|\ell_\phi)$ is the immediate $\pre$-successor of $(i_\phi|k_\phi)$, which in particular implies $\presp\;\neq\;\pre\;\neq\;\prepp$.

\begin{claim}\label{claim Gray 1}
	The shuffle $\pre$ is $\LHD$-minimum among those shuffles $\pre'$ on $S$ satisfying $\presp \LHD \pre' \LHD \prepp$ and $(i_\phi|k_\phi) \pre' (j_\phi|\ell_\phi)$.
\end{claim}
\begin{proof}[Proof of \cref{claim Gray 1}]
Suppose that $(x|y),(z|w) \in S$ satisfy $(x|y) \presp (z|w)$ and $x<z$.
Then we must have $(x|y) \prepp (z|w)$ since $\presp \LHD \prepp$.
Now it follows from our construction of $\pre$ that $(x|y) \pre (z|w)$ holds too.
	This prove $\presp \LHD \pre$.
	
	Now let $\pre'$ be a shuffle on $S$ satisfying $\presp \LHD \pre' \LHD \prepp$ and $(i_\phi|k_\phi) \pre' (j_\phi|\ell_\phi)$.
	Let $(x|y),(z|w) \in S$ and suppose that both $(x|y) \pre (z|w)$ and $x < z$ hold.
	We wish to show that $(x|y) \pre' (z|w)$ holds.
	\begin{itemize}
		\item
		If $(x|y), (z|w) \in I_u$ for some $u$, then $(x|y) \presp (z|w)$ by the definition of $\pre$.
		Thus our assumption $\presp \LHD \pre'$ implies $(x|y) \pre' (z|w)$.
		\item
		If $(x|y) \in I_1$ and $(z|w) \in I_u$ for some $u \ge 2$ then $(x|y) \presp (j_\phi|\ell_\phi) \presp (z|w)$.
		Thus our assumption $\presp \LHD \pre'$ implies $(x|y) \pre' (z|w)$.
		\item
		Using $(i_\phi|k_\phi)$ in place of $(j_\phi|\ell_\phi)$ in the previous item, we can prove that if $(z|w) \in I_4$ then $(x|y) \pre' (z|w)$.
		\item
		The remaining case is when $(x|y) \in I_2$ (so in particular $(x|y) \presp (i_\phi|k_\phi)$ and $(x|y) \prepp (i_\phi|k_\phi)$) and $(z|w) \in I_3$.
		\begin{itemize}
			\item If $x < i_\phi$, then it follows from our assumptions $(x|y) \presp (i_\phi|k_\phi)$ and $\presp \LHD \pre'$ that $(x|y) \pre' (i_\phi|k_\phi)$.
			\item If $x > i_\phi$, then it follows from our assumptions $(x|y) \prepp (i_\phi|k_\phi)$ and $\pre' \LHD \prepp$ that $(x|y) \pre' (i_\phi|k_\phi)$.
			\item If $x=i_\phi$, then we must have $y \le k_\phi$ since $(x|y) \presp (i_\phi|k_\phi)$ holds and $\presp$ is a shuffle.
			Thus the shuffle $\pre'$ must also satisfy $(x|y) \pre' (i_\phi|k_\phi)$.
		\end{itemize}
		We can similarly deduce $(j_\phi|\ell_\phi) \pre' (z|w)$ and hence
		\[
		(x|y) \pre' (i_\phi|k_\phi) \pre' (j_\phi|\ell_\phi) \pre' (z|w).
		\]
	\end{itemize}
	Therefore we indeed have $(x|y) \pre' (z|w)$, and this shows $\pre \LHD \pre'$.
	In particular, by taking $\pre'\;=\;\prepp$ we can deduce that $\pre \LHD \prepp$.
\end{proof}
Since $\pre$ is a shuffle, it determines a 1-cell $\bar g_\phi$ in $\ttensor_{b}(\bar \theta_1, \dots, \bar \theta_b)$.
We can upgrade it to a 1-cell in $B$ as in:
\[
\begin{tikzcd}[row sep = large]
{[1;0]}
\arrow [drr, bend left = 20, "\bar g_\phi"]
\arrow [dr, dashed, "g_\phi"]
\arrow [d, "{\phi \cdot \eta_v^{s_\phi}}", swap] & & \\
\ttensor_b(\theta_1, \dots, \theta_b)
\arrow [dr, bend right, "{\langle \pi_1, \dots, \pi_b \rangle}", swap] &
\ttensor_b(\theta_1, \dots, \theta_b)
\arrow [dr, phantom, "\lrcorner" very near start]
\arrow [r]
\arrow [d] &
\ttensor_b(\bar \theta_1, \dots, \bar \theta_b)
\arrow [d] \\
& \theta_1 \times \dots \times \theta_b
\arrow [r] &
\bar \theta_1 \times \dots \times \bar \theta_b
\end{tikzcd}
\]
Then $\phi \cdot \eta_v^{s_\phi} < g_\phi \le \phi \cdot \eta_v^{s_\phi+1}$ holds in the hom-poset $\ttensor_b(\ttheta)(\zzero,\tt)$ by \cref{claim Gray 1}.
Consider the following condition on $\phi$:
\begin{itemize}
	\item [(\tast)] $\phi \cdot \eta_v^{s_\phi+1} = g_{\phi}$.
\end{itemize}

\begin{claim}\label{claim Gray 2}
	Suppose that $\phi$ is a non-degenerate $(1;q)$-cell in $A^2 \setminus A^1$ not satisfying {(\tast)}.
	Then there exists a unique non-degenerate $(1;q+1)$-cell $\psi$ in $A^2 \setminus A^1$ such that $\psi$ satisfies {(\tast)} and $\phi = \psi \cdot \delta_v^{1;s_\psi+1}$.
\end{claim}
\begin{proof}[Proof of \cref{claim Gray 2}]
	Observe that, if such $\psi$ exists, then it must satisfy $\psi \cdot \delta_v^{1;s_\phi+1} = \phi$ and $\psi \cdot \eta_v^{s_\phi+1} = g_\phi$.
	(Note that we are using $s_\phi$ and not $s_\psi$.)
	So we define $\psi$ to be the unique cell determined by these conditions, whose existence follows from the observation $\phi \cdot \eta_v^{s_\phi} < g_\phi \le \phi \cdot \eta_v^{s_\phi+1}$.
	This cell $\psi$ is not in $A^1$ since it contains $\phi$ as a face and $\phi$ is not in $A^1$.
	
	We show that $\psi$ is an upper cell (and hence contained in $A^2$).
	Suppose for contradiction that $\psi$ contains a lower impurity $\mathcal{L}$.
	Since $\phi$ contains no lower impurities and $\psi \cdot \eta_v^{q+1} = \phi \cdot \eta_v^q$, this impurity $\mathcal{L}$ must be of the form
	\[
	\mathcal{L} = \langle (i|k),(j|\ell),(m|n),s_\phi+1\rangle.
	\]
	In other words, we have:
	\begin{itemize}
		\item $i < j$;
		\item $\rho(i) = \rho(j) \neq \rho(m)$;
		\item $(i|k) \preq (j|\ell)$; and
		\item $(i|k) \suc (m|n) \suc (j|\ell)$
	\end{itemize}
	where $\prep$ are the underlying shuffles of $\phi$ (and not of $\psi$) and $\pre$ is the shuffle constructed above.
	Note that we can deduce from $i < j$, $(j|\ell) \pre (i|k)$ and $\pres \LHD \pre$ that $(j|\ell) \presp (i|k)$.
	
	Since $\phi$ is upper, $\langle (i|k),(j|\ell),(m|n),s_\phi\rangle$ is not a lower impurity.
	Hence we must have either $(i|k) \presp (m|n)$ or $(m|n) \presp (j|\ell)$.
	In the former case, the assumption $(m|n) \pre (i|k)$ implies $(m|n) \in I_2$ and $(i|k) \in I_3$.
	In particular, we have $(j_\phi|\ell_\phi) \presp (m|n) \presp (i_\phi|k_\phi)$ and $(j_\phi|\ell_\phi) \presp (i|k) \presp (i_\phi|k_\phi)$.
	Hence for neither $\langle(i_\phi|k_\phi),(j_\phi|\ell_\phi),(i|k),s_\phi\rangle$ nor $\langle(i_\phi|k_\phi),(j_\phi|\ell_\phi),(m|n),s_\phi\rangle$ to be a lower impurity in $\phi$, we must have both $\rho(i_\phi) = \rho(i)$ and $\rho(i_\phi) = \rho(m)$.
	This contradicts our assumption $\rho(i) \neq \rho(m)$.
	We can derive a similar contradiction in the case $(m|n) \presp (j|\ell)$ too, and this proves that $\psi$ is upper.
	
	Finally we prove that the minimum impurity $\mathcal{U}_\psi$ in $\psi$ is 
	\[
	\langle(i_\phi|k_\phi),(j_\phi,\ell_\phi),(m_\phi|n_\phi),p_\phi+1\rangle.
	\]
	Assuming this fact, it is straightforward to check that $\psi$ satisfies the condition (\tast) and $\phi = \psi \cdot \delta_v^{1;s_\psi+1}$.
	
	Since $\psi \cdot \delta_v^{1:s_\phi+1} = \phi$, if $\psi$ has an upper impurity $\mathcal{U}$ that is smaller than our tentative $\mathcal{U}_\psi$ above then it must be of the form
	\[
	\mathcal{U} = \langle (i|k),(j|\ell),(m|n),s_\phi+1\rangle.
	\] 
	In other words, we have:
	\begin{itemize}
		\item $i < j$;
		\item $\rho(i) = \rho(j) \neq \rho(m)$;
		\item $(i|k) \sucz (j|\ell)$; and
		\item $(i|k) \pre (m|n) \pre (j|\ell)$.
	\end{itemize}
	Since $\mathcal{U}_\phi$ is the minimum upper impurity in $\phi$, $\langle (i|k),(j|\ell),(m|n),s_\phi\rangle$ is not an upper impurity.
	Hence we must have either $(m|n) \presp (i|k)$ or $(j|\ell) \presp (m|n)$.
	In the former case, the assumption $(i|k) \pre (m|n)$ implies $(i|k) \in I_2$ and $(m|n) \in I_3$.
	In particular, we have $(j_\phi|\ell_\phi) \presp (i|k) \presp (i_\phi|k_\phi)$ and $(j_\phi|\ell_\phi) \presp (m|n) \presp (i_\phi|k_\phi)$.
	For neither $\langle(i_\phi|k_\phi),(j_\phi|\ell_\phi),(i|k),s_\phi\rangle$ nor $\langle(i_\phi|k_\phi),(j_\phi|\ell_\phi),(m|n),s_\phi\rangle$ to be a lower impurity in $\phi$, we must have both $\rho(i_\phi) = \rho(i)$ and $\rho(i_\phi) = \rho(m)$.
	This contradicts our assumption $\rho(i) \neq \rho(m)$.
	We can derive a similar contradiction in the case $(j|\ell) \presp (m|n)$ too, and this completes the proof of \cref{claim Gray 2}.
\end{proof}

We wish to prove that $A^2$ may be obtained from $A^1$ by gluing those $\phi$ satisfying (\tast) along the inner horn $\horn_v^{1;s_\phi+1}$ in lexicographically increasing order of $\sil(\phi)$, $\dim(\phi)$, ${p_\phi}$ and $s_\phi$ where ${p_\phi}$ is regarded as an element of $[q]\op$.
This conclusion can be deduced from the following analysis of the hyperfaces of $\phi$.

\begin{temp}
	In this proof, if $\phi,\psi$ are as described in \cref{claim Gray 2} then we say $\psi$ is the \emph{\tast-parent} of $\phi$.
\end{temp}
Let $\phi$ be a non-degenerate $(1;q)$-cell in $A^2 \setminus A^1$ satisfying (\tast).
The outer hyperfaces $\phi \cdot \delta_v^{1;0}$ and $\phi \cdot \delta_v^{1;q}$ have smaller silhouettes than $\sil(\phi)$.
The hyperface $\phi \cdot \delta_v^{1;s_\phi+1}$ is treated in \cref{claim Gray 3} below.
The hyperface $\phi \cdot \delta_v^{1;s_\phi}$ is:
\begin{itemize}
	\item contained in $A^0$; or
	\item contained in $A^2 \setminus A^1$ and:
	\begin{itemize}
		\item it satisfies (\tast); or
		\item it does not satisfy (\tast), in which case its \tast-parent $\psi$ necessarily has $\sil(\psi) = \sil(\phi)$, $\dim(\psi) = \dim(\phi)$, $p_\psi = p_\phi$ and $s_\psi = s_\phi -1$.
	\end{itemize}
\end{itemize}
The hyperface $\phi \cdot \delta_v^{1;p_\phi}$ is:
\begin{itemize}
	\item contained in $A^0$; or
	\item contained in $A^2 \setminus A^1$ and:
	\begin{itemize}
		\item it satisfies (\tast); or
		\item it does not satisfy (\tast), in which case its \tast-parent $\psi$ necessarily has $\sil(\psi) = \sil(\phi)$, $\dim(\psi) = \dim(\phi)$ and ${p_\psi} > {p_\phi}$.
	\end{itemize}
\end{itemize}
For any other value of $j$, the hyperface $\phi \cdot \delta_v^{1;j}$ is:
\begin{itemize}
	\item contained in $A^0$; or
	\item contained in $A^2 \setminus A^1$ and it satisfies (\tast).
\end{itemize}

\begin{claim}\label{claim Gray 3}
	The hyperface $\chi = \phi \cdot \delta_v^{1;s_\phi+1}$ is a non-degenerate cell in $A^2 \setminus A^1$ and the minimum upper impurity $\mathcal{U}_\chi$ in $\chi$ is
	\[
	\bigl\langle(i_\phi|k_\phi),(j_\phi|\ell_\phi),(m_\phi|n_\phi),p_\phi-1\bigr\rangle.
	\]
	Consequently $\chi$ does not satisfy {(\tast)}.
\end{claim}
\begin{proof}[Proof of \cref{claim Gray 3}]
The cell $\chi$ is not contained in $A$ since it possesses an impurity 
\[
\bigl\langle(i_\phi|k_\phi),(j_\phi|\ell_\phi),(m_\phi|n_\phi),p_\phi-1\bigr\rangle.
\]
	Now to prove that $\chi$ is not in $A^0$, it suffices to show that $\chi$ is not in $B_x(\delta)$ for any $1 \le x \le b$ and for any hyperface $\delta : \kappa \to \theta_x$, where $B_x(\delta) \subset B$ denotes the image of the map
	\[
	\ttensor_b(\theta_1,\dots,\theta_{x-1},\kappa,\theta_{x+1},\dots,\theta_b) \to \ttensor_b(\theta_1,\dots,\theta_b)
	\]
	induced by $\delta$.
	
	If $\delta$ is either a vertical hyperface or an outer horizontal hyperface, then $\chi$ is in $B_x(\delta)$ if and only if the projection $\pi_x(\chi)$ is in $\delta$, and similarly for $\phi$.
	Since $\phi$ is not in $B_x(\delta)$ and $\pi_x(\phi)$ is a degeneracy of $\pi_x(\chi)$, it follows that $\chi$ is not in $B_x(\delta)$.
	
	Now consider the case where $\delta$ is a $y$-th horizontal hyperface with $1 \le y \le t_x-1$.
	Suppose for contradiction that $\chi$ is in $B_x(\delta)$.
	Then \cref{Gray mono} implies that $(x|y+1)$ is the immediate $\prep$-successor of $(x|y)$ for all $0 \le p \le q$ with $p \neq s_\phi + 1$.
	We show that $(x|y+1)$ must then be the immediate successor of $(x|y)$ with respect to $\pre\;=\;\pre_{s_\phi+1}$ too.
	Note that this is automatic if $(x|y),(x|y+1) \in I_u$ for some $u$ by our construction of $\pre$.
	\begin{itemize}
		\item If $(x|y+1) \prec_{s_\phi} (j_\phi|\ell_\phi)$ then $(x|y),(x|y+1) \in I_1$.
		\item If $(x|y+1) = (j_\phi|\ell_\phi)$ then $\langle(i_\phi|k_\phi),(x|y),(m_\phi|n_\phi),{p_\phi}\rangle$ is a strictly smaller impurity than $\mathcal{U}_\phi$, which contradicts our choice of $\mathcal{U}_\phi$.
		\item Suppose $(j_\phi|\ell_\phi) \presp (x|y) \presp (x|y+1) \presp (i_\phi|k_\phi)$.
		Since $(x|y+1)$ is the immediate $\prepp$-successor of $(x|y)$, it follows that either $(x|y),(x|y+1) \in I_2$ or $(x|y),(x|y+1) \in I_3$.
		\item The case $(i_\phi|k_\phi) \presp (x|y)$ can be treated similarly to the first two cases.
	\end{itemize}
	Therefore $(x|y+1)$ is the immediate $\prep$-successor of $(x|y)$ for all $0 \le p \le q$, including $p = s_\phi+1$.
	By \cref{Gray mono}, this implies that $\phi$ is in $B_x(\delta)$ (for the same $\delta$) which contradicts our assumption that $\phi$ is not in $A^0$.
	
	Finally, to see that $\chi$ is not contained in $A^1$, recall that we have $\presp\;\neq\;\pre\;\neq\;\prepp$ (observed immediately before \cref{claim Gray 1}).
	Since $\phi \cdot \eta_v^{s_\phi+1} = g_{\phi}$ has $\pre$ as the underlying shuffle, it follows from $\pre \LHD \prepp$ that $s_\phi+1 < {p_\phi}$.
	Thus $\chi$ an inner face of $\phi$, which implies that $\chi$ is $\sil$-uncuttable (as $\sil(\chi) = \sil(\phi)$ by \cref{silhouette of face}).
	
	It is now straightforward to check that
	\[
	\mathcal{U}_\chi = \bigl\langle(i_\phi|k_\phi),(j_\phi|\ell_\phi),(m_\phi|n_\phi),p_\phi-1\bigr\rangle.
	\]
	This implies that $s_\chi = s_\phi$ and $g_\chi = g_\phi$.
	Since $\phi$ is non-degenerate, it follows that
	\[
	\chi \cdot \eta_v^{s_\chi+1} = \bigl(\phi \cdot \delta_v^{1;s_\phi+1}\bigr) \cdot \eta_v^{s_\phi+1} = \phi \cdot \eta_v^{s_\phi+2}
	\]
	is not equal to $g_\chi = g_\phi = \phi \cdot \eta_v^{s_\phi+1}$.
	This shows that $\chi$ does not satisfy (\tast).
\end{proof}

This completes the proof of \cref{nu 2}.
\end{proof}

\begin{lemma}\label{nu 3}
	The inclusion $A^2 \incl B$ is in $\celll(\H_v)$.
\end{lemma}
\begin{proof}
	The proof is essentially dual to that of \cref{nu 2}.
\end{proof}

\section{Consequences of associativity}\label{section consequences}
We will discuss two consequences of \cref{hat mu} in this section.

\subsection{$\otimes_a$ is left Quillen}
First, we generalise \cref{binary tensor is left Quillen}.
\begin{theorem}\label{left Quillen}
	The Gray tensor product functor $\tensor_a$ is left Quillen for any $a \ge 1$.
	That is, the Leibniz Gray tensor product
	\[
	\hat \tensor_a(f_1,\dots,f_a)
	\]
	is a monomorphism if each $f_i$ is, and it is a trivial cofibration if moreover some $f_i$ is so.
\end{theorem}
\begin{proof}
	We proceed by induction on $a$.
	The case $a=1$ is trivial, and the case $a=2$ is \cref{binary tensor is left Quillen}.
	
	Let $a \ge 3$ and suppose that $\tensor_{a-1}$ is left Quillen.
	We already know that $\hat \tensor_a$ preserves monomorphisms (\cref{boundary}).
	So let $f_1,\dots,f_a$ be monomorphisms in $\hattheta$, and suppose that $f_i$ is a trivial cofibration for some $i$.
	We wish to show that $\hat \tensor_a(f_1,\dots,f_a) : A \to B$ is a trivial cofibration.
	Note that applying the Leibniz construction of $\tensor_2\bigl(\tensor_{a-1},\tensor_1\bigr)$ to $f_1,\dots,f_a$ yields
	\[
	\hat \tensor_2 \bigl(\hat \tensor_{a-1}(f_1,\dots,f_{a-1}),f_a\bigr)
	\]
	by \cite[Observation 3.21]{Oury} (which may also be found at \cite[Observation 1.4.16]{Gindi:rigidification}), which we denote by $g : X \to Y$.
	This map is a trivial cofibration by the inductive hypothesis and \cref{binary tensor is left Quillen}.
	We can factorise $\hat \tensor_a(f_1,\dots,f_a)$ as:
	\[
	\begin{tikzcd}
	X
	\arrow [d, "\mu", swap]
	\arrow [r, "g"]
	\arrow [dr, phantom, "\ulcorner" very near end] &
	Y
	\arrow [d]
	\arrow [ddr, "\mu", bend left] & \\
	A
	\arrow [r]
	\arrow [drr, "{\hat \tensor_a(f_1,\dots,f_a)}" description, bend right] &
	\cdot
	\arrow [dr, dashed, "h"] & \\
	& & B
	\end{tikzcd}
	\]
	A straightforward analysis of the universal property of the unlabelled object reveals that
	\[
	h = \hat \mu (f_1,\dots,f_a).
	\]
	Thus $\hat \tensor_a(f_1,\dots,f_a)$ is a trivial cofibration by \cref{hat mu}.
\end{proof}

\subsection{The closed structure}
The previous subsection completes the ``monoidal'' part of the story, and now we consider the ``closed'' part.
By construction of the Gray tensor product, the functor
\[
\tensor_{a+1+b}(X^1,\dots,X^a,-,Y^1,\dots,Y^b) : \hattheta \to \hattheta
\]
admits a right adjoint (which preserves fibrations and trivial fibrations by \cref{left Quillen}) for any $a,b \ge 0$ and for any $X^1,\dots,X^a,Y^1,\dots,Y^b \in \hattheta$.
\begin{definition}\label{closed definition}
    We will write
    \[
    (Y^1,\dots,Y^b) \rtri (-) \ltri (X^1,\dots,X^a)
    \]
    or more succinctly
    \[
    \YY \rtri (-) \ltri \XX
    \]
    for this right adjoint.
\end{definition}
\begin{corollary}\label{closed corollary}
Let $X^1,\dots,X^a,Y^1,\dots,Y^b,Z^1,\dots,Z^c,W^1,\dots,W^d \in \hattheta$.
Then there is a natural transformation
\[
\omega : \bigl((\ZZ,\WW)\rtri(-)\ltri(\XX,\YY)\bigr) \longrightarrow
\bigl(\ZZ \rtri(\WW \rtri (-) \ltri \XX) \ltri \YY\bigr).
\]
Moreover, the $A$-component of $\omega$ at any 2-quasi-category $A$ is a trivial fibration.
\end{corollary}
\begin{proof}
The natural transformation $\omega$ is the mate of $\mu$, \emph{i.e.}~the pasting
\[
\begin{tikzpicture}
\foreach \x in {0,4,8}
\filldraw
(\x,0) circle [radius = 1pt]
(\x+2,3) circle [radius = 1pt];

\draw[->] (0.2,0.3) -- (1.8,2.7);
\draw[->] (8.2,0.3) -- (9.8,2.7);
\draw[->] (2.2,2.7) -- (3.8,0.3);
\draw[->] (6.2,2.7) -- (7.8,0.3);
\draw[->] (0.3,0) -- (3.7,0);
\draw[->] (4.3,0) -- (7.7,0);
\draw[->] (2.3,3) -- (5.7,3);
\draw[->] (6.3,3) -- (9.7,3);

\node[scale = 0.8] at (2,-0.2) {$\id$};
\node[scale = 0.8] at (6,-0.2) {$\id$};
\node[scale = 0.8] at (4,3.2) {$\id$};
\node[scale = 0.8] at (8,3.2) {$\id$};

\node[scale = 0.8] at (-0.7,1.5) {$(\ZZ,\WW)\rtri(-)\ltri(\XX,\YY)$};
\node[scale = 0.8] at (11,1.5) {$\ZZ \rtri(\WW \rtri (-) \ltri \XX) \ltri \YY$};
\node[scale = 0.8, fill = white] at (3.5,0.7) {$\tensor(\XX,\YY,-,\ZZ,\WW)$};
\node[scale = 0.8, fill = white] at (6,2.3) {$\tensor(\XX,\tensor(\YY,-,\ZZ),\WW)$};

\draw[->, double] (2,2) -- (2,0.5);
\draw[->, double] (8,2.5) -- (8,1);
\draw[->, double] (5.35,2.025) -- (4.65,0.975);

\node[scale = 0.8] at (1.7,1.2) {$\epsilon$};
\node[scale = 0.8] at (8.3,1.8) {$\eta$};
\node[scale = 0.8, fill = white] at (5,1.5) {$\mu$};
\end{tikzpicture}
\]
where each vertex is $\hattheta$ and the 2-cells $\eta$, $\epsilon$ are the unit and counit of appropriate adjunctions.
Fix a monomorphism $B \incl C$ in $\hattheta$ and a 2-quasi-category $A$.
We wish to show that any commutative square of the form
\[
\begin{tikzcd}[row sep = large, column sep = small]
{B}
\arrow [r]
\arrow [d, hook] &
(\ZZ,\WW)\rtri A \ltri(\XX,\YY)
\arrow [d, "\omega"] \\
{C}
\arrow [r]
\arrow [ur, dashed]&
\ZZ \rtri(\WW \rtri A \ltri \XX) \ltri \YY
\end{tikzcd}
\]
admits a diagonal lift as indicated.
By construction of $\omega$, such a commutative square corresponds to one of the form
\[
\begin{tikzcd}[row sep = large, column sep = small]
{\tensor(\XX,\YY,B,\ZZ,\WW) \coprod_{\tensor(\XX,\tensor(\YY,B,\ZZ),\WW)}\tensor\bigl(\XX,\tensor(\YY,C,\ZZ),\WW\bigr)}
\arrow [r]
\arrow [d, hook] &
A
\arrow [d] \\
\tensor(\XX,\YY,C,\ZZ,\WW)
\arrow [r]
\arrow [ur, dashed] &
{1}
\end{tikzcd}
\]
and moreover either square admits a diagonal lift if and only if the other does.
The latter square indeed admits a lift by \cref{hat mu} since the left vertical map is an instance of $\hat \mu$ evaluated at the monomorphisms
\[
\varnothing \incl X^i,\hspace{10pt}\varnothing \incl Y^j,\hspace{10pt}
B \incl C,\hspace{10pt}
\varnothing \incl Z^k,\hspace{5pt} \text {and} \hspace{5pt} \varnothing \incl W^\ell.
\]
This completes the proof.
\end{proof}

\begin{remark}
	The natural transformation $\omega$ is really part of the functor $\check R_k$ as defined in \cref{model cats prelim} where we take $F$ to be the functor $M$ from \cref{hat mu subsection}.
	Thus \cref{closed corollary} is in fact  \cref{calculus} combined with a special instance of \cref{hat mu}.
	There is a relative version of \cref{closed corollary}, corresponding to the general statement of \cref{hat mu}, which asserts that the Leibniz version of $\omega$ evaluated at a fibration sandwiched between $(c+d)+(a+b)$ many cofibrations is a trivial fibration.
	We leave its precise statement to the reader.
\end{remark}

\appendix

\section{Left Quillen $n$-ary functors}\label{left Quillen proof}
This appendix is devoted to proving \cref{left Quillen characterisation}.
First we recall the definition of left Quillen $n$-ary functor.

\begin{definition}\label{left Quillen definition}
	Let $\M_1,\dots,\M_n,\N$ be model categories.
	An $n$-ary functor
	\[
	F : \M_1 \times \dots \times \M_n \to \N
	\]
	is said to be \emph{left Quillen} if:
	\begin{itemize}
		\item[(1)] for any $1 \le k \le n$ and for any choice of objects $X_i \in \M_i$ for $i \neq k$, the functor
		\[
		F(X_1,\dots,X_{k-1},-,X_{k+1},\dots,X_n) : \M_k \to \N
		\]
		admits a right adjoint; and
		\item[(2)] the Leibniz construction $\hat F(f_1,\dots,f_n)$ is a cofibration for any cofibrations $f_1,\dots,f_n$, and it is moreover trivial if $f_k$ is so for some $1 \le k \le n$.
	\end{itemize}
\end{definition}

\subsection{Lifting between Leibniz constructions}\label{model cats prelim}
Let $F : \M_1 \times \dots \times \M_n \to \N$ be an $n$-ary functor and fix $1 \le k \le n$.
Suppose that $\M_k$ has finite connected limits and $\N$ has finite connected colimits.
Suppose further that, for any choice of $X_i \in \M_i$ for $i \neq k$, the functor
\[
F(X_1,\dots,X_{k-1},-,X_{k+1},\dots,X_n) : \M_k \to \N
\]
admits a right adjoint.
Then these right adjoints assemble into a single functor
\[
R_k : \M_1\op \times \dots \times \M_{k-1}\op \times \N \times \M_{k+1}\op \times \dots \times \M_n\op \to \M_k.
\]
(See \cite[\textsection1.10]{Kelly:basic} for a proof.)
In this situation, we write $\check R_k$ for the Leibniz construction applied to
\[
R_k\op : \M_1 \times \dots \times \M_{k-1} \times \N\op \times \M_{k+1} \times \dots \times \M_n \to \M_k\op
\]
so that the codomain of $\check R_k(f_1,\dots,f_{k-1},g,f_{k+1},\dots,f_n)$ is the limit of a cube-like-shaped diagram in $\M_k$.

The following proposition is well known.
For instance, a binary variant can be found at \cite[Proposition 7.6]{JT}.

\begin{proposition}\label{calculus}
	Let $F$ and $R_k$ be as above.
	Let $f_i : X_i^0 \to X_i^1$ be a morphism in $\M_i$ for $1 \le i \le n$, and let $g : Y^0 \to Y^1$ be a morphism in $\N$.
	Then $\hat F(f_1,\dots,f_n)$ has the left lifting property with respect to $g$ if and only if $f_k$ has the left lifting property with respect to $\check R_k(f_1,\dots,f_{k-1},g,f_{k+1},\dots,f_n)$.
\end{proposition}
\begin{proof}
	Let $G : \two^{n+1} \to \Set$ be the functor whose object part is given by
	\[
	G(\epsilon_1,\dots,\epsilon_n,\epsilon) = \N\bigl(F\bigl(X_1^{1-\epsilon_1},\dots,X_n^{1-\epsilon_n}\bigr),Y^{\epsilon}\bigr)
	\]
	and whose morphism part is the obvious one.
	Denote by $I$ the inclusion of the full subcategory of $\two^{n+1}$ spanned by all non-initial objects.
	Then $G$ defines a cone over the diagram $GI$, so we obtain an induced morphism
	\[
	\N\bigl(F\bigl(X_1^1,\dots,X_n^1\bigr),Y^0\bigr) \to \lim GI.
	\]
	One can check that $\hat F(f_1,\dots,f_n)$ has the left lifting property with respect to $g$ if and only if this induced morphism is a surjection.
	
	Now by definition of $R_k$, the functor $G$ is naturally isomorphic to $G'$ given by
	\[
	G'(\epsilon_1,\dots,\epsilon_n,\epsilon) = \M_k\bigl(X_k^{1-\epsilon_k},R_k\bigl(X_1^{1-\epsilon_1},\dots,X_{k-1}^{1-\epsilon_{k-1}},Y^{\epsilon},X_{k+1}^{1-\epsilon_{k+1}}, \dots, X_n^{1-\epsilon_n}\bigr)\bigr).
	\]
	One can check that $f_k$ has the left lifting property with respect to 
	\[
	\check R_k(f_1,\dots,f_{k-1},g,f_{k+1},\dots,f_n)
	\]
	if and only if the map
	\[
	\M_k\bigl(X_k^1,R_k\bigl(X_1^1,\dots,X_{k-1}^1,Y^0,X_{k+1}^1,\dots,X_n^1\bigr)\bigr) \to \lim G'I
	\]
	induced by $G'$ (regarded as a cone over $G'I$) is a surjection.
	The desired equivalence now follows.
\end{proof}

\subsection{Pseudo-generating sets}
For many examples of model categories $\M$, we can explicitly describe a generating set $\I_\M$ of cofibrations (in the sense that the trivial fibrations are precisely those maps with the right lifting property with respect to $\I_\M$) but not of trivial cofibrations.
Instead, we often have an explicit  \emph{pseudo-generating} set in the following sense.

\begin{definition}[{\cite[\textsection 9.9]{Simpson}}]
	A set $\J_\M$ of trivial cofibrations in a model category $\M$ is said to be \emph{pseudo-generating} if, for any map in $\M$ with a fibrant codomain, being a fibration is equivalent to having the right lifting property with respect to the maps in $\J_\M$.
\end{definition}

We show that such pseudo-generating sets suffice for detecting left Quillen functors of arbitrary arity $n$.
The special cases with $n=1$ and $n=2$ are respectively consequences of \cite[Lemma 7.14]{JT} and \cite[Lemma B.0.12]{Henry:weak}.

\begin{proposition}\label{pseudo-generating}
	Let $\M_i$ be a model category with a generating set $\I_i$ of cofibrations and a pseudo-generating set $\J_i$ of trivial cofibrations for $1 \le i \le n$.
	Let
	\[
	F : \M_1 \times \dots \times \M_n \to \N
	\]
	be a functor into another model category $\N$ satisfying \cref{left Quillen definition}(1).
	Then $F$ is left Quillen if and only if:
	\begin{itemize}
		\item [(i)]
		each map in $\hat F(\I_1,\dots,\I_n)$ is a cofibration; and
		\item [(ii)]
		each map in $\hat F(\I_1,\dots,\I_{k-1},\J_k,\I_{k+1},\dots,\I_n)$ is a trivial cofibration for any $1 \le k \le n$.
	\end{itemize}
\end{proposition}
\begin{proof}
	The ``only if'' part is clear as both (i) and (ii) are instances of \cref{left Quillen definition}(2).
	
	For the ``if'' direction, suppose that $F$ satisfies (i) and (ii).
	Then it follows from \cref{celll} that $F(f_1,\dots,f_n)$ is a cofibration for any cofibrations $f_i$ in $\M_i$.
	Thus $F$ satisfies the first part of \cref{left Quillen definition}(2).
	
	Now fix $1 \le k \le n$.
	Let $f_i : X_i^0 \to X_i^1$ be a cofibration in $\M_i$ for $1 \le i \le n$ and suppose that $f_k$ is trivial.
	We wish to show that $\hat F(\ff)$, which we already know to be a cofibration, is trivial.
	By \cite[Lemma 7.14]{JT}, $\hat F(\ff)$ is trivial if and only if it has the left lifting property with respect to all fibrations between fibrant objects.
	By \cref{calculus}, the latter is equivalent to the statement that $f_k$ has the left lifting property with respect to $\check R_k (\ff\{g\})$ for any fibration $g : Y^0 \to Y^1$ between fibrant objects in $\N$, where $\check R_k$ is defined as in the previous subsection and
	\[
	\ff\{g\} = (f_1,\dots,f_{k-1},g,f_{k+1},\dots,f_n)
	\]
	(see \cref{things}).
	Thus it suffices to show that $\check R_k (\ff\{g\})$ is a fibration between fibrant objects whenever $g$ is so.
	By (ii) and \cref{calculus}, this reduces to checking that the codomain of $\check R_k (\ff\{g\})$ is fibrant whenever $g$ is a fibration between fibrant objects.
	
	We proceed by induction on the cardinality of the union
	\[
	\{i : i \neq k,\hspace{5pt} X_i^0 \neq 0\} \cup \{* : Y^1 \neq 1\}
	\]
	where $0$ and $1$ denote the initial and terminal objects in appropriate categories.
	(The second set simply contributes $1$ to the cardinality if $Y^1 \neq 1$ and contributes $0$ if $Y^1 = 1$.)
	The base case is trivial since $X_i^0 = 0$ for all $i \neq k$ and $Y^1 = 1$ would imply that the codomain of $\check R_k (\ff\{g\})$ is the terminal object in $\M_k$.
	
	For the inductive step, let $G : \two^n \to \M_k$ be the functor given by
	\[
	G(\eepsilon) = R_k\bigl(X_1^{1-\epsilon_1},\dots, X_{k-1}^{1-\epsilon_{k-1}}, Y^{\epsilon_k}, X_{k+1}^{1-\epsilon_{k+1}}, \dots, X_n^{1-\epsilon_n}\bigr)
	\]
	and let $I : \C \incl \two^n$ denote the inclusion of the full subcategory spanned by all non-initial objects.
	Then the codomain of $\check R_k (\ff\{g\})$ is the limit of $GI$.
	Observe that $\C$ admits a Reedy structure with $\deg(\eepsilon) = n-\sum\eepsilon$ such that all maps are degree-lowering.
	Since there is no degree-raising map in $\C$, the diagonal functor $\M_k \to [\C,\M_k]$ is left Quillen.
	Thus it remains to show that $GI$ is Reedy fibrant.
	
	Fix an object $\eepsilon \in \C$.
	We wish to show that the $\eepsilon$-th matching map for $GI$ is a fibration.
	Observe that this matching map is precisely $\check R_k(\ffd\{g'\})$ where
	\[
	f'_i = \left\{\begin{array}{cl}
	f_i & \text {if $\epsilon_i = 0$,}\\
	0 \to X_i^0 & \text {if $\epsilon_i = 1$}
	\end{array}\right.
	\]
	for each $i \neq k$ and
	\[
	g' = \left\{\begin{array}{cl}
	g & \text {if $\epsilon = 0$,}\\
	Y^1 \to 1 & \text {if $\epsilon = 1$.}
	\end{array}\right.
	\]
	Since $\eepsilon \in \C$, we can choose $1 \le i \le n$ such that $\epsilon_i=1$.
	If we have either:
		\begin{itemize}
			\item $i \neq k$ and $X_i^0=0$; or
			\item $i=k$ and $Y^1=1$
		\end{itemize}
	then $\check R_k(\ffd\{g'\})$ is invertible.
	So suppose that either:
	\begin{itemize}
		\item $i \neq k$ and $X_i^0 \neq 0$; or
		\item $i=k$ and $Y^1 \neq 1$.
	\end{itemize}
Then it follows by the inductive hypothesis that the codomain of $\check R_k(\ffd\{g'\})$ is fibrant.
Moreover, $\check R_k(\ffd\{g'\})$ has the right lifting property with respect to all maps in $\J_k$ by (ii) and \cref{calculus}.
Since $\J_k$ is pseudo-generating, it follows that $\check R_k(\ffd\{g'\})$ is a fibration.
	This completes the proof.
\end{proof}

\subsection{Application to Ara's model structure}
We now prove \cref{left Quillen characterisation}.
We repeat the theorem for the reader's convenience.
Recall the sets $\I$ of boundary inclusions (\cref{I}) and $\J$ of inner horn inclusions and equivalence extensions (\cref{J}).

\begin{theorem*}
	Let
	\[
	F: \hattheta \times \dots \times \hattheta \to \mathscr{M}
	\]
	be an $n$-ary functor into a model category $\mathscr{M}$.
	Suppose that for any $1 \le k \le n$ and for any choice of objects $X_i \in \hattheta$ for $i \neq k$, the functor
	\[
	F(X_1,\dots,X_{k-1},-,X_{k+1},\dots,X_n) : \hattheta \to \M
	\]
	admits a right adjoint.
	Then $F$ is left Quillen if and only if:
	\begin{itemize}
		\item[(i)] each map in $\hat F(\I,\dots,\I)$ is a cofibration; and
		\item[(ii)] each map in $\hat F(\I,\dots,\I,\J,\I,\dots,\I)$ is a trivial cofibration for any position of $\J$.
	\end{itemize}
	In particular, each map in $\J$ is a trivial cofibration.
\end{theorem*}

\begin{remark}
	The map $[\id;e]$ is, in fact, the lowest dimensional member of the infinite family of vertical equivalence extensions that Oury constructed \cite[Definition 3.83]{Oury}.
	We use the notation from \cite[Definition 2.29]{Maehara:horns} and denote a general vertical equivalence extension by
	\[
	\Psi^k\nq \incl \Phi^k\nq
	\]
	where $\nq \in \cell$ and $1 \le k \le n$ satisfy $q_k = 0$.
	In particular, the map $[\id;e]$ is (isomorphic to) the inclusion $\Psi^1[1;0] \incl \Phi^1[1;0]$.
	See the aforementioned papers for the definition and more on these general equivalence extensions.
\end{remark}

\begin{proof}
	Theorem 6.1 of \cite{Maehara:horns} states that the union $\J_{\hattheta} = \J \cup \E_v$ is a pseudo-generating set of trivial cofibrations for Ara's model structure where $\E_v$ is the set of all vertical equivalence extensions.
	(In particular, it is proven that each map in $\J_{\hattheta}$ is a trivial cofibration.)
	Thus by \cref{pseudo-generating}, it suffices to show that each map in $\hat F(\I,\dots,\I,\J_{\hattheta},\I,\dots,\I)$ is a trivial cofibration.

	Let $[1;0] \neq \nq \in \cell$ and $1 \le k \le n$ be such that $q_k = 0$.
	Since $F$ preserves colimits in each variable, \cite[Observation 3.8]{Oury} (which can also be found at \cite[Lemma 1.4.12]{Gindi:rigidification}) implies that any map of the form $\hat F(\ff\{hg\})$ may be obtained as a composite of $\hat F(\ff\{h\})$ and a pushout of $\hat F(\ff\{g\})$.
	In particular, we may apply this fact to the vertical equivalence extension $h : \Psi^k\nq \incl \Phi^k\nq$ and the monomorphism $g : \cell\nq \to \Psi^k\nq$ described after Lemma 3.11 of \cite{Maehara:horns}.
	Since \cite[Lemmas 3.12-16]{Maehara:horns} show that both $g$ and $hg$ belong to the class
	\[
	\celll\Bigl(\H_h \cup \H_v \cup \bigl\{\Psi^\ell\mp \incl \Phi^\ell\mp : \dim\mp < \dim\nq\bigr\}\Bigr),
	\]
	it now follows by the 2-out-of-3 property and induction on $\dim\nq$ that each map in $\hat F(\I,\dots,\I,\E_v,\I,\dots,\I)$ is a trivial cofibration.
	This completes the proof.
\end{proof}

\section{Braid monoids with zero}\label{braid}
In this appendix, we complete the proof of \cref{tensor of bars} using the \emph{braid monoids with zero}.
A special case of \cref{tensor of bars} where $\theta_i = [1;0]$ for each $i$ was first proved by Gray \cite[Theorem 2.2]{Gray:coherence} using the \emph{braid groups}.
Our argument here is a minor modification of Street's proof of that same special case \cite[Theorem 1]{Street:Gray}.
\begin{definition}
A \emph{monoid with zero} is a monoid $M$ with a distinguished element $0 \in M$ such that
\begin{equation}\label{zero}
x0 = 0 = 0x
\end{equation}
for all $x \in M$.
\end{definition}
\begin{definition}
For any $n \ge 1$, let $\BB_n$ be the monoid with zero presented by generators $\beta_1,\beta_2,\dots,\beta_{n-1}$ subject to the relations
\begin{alignat}{2}
\beta_q\beta_p &= \beta_p\beta_q &\quad&\text {for $p+1 < q$,}\label{far crossings}\\
\beta_{p+1}\beta_p\beta_{p+1} &= \beta_p\beta_{p+1}\beta_p, &&\text {and}\label{consecutive crossings}\\
\beta_p\beta_p &= 0. &&\label{square to zero}
\end{alignat}
\end{definition}

It is called the \emph{braid monoid with zero} since \cref{far crossings,consecutive crossings} are precisely the relations in the standard presentation of the \emph{braid group}.
The elements of $\BB_n$ can be thus visualised as certain braids on $n$ strands where each generator $\beta_p$ crosses the $p$-th and the $(p+1)$-th strands:
\[
\begin{tikzpicture}[scale = 1.4]
\draw (0,0) -- (0,1) (1,1) -- (1,0) (2.5,0) -- (2.5,1) (3.5,0) -- (3.5,1);
\foreach \x in {0.4,0.5,0.6,2.9,3,3.1}
\filldraw (\x,0.5) circle [radius = 0.2pt]
(\x, 1.2) circle [radius = 0.2pt];
\node[scale = 0.7] at (0,1.2) {$1$};
\node[scale = 0.7] at (1,1.2) {$p-1$};
\node[scale = 0.7] at (1.5,1.2) {$p$};
\node[scale = 0.7] at (2,1.2) {$p+1$};
\node[scale = 0.7] at (2.5,1.2) {$p+2$};
\node[scale = 0.7] at (3.5,1.2) {$n$};
\draw (1.5,0) -- (2,1);
\draw[white, line width = 3] (1.5,1) -- (2,0);
\draw (1.5,1) -- (2,0);
\end{tikzpicture}
\]
and the composition is given by vertically stacking the braids.
Then omitting the irrelevant strands, \cref{far crossings,consecutive crossings,square to zero} look like
\[
\begin{tikzpicture}[baseline = 12]
\draw (0,1) -- (0,0.5) (0.5,1) -- (0.5,0.5) -- (0,0) (1.5,0.5) -- (1.5,0) (1.5,1) -- (1,0.5) -- (1,0);
\draw[white, line width = 3] (0,0.5) -- (0.5,0) (1,1) -- (1.5,0.5);
\draw (0,0.5) -- (0.5,0) (1,1) -- (1.5,0.5);
\end{tikzpicture}
\hspace {10pt} = \hspace{10pt}
\begin{tikzpicture}[baseline = 12]
\draw (0.5,0.5) -- (0.5,0) (0.5,1) -- (0,0.5) -- (0,0) (1,1) -- (1,0.5) (1.5,1) -- (1.5,0.5) -- (1,0);
\draw[white, line width = 3] (0,1) -- (0.5,0.5) (1,0.5) -- (1.5,0);
\draw (0,1) -- (0.5,0.5) (1,0.5) -- (1.5,0);
\end{tikzpicture}
\hspace{5pt},
\]
\[
\begin{tikzpicture}[baseline = 18]
\draw (0,1.5) -- (1,0.5) -- (1,0) (0.5,1.5) -- (0,1) -- (0,0.5) -- (0.5,0) (1,1.5) -- (1,1) -- (0,0);
\draw[white, line width = 3] (0,1.5) -- (1,0.5) (0,0.5) -- (0.5,0);
\draw (0,1.5) -- (1,0.5) (0,0.5) -- (0.5,0);
\end{tikzpicture}
\hspace {10pt} = \hspace{10pt}
\begin{tikzpicture}[baseline = 18]
\draw (0,1.5) -- (0,1) (1,1) -- (1,0.5) -- (0.5,0) (1,1.5) -- (0,0.5) -- (0,0);
\draw[white, line width = 3] (0,1) -- (1,0) (0.5,1.5) -- (1,1);
\draw (0,1) -- (1,0) (0.5,1.5) -- (1,1);
\end{tikzpicture}
\hspace{20pt} \text {and} \hspace{20pt}
\begin{tikzpicture}[baseline = 12]
\draw (0.5,1) -- (0,0.5) (0.5,0.5) -- (0,0);
\draw[white, line width = 3] (0,1) -- (0.5,0.5) (0,0.5) -- (0.5,0);
\draw (0,1) -- (0.5,0.5) -- (0.4,0.4) (0.1,0.6) -- (0,0.5) -- (0.5,0);
\end{tikzpicture}
\hspace {10pt} = \hspace{10pt} 0
\]
respectively.
For $p+1 \ge q$, let
\[
\beta_{p,q} \defeq \beta_p\beta_{p-1}\dots \beta_{q}
\]
so that it looks like:
\[
\begin{tikzpicture}[scale = 1.4]
\draw (2,1) -- (0,0);
\draw[white, line width = 3] (0,1) -- (0.5,0) (0.5,1) -- (1,0) (1.5,1) -- (2,0);
\draw (0,1) -- (0.5,0) (0.5,1) -- (1,0) (1.5,1) -- (2,0);
\node[scale = 0.7] at (0,1.2) {$q$};
\node[scale = 0.7] at (0.5,1.2) {$q+1$};
\node[scale = 0.7] at (1.5,1.2) {$p$};
\node[scale = 0.7] at (2,1.2) {$p+1$};
\foreach \x in {0.9,1,1.1}
\filldraw (\x+0.25,0.5) circle [radius = 0.2pt]
(\x, 1.2) circle [radius = 0.2pt];
\end{tikzpicture}
\]
(We interpret $\beta_{p,p+1}$ to be the identity.)

The following theorem describes a normal form for non-zero elements of $\BB_n$.
\begin{theorem}\label{normal form}
	Any non-zero element $x \in \BB_n$ can be written uniquely as a product of the form
	\[
	x = \beta_{1,q_1}\beta_{2,q_2}\dots\beta_{{n-1},p_{n-1}}
	\]
	where $p+1 \ge q_p$ for each $p$.
	Conversely, $\beta_{1,q_1}\beta_{2,q_2}\dots\beta_{{n-1},p_{n-1}} \neq 0$ for any $p+1 \ge q_p$.
\end{theorem}

\begin{remark}
This normal form is reminiscent of the sorting algorithm called \emph{insertion sort} in computer science.
At the $p$-th stage, $\beta_{p,q_p}$ takes the $(p+1)$-th strand at the top and inserts it to the correct position relative to the previously sorted strands.
\end{remark}

\begin{proof}
	We will summarise the proof in \cite[\textsection 6]{Eilenberg;Street:rewrite} and fill in the gaps therein.
	We consider the \emph{rewrite system} on the alphabet $\{\beta_1,\dots,\beta_{n-1},0\}$ given by the following rewrite rules:
	\begingroup
	\renewcommand{\arraystretch}{1.3}
	\[
	\begin{array}{rcccl}
	t_{p,q} : & \beta_q\beta_p & \rightsquigarrow & \beta_p\beta_q & \text{for $p+1<q$}\\
	r_{p,q} : & \beta_{p,q}\beta_p & \rightsquigarrow & \beta_{p-1}\beta_{p,q} & \text{for $p > q$}\\
	s_p : & \beta_p\beta_p & \rightsquigarrow & 0 & \\
	y_p : & \beta_p0 & \rightsquigarrow & 0 & \\
	z_p : & 0\beta_p & \rightsquigarrow & 0 & \\
	0 : & 00 & \rightsquigarrow & 0 & 
	\end{array}
	\]
	\endgroup
	That is, we consider the process of rewriting a given string in $\{\beta_1,\dots,\beta_{n-1},0\}$ by applying these rules to its substrings.
	If a string $u$ can be rewritten to another string $v$, we say $v$ is a \emph{rewriting} of $u$.
	
	Note that these rewrite rules subsume the relations in the presentation of $\BB_n$ as \cref{consecutive crossings} corresponds to $r_{p+1,p}$.
	Conversely, none of the rules affects the element of $\BB_n$ that the string represents;
	the rule $r_{p,q}$ looks like
	
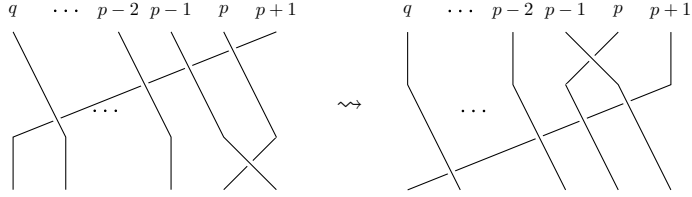
\[
\begin{tikzpicture}[scale = 1.4, baseline = 10]
\draw (2.5,1) -- (0,0) -- (0,-0.5) (2.5,0) -- (2,-0.5);
\draw[white, line width = 3] (0,1) -- (0.5,0) (1,1) -- (1.5,0) (1.5,1) -- (2,0) -- (2.5,-0.5) (2,1) -- (2.5,0);
\draw (0,1) -- (0.5,0) -- (0.5,-0.5) (1,1) -- (1.5,0) -- (1.5,-0.5) (1.5,1) -- (2,0) -- (2.5,-0.5) (2,1) -- (2.5,0);
\node[scale = 0.7] at (0,1.2) {$q$};
\node[scale = 0.7] at (1,1.2) {$p-2$};
\node[scale = 0.7] at (1.5,1.2) {$p-1$};
\node[scale = 0.7] at (2,1.2) {$p$};
\node[scale = 0.7] at (2.5,1.2) {$p+1$};
\foreach \x in {0.4,0.5,0.6}
\filldraw
(\x+0.375,0.25) circle [radius = 0.2pt]
(\x, 1.2) circle [radius = 0.2pt];
\end{tikzpicture}
\hspace{10pt}
\rightsquigarrow
\hspace{10pt}
\begin{tikzpicture}[scale = 1.4, baseline = 30]
\draw (2.5,1.5) -- (2.5,1) -- (0,0) (2,1.5) -- (1.5,1);
\draw[white, line width = 3] (0,1) -- (0.5,0) (1,1) -- (1.5,0) (1.5,1) -- (2,0) (2,1) -- (2.5,0) (1.5,1.5) -- (2,1);
\draw (0,1.5) -- (0,1) -- (0.5,0) (1,1.5) -- (1,1) -- (1.5,0) (1.5,1) -- (2,0) (1.5,1.5) -- (2,1) -- (2.5,0);
\node[scale = 0.7] at (0,1.7) {$q$};
\node[scale = 0.7] at (1,1.7) {$p-2$};
\node[scale = 0.7] at (1.5,1.7) {$p-1$};
\node[scale = 0.7] at (2,1.7) {$p$};
\node[scale = 0.7] at (2.5,1.7) {$p+1$};
\foreach \x in {0.4,0.5,0.6}
\filldraw
(\x+0.125,0.75) circle [radius = 0.2pt]
(\x, 1.7) circle [radius = 0.2pt];
\end{tikzpicture}
\]
and it follows from \cref{far crossings,consecutive crossings} that the two sides are equal.

	First we wish to show that this rewrite system is \emph{bounded}, \emph{i.e.}~for any given (fixed) string, there is an upper bound on how many times the rewrite rules may be applied.
	This is done by assigning a natural number to each string in such a way that applying any of these rules decreases that number.
	Given a string $\beta_{p_1}\dots\beta_{p_m}$, where we interpret $\beta_0$ to mean $0$, we assign the following natural number:
	\[
	\rho(\beta_{p_1}\dots\beta_{p_m}) \defeq m + \sum_{1 \le i \le m}p_i^2 + \Bigl|\bigl\{(i,j)~|~i<j \text{~and~}0 < p_j < p_i\bigr\}\Bigr|
	\]
	The original formula in \cite{Eilenberg;Street:rewrite} does not have the exponent $2$ in the second term, but this exponent is necessary for the rewrite rule $r_{p,q}$ to decrease the value of $\rho$.
	(The rule $r_{p,q}$ decreases the second term of $\rho$ by $p^2-(p-1)^2 = 2p-1$ and increases the third term by $p-q-1$.
	Without the exponent $2$, it only decreases the second term by $1$.)
	
	Next we need to show that this rewrite system is \emph{locally confluent}, \emph{i.e.}~if a given string admits two (possibly overlapping) substrings to each of which some rewrite rule can be applied, then the two resulting strings have a common rewriting.
	It suffices to check certain special cases (see \cite[Proposition 5.2]{Eilenberg;Street:rewrite}), and most of these cases are checked in \cite[Proposition 6.2]{Eilenberg;Street:rewrite}.
	There are a few cases missing in their proof (more precisely, their analysis of the pair $(r,t)$ assumes $j = p$), but they are no more difficult than the other cases.
	
	These properties of the rewrite system imply that each string admits a unique \emph{normal form}, \emph{i.e.}~a rewriting that admits no further rewritings.
	It remains to check that a string is in its normal form if and only if it is either $0$ or of the form described in the theorem.
	This is done in \cite[Theorem 6.3]{Eilenberg;Street:rewrite}.
\end{proof}

Recall that the symmetric group $\mathbb{S}_n$ on $n$ letters $1, \dots, n$ may be presented by generators $\beta_1,\beta_2,\dots,\beta_{n-1}$ subject to \cref{far crossings,consecutive crossings} and $\beta_p\beta_p= 1$.
Hence we can define a function
\[
\sigma_{(-)}: \BB_n \setminus \{0\} \to \mathbb{S}_n
\]
by assigning the transposition of $p$ and $p+1$ to $\beta_p$ and then extending this assignation according to $\sigma_{xy} = \sigma_x \circ \sigma_y$.
Graphically, $\sigma_x(p) = q$ if the braid $x$ takes the strand in the $p$-th position at the bottom to the $q$-th position at the top.

\begin{corollary}\label{permutation determines braid}
	The function $\sigma_{(-)}$ is injective.
\end{corollary}
\begin{proof}
	Observe that if
	\[
	x = \beta_{1,q_1}\beta_{2,q_2}\dots\beta_{{n-1},p_{n-1}}
	\]
	then $q_p$ is precisely the number of $1 \le r \le p+1$ such that $\sigma^{-1}_x(r) \le \sigma^{-1}_x(p+1)$.
	This shows that we can recover (the normal form of) $x$ from $\sigma_x$.
\end{proof}
\begin{proof}[Proof of \cref{tensor of bars} continued]
	It remains to prove that the 2-functor
	\[
	F : \ttensor_a(\theta_1,\dots,\theta_a) \to \T
	\]
	is locally faithful.
	Since $\T$ is poset-enriched, this is equivalent to showing that $\ttensor_a(\theta_1,\dots,\theta_a)$ is also poset-enriched.
	
	Fix two objects $\ss,\tt$ and let $n = |S(\ss,\tt)|$.
	In this proof, we identify each object $\pre$ in the hom-category $\ttensor_a(\theta_1,\dots,\theta_a)(\ss,\tt)$ with the unique order-preserving bijection
	\[
	f : \bigl(\{1,\dots,n\}, \le \bigr) \to \bigl(S(\ss,\tt),\pre\bigr).
	\]
	We define an action of the monoid (with zero) $\BB_n$ on the set
	\[
	\ob\bigl(\ttensor_a(\theta_1,\dots,\theta_a)(\ss,\tt) \bigr)\cup \{*\}
	\]
	as follows.
	The zero element $0 \in \BB_n$ sends everything to $*$, and $*$ is fixed by every element in $\BB_n$.
	Given a bijection $f$ as above and $1 \le p < n$, we define:
	\[
	f \cdot \beta_p \defeq
	\left\{\begin{array}{cl}
	f \circ \sigma_{\beta_p} & \text{if $\pi_1 \circ f(p) > \pi_1 \circ f(p+1)$,}\\
	* & \text{otherwise}
	\end{array}\right.
	\]
	where the projection $\pi_1 : S(\ss,\tt) \to \{1,\dots,a\}$ sends each $(i|k)$ to $i$.
	\begin{claim*}
		This specification indeed extends to an action of $\BB_n$.
		Moreover, for any non-zero element $x \in \BB_n$ and any bijection $f$ as above, either $f \cdot x = f \circ \sigma_x$ or $f \cdot x = *$.
	\end{claim*}
	\begin{proof}[Proof of the claim]
		Assuming the first part, the second part follows from the equation $\sigma_{xy} = \sigma_x \circ \sigma_y$.
		It suffices to check that, for each of \cref{far crossings,consecutive crossings,square to zero}, (the action determined by) either side sends a given bijection $f$ as above to $*$ if and only if the other side does.
		
		For any bijection $f$ as above and any $p+1 < q$, the following are equivalent:
		\begin{itemize}
			\item $f \cdot \beta_q \neq *$ and $(f \circ \sigma_{\beta_q}) \cdot \beta_p \neq *$;
			\item $\pi_1 \circ f(p) > \pi_1 \circ f(p+1)$ and $\pi_1 \circ f(q) > \pi_1 \circ f(q+1)$; and
			\item $f \cdot \beta_p\neq *$ and $(f \circ \sigma_{\beta_p}) \cdot \beta_q \neq *$.
		\end{itemize}
		Thus the two sides of \cref{far crossings} determine the same action.
		A similar analysis can be done for \cref{consecutive crossings}, and the action of any $\beta_p$ applied twice sends any $f$ to $*$.
		This completes the proof of the claim.
	\end{proof}
	
	If $f(p) = (j|\ell)$, $f(p+1) = (i|k)$ and $j > i$ then there is a morphism $f \to f \circ \sigma_{\beta_p}$ in the hom-category $\ttensor_a(\theta_1, \dots, \theta_a)(\ss,\tt)$ which looks like
	\begin{equation}\label{whisker}
	\begin{tikzpicture}[baseline = -2]
	\node at (-0.5,0) {$(k-1,\ell-1)$};
	\node at (1.5,-1) {$(k,\ell-1)$};
	\node at (1.5,1) {$(k-1,\ell)$};
	\node at (3,0) {$(k,\ell)$};
	\draw[->] (0.3,0.3) -- (0.7,0.7);
	\draw[->] (0.3,-0.3) -- (0.7,-0.7);
	\draw[->] (2.3,0.7) -- (2.7,0.3);
	\draw[->] (2.3,-0.7) -- (2.7,-0.3);
	\draw[->] (-2.2,0) -- (-1.7,0);
	\draw[->] (3.5,0) -- (4,0);
	\foreach \x in {-2.5,-2.7,-2.9,4.3,4.5,4.7}
	\filldraw (\x, 0) circle [radius = 0.2pt];
	\draw[->] (-3.7,0) -- (-3.2,0);
	\draw[->] (5,0) -- (5.5,0);
	\node at (-4,0) {$\ss$};
	\node at (5.8,0) {$\tt$};
	\draw[->, double] (1.5,0.6) -- (1.5,-0.6);
	\end{tikzpicture}
	\end{equation}
	where we are suppressing all but the $i$-th and $j$-th coordinates of the middle four objects.
	We abuse the notation and call this morphism $\beta_p$.
	Since the hom-category $\ttensor_a(\theta_1,\dots,\theta_a)(\ss,\tt)$ is generated by the morphisms of the form (\ref{whisker}), it follows that any morphism $f \to g$ admits a factorisation of the form
	\begin{equation}\label{braid factorisation}
	\begin{tikzcd}
	f
	\arrow [r, "\beta_{p_1}"] &
	f \circ \sigma_{\beta_{p_1}}
	\arrow [r, "\beta_{p_2}"] &
	\dots
	\arrow [r, "\beta_{p_r}"] &
	f \circ \sigma_{\beta_1\dots\beta_{p_r}}.
	\end{tikzcd}
	\end{equation}
	We wish to show that the word $\beta_{p_1} \dots \beta_{p_r}$ determines a non-zero element in $\BB_n$.
	It follows from the proof of \cref{normal form} that this word can be reduced either to $0$ or to a normal form specified in the theorem by successively applying \cref{zero,far crossings,consecutive crossings,square to zero}.
	We claim that this reduction process may be reproduced in $\ttensor_a(\theta_1,\dots,\theta_a)(\ss,\tt)$ with $\beta_p$'s regarded as morphisms (and concatenation interpreted as composition in reverse order).
	Indeed, \cref{far crossings} corresponds to the interchange law for a 2-category and \cref{consecutive crossings} corresponds to the commutativity of the cube
	\[
	\begin{tikzpicture}[baseline = -2]
	\node[scale = 0.7] at (150:2) {$(m-1,k-1,\ell-1)$};
	\node[scale = 0.7] at (90:2) {$(m,k-1,\ell-1)$};
	\node[scale = 0.7] at (30:2) {$(m,k,\ell-1)$};
	\node[scale = 0.7] at (-30:2) {$(m,k,\ell)$};
	\node[scale = 0.7] at (-90:2) {$(m-1,k,\ell)$};
	\node[scale = 0.7] at (-150:2) {$(m-1,k-1,\ell)$};
	
	\draw[->] (-135:1.8) --(-105:1.8);
	\draw[->] (-75:1.8) -- (-45:1.8);
	\draw[->] (135:1.8) --(105:1.8);
	\draw[->] (75:1.8) -- (45:1.8);
	\draw[->] (-1.7,0.7) -- (-1.7,-0.7);
	\draw[->] (1.7,0.7) -- (1.7,-0.7);
	
	\node[scale = 0.7] at (0,0) {$(m-1,k,\ell-1)$};
	
	\draw[->] (150:1.5) -- (150:0.5);
	\draw[->] (0,-0.3) -- (0,-1.5);
	\draw[->] (30:0.5) -- (30:1.5);
	
	\draw[->,double] (-150:1) + (-0.2,-0.2) --+ (0.2,0.2);
	\draw[->,double] (90:1) + (-0.2,-0.2) --+ (0.2,0.2);
	\draw[->,double] (-30:1) + (-0.2,-0.2) --+ (0.2,0.2);
	\end{tikzpicture}
	\hspace{10pt} = \hspace{10pt}
	\begin{tikzpicture}[baseline = -2]
	\node[scale = 0.7] at (150:2) {$(m-1,k-1,\ell-1)$};
	\node[scale = 0.7] at (90:2) {$(m,k-1,\ell-1)$};
	\node[scale = 0.7] at (30:2) {$(m,k,\ell-1)$};
	\node[scale = 0.7] at (-30:2) {$(m,k,\ell)$};
	\node[scale = 0.7] at (-90:2) {$(m-1,k,\ell)$};
	\node[scale = 0.7] at (-150:2) {$(m-1,k-1,\ell)$};
	
	\draw[->] (-135:1.8) --(-105:1.8);
	\draw[->] (-75:1.8) -- (-45:1.8);
	\draw[->] (135:1.8) --(105:1.8);
	\draw[->] (75:1.8) -- (45:1.8);
	\draw[->] (-1.7,0.7) -- (-1.7,-0.7);
	\draw[->] (1.7,0.7) -- (1.7,-0.7);
	
	\node[scale = 0.7] at (0,0) {$(m,k-1,\ell)$};
	
	\draw[->] (-150:1.5) -- (-150:0.5);
	\draw[->] (0,1.5) -- (0,0.3);
	\draw[->] (-30:0.5) -- (-30:1.5);
	
	\draw[->,double] (150:1) + (-0.2,-0.2) --+ (0.2,0.2);
	\draw[->,double] (-90:1) + (-0.2,-0.2) --+ (0.2,0.2);
	\draw[->,double] (30:1) + (-0.2,-0.2) --+ (0.2,0.2);
	\end{tikzpicture}
	\]
	for $(h|m),(i|k),(j|\ell) \in S(\ss,\tt)$ with $h<i<j$, which follows from \cref{coherence:natural,coherence:composition}.
	Moreover, \cref{square to zero} (and hence \cref{zero}) cannot appear in this process since there is no composable pair of the form
	$\begin{tikzcd}
	\cdot
	\arrow [r, "\beta_p"] &
	\cdot
	\arrow [r, "\beta_p"] &
	\cdot
	\end{tikzcd}$
	in $\ttensor_a(\theta_1,\dots,\theta_a)(\ss,\tt)$.
	
	Now fix $f,g \in \ttensor_a(\ttheta)(\ss,\tt)$.
	We have shown that any map $f \to g$ admits a factorisation of the form (\ref{braid factorisation}) such that
	\[
	x = \beta_{p_1}\dots \beta_{p_r}
	\]
	is a normal form for some $0 \neq x \in \BB_n$.
	Since we must have $\sigma_x = f^{-1} \circ g$, \cref{permutation determines braid} implies that there is at most one morphism $f \to g$.
	This completes the proof.
\end{proof}

\section{Special outer horns}\label{section special}
The purpose of this appendix is to prove that certain \emph{special outer horn inclusions} are trivial cofibrations.

We first consider the horizontal case.
Let $\nq \in \cell$ with $n \ge 2$ and $q_1 = 0$.
\begin{definition}
	We will denote by $\hornt_h^0\nq$ and $\cellt^0\nq$ the cellular sets defined by the following pushout squares
	\[
	\begin{tikzcd}[row sep = large]
	{\cell[1;0]} \arrow [d, hook] \arrow [r] \arrow [dr, phantom, "\ulcorner", very near end] & \horn_h^0\nq \arrow [d, hook] \arrow [r, hook] \arrow [dr, phantom, "\ulcorner", very near end] & \cell\nq \arrow [d, hook]\\
	J \arrow [r] & \hornt_h^0\nq \arrow [r, hook] & \cellt^0\nq
	\end{tikzcd}
	\]
	where the composite of the upper row is $\eta_h^1 : \cell[1;0] \to \cell\nq$ and the left vertical maps pick out the $(1;0)$-cell $\{\lozenge \to \blacklozenge\}$.
\end{definition}

\begin{lemma}\label{special horizontal}
	The map $\hornt_h^0\nq \incl \cellt^0\nq$ is a trivial cofibration.
\end{lemma}
\begin{proof}
	Recall that $\Jh = \{\lozenge \cong \blacklozenge\}$ denotes the chaotic category on two objects so that $N\Jh \cong J$.
	Let $\HH$ denote the 2-category defined by the pushout
	\[
	\begin{tikzcd}[row sep = large]
	{[1;0]}
	\arrow [r]
	\arrow [d]
	\arrow [dr, phantom, "\ulcorner" very near end] &
	\nq
	\arrow [d] \\
	\Jh
	\arrow [r] &
	\HH
	\end{tikzcd}
	\]
	where the upper horizontal map is $\eta_h^1$ and the left vertical map picks out the 1-cell $\lozenge \to \blacklozenge$.
	Define a preorder $\pre$ on the set $[n] = \{0,\dots,n\}$ so that $i \pre j$ if and only if:
	\begin{itemize}
		\item $i \le j$ (with respect to the usual order); or
		\item $i = 1$ and $j = 0$.
	\end{itemize}
	Then an $(m;\pp)$-cell $[\alpha;\aalpha]$ in the nerve $N\HH$ consists of an order preserving map
	\[
	\alpha : \bigl([m],\le \bigr)\to \bigl([n],\pre\bigr)
	\]
	together with a simplicial operator $\alpha_k : [p_\ell] \to [q_k]$ for each $\ell \in [m]$ and $\alpha(\ell-1) < k \le \alpha(\ell)$.
	
	We will regard $\hornt_h^0\nq$ and $\cellt^0\nq$ as cellular subsets of $N\HH$ via the obvious monomorphisms $\hornt_h^0\nq \incl \cellt^0\nq \incl N\HH$.
	The desired result follows once we prove that both of the inclusions $\cellt^0\nq \incl N\HH$ and $\hornt_h^0\nq \incl N\HH$ are trivial cofibrations.
	These facts are proved in \cref{special horizontal cell,special horizontal horn} below.
\end{proof}

Observe that an $(m;\pp)$-cell $[\alpha;\aalpha]$ in $N\HH$ is contained in $N\HH \setminus \cellt^0\nq$ if and only if:

\begin{itemize}
	\item[(a)] there is $0 \le \ell < m$ such that $\alpha(\ell) = 1$ and $\alpha(\ell+1) = 0$; and
	\item[(b)] $\alpha(m) \ge 2$.
\end{itemize}
The only non-degenerate cells in $\cellt^0\nq \setminus \hornt_h^0\nq$ are $[\id;\iid]$ and $[\delta^0;\iid]$.

\begin{definition}
	An order-preserving map $\alpha : \bigl([m],\le\bigr) \to \bigl([n],\pre\bigr)$ is said to be \emph{dull} if $\alpha(0) \ge 2$.
	For any non-dull order-preserving map $\alpha : \bigl([m],\le\bigr) \to \bigl([n],\pre\bigr)$, we define $\ell_\alpha \defeq \max\bigl(\alpha^{-1}\bigl(\{0,1\}\bigr)\bigr)$.
\end{definition}

\begin{definition}
	An $(m;\pp)$-cell $[\alpha;\aalpha]$ in $N\HH$ is called \emph{dull} if $\alpha$ is dull.
	We say a non-degenerate, non-dull cell $[\alpha;\aalpha]$ in $N\HH$ is of:
	\begin{itemize}
		\item \emph{type 0} if $\alpha(\ell_\alpha) = 0$; and
		\item \emph{type 1} if $\alpha(\ell_\alpha) = 1$.
	\end{itemize}
\end{definition}
Note that $\hornt_h^0\nq$ (and hence $\cellt^0\nq$) contains all dull cells.

\begin{lemma}\label{special horizontal cell}
	The inclusion $\cellt^0\nq \incl N\HH$ is in $\celll(\H_h)$.
\end{lemma}
\begin{proof}
	It is easy to check (using the conditions (a) and (b) above) that the set of non-degenerate cells in $N\HH \setminus \cellt^0\nq$ can be partitioned into pairs of the form
	\[
	\bigl\{[\alpha;\aalpha], [\alpha;\aalpha] \cdot \delta_h^{\ell_\alpha;\langle!,\id\rangle}\bigr\}
	\]
	where $[\alpha;\aalpha]$ is of type 1.
	Moreover, for any $[\alpha;\aalpha]$ of type 1 in $N\HH \setminus \cellt^0\nq$, any of its hyperfaces other than the (unique) $\ell_\alpha$-th horizontal one is:
	\begin{itemize}
		\item degenerate;
		\item contained in $\cellt^0\nq$; or
		\item of type 1.
	\end{itemize}
	It follows that $N\HH$ may be obtained from $\cellt^0\nq$ by gluing those $(m;\pp)$-cells $[\alpha;\aalpha]$ of type 1 in $N\HH \setminus \cellt^0\nq$ along the horn $\horn_h^{\ell_\alpha}\mp$ in increasing order of $\dim\mp$.
	This horn is inner since (a) implies $\ell_\alpha \neq 0$ and (b) implies $\ell_\alpha \neq m$.
	This completes the proof.
\end{proof}

\begin{lemma}\label{special horizontal horn}
	The inclusion $\hornt_h^0\nq \incl N\HH$ is in $\celll(\H_h)$.
\end{lemma}
\begin{proof}
	Let $X \subset N\HH$ denote the cellular subset consisting of those cells that do not contain $[\delta^0;\iid]$ as a face.
	Then the horn inclusion can be factorised as
	\[
	\hornt_h^0\nq \incl X \incl N\HH.
	\]
	Moreover:
	\begin{itemize}
		\item the non-degenerate cells in $X \setminus \hornt_h^0\nq$ can be partitioned into pairs of the form \[
		\bigl\{[\alpha;\aalpha], [\alpha;\aalpha] \cdot \delta_h^{\ell_\alpha;\langle!,\id\rangle}\bigr\}
		\]
		where $[\alpha;\aalpha]$ is of type 1; and
		\item the non-degenerate cells in $N\HH \setminus X$ can be partitioned into pairs of the form
		\[
		\bigl\{[\alpha;\aalpha], [\alpha;\aalpha] \cdot \delta_h^{\ell_\alpha;\langle!,\id\rangle}\bigr\}
		\]
		where $[\alpha;\aalpha]$ is of type 0.
	\end{itemize}
	The rest of the proof is similar to that of \cref{special horizontal cell} and is left to the reader.
\end{proof}

Taking the ``suspension'' of the above argument yields the following vertical case.
Fix $[1;q] \in \cell$ with $q \ge 2$.
\begin{definition}\label{special vertical definition}
	We denote by $\hornt_v^{1;0}[1;q]$ and $\cellt^{1;0}[1;q]$ the cellular sets defined by the following pushout squares
	\[
	\begin{tikzcd}[row sep = large]
	{\cell[1;1]} \arrow [d, hook] \arrow [r] \arrow [dr, phantom, "\ulcorner", very near end] & \horn_v^{1;0}{[1;q]} \arrow [d, hook] \arrow [r, hook] \arrow [dr, phantom, "\ulcorner", very near end] & \cell{[1;q]} \arrow [d, hook]\\
	\cell[1;J] \arrow [r] & \hornt_v^{1;0}{[1;q]} \arrow [r, hook] & \cellt^{1;0}{[1;q]}
	\end{tikzcd}
	\]
	where the composite of the upper row is $\eta_v^1 : \cell[1;1] \to \cell[1;q]$ and the left vertical map picks out the $(1;1)$-cell $\left\{
	\begin{tikzpicture}[baseline = -3]
	\filldraw
	(0,0) circle [radius = 1pt]
	(1,0) circle [radius = 1pt];
	\draw[->] (0.1,0.1) .. controls (0.4,0.4) and (0.6,0.4) .. (0.9,0.1);
	\draw[->] (0.1,-0.1) .. controls (0.4,-0.4) and (0.6,-0.4) .. (0.9,-0.1);
	\draw[->, double] (0.5,0.25) -- (0.5,-0.25);
	\node[scale = 0.7] at (0.5, 0.5) {$\lozenge$};
	\node[scale = 0.7] at (0.5, -0.5) {$\blacklozenge$};
	\end{tikzpicture}\right\}$.
\end{definition}

\begin{lemma}\label{special vertical}
	The map $\hornt_v^{1;0}[1;q] \incl \cellt^{1;0}[1;q]$ is a trivial cofibration.
\end{lemma}

\section*{Acknowledgements}
This paper is based on the author's PhD thesis.
He would like to thank his principle supervisor Dominic Verity for constant and helpful feedback throughout the project.
He also gratefully acknowledges the support of an International Macquarie University Research Training Program Scholarship (Allocation Number: 2017127).

\bibliographystyle{alpha}
\bibliography{ref}

\end{document}